\documentclass[12pt]{amsart}
\usepackage{amsthm,amsfonts,amssymb,amscd,mathrsfs}
\usepackage{bezier,longtable,amstext}
\usepackage{amssymb}

\usepackage[all]{xy}

\newcommand\codim{\text{codim}}
\newcommand\coker{\operatorname{coker}\nolimits}

\newcommand\Ext{\operatorname{Ext}\nolimits}

\newcommand{\Hilb}{\operatorname{Hilb}\nolimits}

\newcommand\Hom{\operatorname{Hom}\nolimits}

\newcommand\im{\operatorname{Im}\nolimits}
\newcommand\Ker{\operatorname{Ker}\nolimits}

\newcommand\rk{\text{rk}}

\newcommand\Span{\text{Span}}

\newcommand\Supp{{\operatorname{Supp}\nolimits}}

\newtheorem{theorem}{Theorem}[section]
\newtheorem{proposition}[theorem]{Proposition}

\newtheorem{lemma}[theorem]{Lemma}
\newtheorem{corollary}[theorem]{Corollary}
\newtheorem{sub}[subsection]{}

\theoremstyle{definition}
\newtheorem{definition}[theorem]{Definition}

\newtheorem{proposition-definition}[theorem]{Proposition-Definition}

\newtheorem{remark}[theorem]{Remark}

\nonstopmode

\begin{document}


\vspace{1cm}

\title{Moduli of mathematical instanton vector bundles with odd $c_2$ on projective space}

\author[Tikhomirov]{\;Alexander S.~Tikhomirov}

\address{
Department of Mathematics\\
State Pedagogical University\\
Respublikanskaya Str. 108
\newline 150 000 Yaroslavl, Russia}
\email{astikhomirov@mail.ru}

\maketitle

\thispagestyle{empty}

\begin{abstract}

We study the problem of irreducibility of the moduli space $I_n$ of rank-2
mathematical instanton vector bundles with second Chern class $n\ge1$ on
the projective space $\mathbb{P}^3$. The irreducibility
of $I_n$ was known for small values of n: for $n=1$ it was proved by Barth (1977), for $n=2$
by Hartshorne (1978), for $n=3$ by Ellingsrud and Str{\o}mme (1981), for $n=4$ by Barth (1981),
for $n=5$ by Coanda, Tikhomirov and Trautmann (2003). In this paper we prove the
irreducibility of $I_n$ for an arbitrary odd $n\ge1$.

Bibliography: 22 items.

\end{abstract}

\section{Introduction}\label{sec0}

By a {\it mathematical $n$-instanton vector bundle} (shortly, a {\it
$n$-instanton}) on 3-dimensional projective space $\mathbb{P}^3$ we understand a rank-2 algebraic
vector bundle $E$ on $\mathbb{P}^3$ with Chern classes
\begin{equation}\label{Chern classes}
c_1(E)=0,\ \ \ c_2(E)=n,\ \ \ n\ge1,
\end{equation}
satisfying the vanishing conditions
\begin{equation}\label{vanishing condns}
h^0(E)=h^1(E(-2))=0.
\end{equation}
Denote by $I_n$ the set of isomorphism classes of $n$-instantons. This space is nonempty for any
$n\ge1$ - see, e.g., \cite{BT}, \cite{NT}. The condition $h^0(E)=0$ for
a $n$-instanton $E$ implies that $E$ is stable in the sense of Gieseker-Maruyama.
Hence $I_n$ is a subset of the moduli scheme $M_{\mathbb{P}^3}(2;0,2,0)$ of semistable rank-2
torsion-free sheaves on $\mathbb{P}^3$ with Chern classes $c_1=0,\ c_2=n,\ c_3=0$.
The condition $h^1(E(-2))=0$ for $[E]\in I_n$ (called the {\it instanton condition})
implies by semicontinuity that $I_n$ is a Zariski open subset
of $M_{\mathbb{P}^3}(2;0,2,0)$, i.e. $I_n$ is a quasiprojective scheme. It is called the
{\it moduli scheme of mathematical $n$-instantons}.

In this paper we study the problem of the irreducibility of the scheme $I_n$. This problem has
an affirmative solution for small values of $n$, up to $n=5$.
Namely, the cases $n=1,3,3,4$ and 5 were settled in papers \cite{B1}, \cite{H}, \cite{ES},
\cite{B3} and \cite{CTT}, respectively.
The aim of this paper is to prove the following result.
\begin{theorem}\label{Irreducibility}
For each $n=2m+1,\ \ m\ge0$, the moduli scheme $I_n$ of mathematical $n$-instantons is an integral scheme of dimension $8n-3$.
\end{theorem}

A guide to the paper is as follows. In section \ref{general} we recall a well-known relation
between mathematical $n$-instantons and nets of quadrics in a fixed $n$-dimensional vector
space $H_n$ over $\mathbf{k}$. The nets of quadrics are considered as vectors of the space
$\mathbf{S}_n=S^2H_n^\vee\otimes \wedge^2V^\vee$,
where
$V=H^0(\mathcal{O}_{\mathbb{P}^3}(1))^\vee$,
and those nets which correspond to $n$-instantons (we call them $n$-instanton nets) satisfy the
so-called Barth's conditions - see definition (\ref{space of nets}). These nets constitute a
locally closed subset $MI_n\subset$ of $\mathbf{S}_n$ which has a structure of a $GL(n)/\{\pm1\}$-bundle
over $I_n$. Thus the irreducibility of the moduli space $I_n$ of $n$-instantons reduces to the
irreducibility of the space $MI_n$ of $n$-instanton nets of quadrics.

Section \ref{decomp} is a study of some linear algebra related to a direct sum decomposition
$\xi:H_{m+1}\oplus H_m\overset{\sim}\to H_{2m+1}$
giving the above embedding $H_{m+1}\hookrightarrow H_{2m+1}$. Using one result
of section \ref{Appendix} we obtain here the relation (\ref{xi3(A)}) which is a key instrument for
our further considerations.
Also, the decomposition $\xi$ enables us to relate $(2m+1)$-instantons $E$ to rank-$(2m+2)$
symplectic vector bundles $E_{2m+2}$ on $\mathbb{P}^3$ satisfying the vanishing conditions
$h^0(E_{2m+2})=h^2(E_{2m+2}(-2))=0$.

In section \ref{Space Xm} we introduce a new set $X_m$ as a locally closed subset of the vector space
$\mathbf{S}_{m+1}\oplus\mathbf{\Sigma}_{m+1}$,
where
$\mathbf{\Sigma}_{m+1}=\Hom(H_m,H_{m+1}^\vee\otimes\wedge^2V^\vee)$,
defined by linear algebraic data somewhat similar to Barth's conditions.
We prove that $X_m$, is isomorphic to a certain dense open subset
$MI_{2m+1}(\xi)$ of $MI_{2m+1}$ determined by the choice of the direct sum decomposition
$\xi$ above, where both $X_m$ and $MI_{2m+1}(\xi)$ are understood as reduced schemes. This reduces the problem of the
irreducibility of $I_{2m+1}$ to that of $X_m$.

The last ingredient in the proof of Theorem \ref{Irreducibility} is a scheme $Z_m$ introduced
in section \ref{Scheme Zm} as a locally closed subscheme of the affine space
$\mathbf{S}_m^\vee\times\Hom(H_m,H_m^\vee\otimes\wedge^2 V^\vee)$ defined by explicit equations
(see (\ref{tildeZm})). In section \ref{Scheme Zm} we reduce the proof of Theorem \ref{Irreducibility}
to the fact that $Z_m$ is an integral locally complete intersection subscheme of the above mentioned
affine space. This and other properties of $Z_m$ are formulated in Theorem \ref{Irreducibility of Zm}.
The rest of the paper is devoted to the proof of Theorem  \ref{Irreducibility of Zm}.

In section \ref{Study Zm} we start the proof of this Theorem by induction on $m$ and prove a part of the
induction step - see Proposition \ref{part of inductn step}. The proof of it contains explicit computations
in linear algebra. These computations seem to be somewhat cumbersome, and Remark \ref{computations} at the end of this
section gives an explanation why these computations could not be essentially simplified.

Proposition \ref{part of inductn step} enables us then in section \ref{Zm and tH} to
relate $Z_m$ to the so-called t'Hooft instantons. As a result, in section \ref{gen pos} we finish the induction
step in the proof of Theorem  \ref{Irreducibility of Zm}.

In Appendix (section \ref{Appendix}) we prove two results of general position for nets of quadrics, which are used in
the text.

\vspace{0.3cm}
\textbf{Acknowledgement.} The author acknowledges the support and hospitality of the Max Planck
Institute for Mathematics in Bonn where this paper was started during the author's stay there in
Winter 2008.

\vspace{0.5cm}
\section{Notation and conventions}\label{notation}
\vspace{0.5cm}

Our notations are mostly standard. The base field $\mathbf{k}$ is assumed to be
algebraically closed of characteristic 0. We identify vector bundles with locally free
sheaves. If $\mathcal{F}$ is a sheaf of $\mathcal{O}_X$-modules on an algebraic
variety or scheme $X$, then $n\mathcal{F}$ denotes a direct sum of $n$ copies of the
sheaf $\mathcal{F}$, $H^i(\mathcal{F})$ denotes the $i^{th}$ cohomology group of
$\mathcal{F}$, $h^i(\mathcal{F}):=\dim H^i(\mathcal{F})$, and $\mathcal{F}^\vee$
denotes the dual to $\mathcal{F}$ sheaf, i.e. the sheaf $\mathcal{F}^\vee:=
\mathcal{H}om_{\mathcal{O}_X} (\mathcal{F},\mathcal{O}_X)$. If $Z$ is a subscheme of
$X$, by $\mathcal{I}_{Z,X}$ we denote the ideal sheaf corresponding to a subscheme $Z$. If
$X=\mathbb{P}^r$ and $t$ is an integer, then by $\mathcal{F}(t)$ we denote the sheaf
$\mathcal{F}\otimes\mathcal{O}_{\mathbb{P}^r}(t)$. $[\mathcal{F}]$ will denote the
isomorphism class of a sheaf $\mathcal{F}$. For any morphism of
$\mathcal{O}_X$-sheaves $f:\mathcal{F}\to\mathcal{F}'$ and any $\mathbf{k}$-vector space $U$
(respectively, for any homomorphism $f:U\to U'$ of $\mathbf{k}$-vector spaces) we will denote,
for short, by the same letter $f$ the induced morphism of sheaves
$id\otimes f:U\otimes\mathcal{F}\to U\otimes\mathcal{F}'$
(respectively, the induced morphism
$f\otimes id:U\otimes\mathcal{F}\to U'\otimes\mathcal{F}$).

Everywhere in the paper $V$ will denote a fixed vector space of dimension 4 over $\mathbf{k}$ and
we set $\mathbb{P}^3:=P(V)$. Also everywhere below we will reserve the letters $u$ and $v$ for
denoting the two morphisms in the Euler exact sequence
$0\to\mathcal{O}_{\mathbb{P}^3}(-1)\overset{u}\to V^\vee\otimes\mathcal{O}_{\mathbb{P}^3}
\overset{v}\to T_{\mathbb{P}^3}(-1)\to0$.
For any $\mathbf{k}$-vector spaces $U$ and $W$ and any vector
$\phi\in{\rm Hom}(U,W\otimes\wedge^2V^\vee)\subset{\rm Hom}(U\otimes V,W\otimes V^\vee)$
understood as a homomorphism $\phi:U\otimes V\to W\otimes V^\vee$ or, equivalently,
as a homomorphism
${}^\sharp\phi:U\to W\otimes\wedge^2V^\vee$, we will denote by $\widetilde{\phi}$ the composition
$U\otimes\mathcal{O}_{\mathbb{P}^3}\overset{{}^\sharp\phi}\to
W\otimes\wedge^2V^\vee\otimes\mathcal{O}_{\mathbb{P}^3}\overset{\epsilon}\to
W\otimes\Omega_{\mathbb{P}^3}(2)$, where $\epsilon$ is the induced morphism in the exact triple
$0\to\wedge^2\Omega_{\mathbb{P}^3}(2)\overset{\wedge^2v^\vee}\to\wedge^2V^\vee\otimes
\mathcal{O}_{\mathbb{P}^3}\overset{\epsilon}\to\Omega_{\mathbb{P}^3}(2)\to0$ obtained by passing
to the second wedge power in the dual Euler exact sequence.
Also, shortening the notation, we will omit sometimes the subscript ${\mathbb{P}^3}$ in the
notation of sheaves on $\mathbb{P}^3$, e.g., write $\mathcal{O},\ \Omega$ etc., instead of
$\mathcal{O}_{\mathbb{P}^3},\ \Omega_{\mathbb{P}^3}$ etc., respectively.

Next, as above, for any integer $n\ge1$ by $H_n$ we understand a fixed $n$-dimensional vector space over $\mathbf{k}$.
(E. g., one can take $\mathbf{k}^n$ for $H_n$.)

Everywhere in the paper for $m\ge1$ we denote by $\mathbf{S}_m$ the vector space
$S^2H_m^\vee\otimes \wedge^2V^\vee$, respectively, by $\mathbf{\Sigma}_{m+1}$ the vector space
$\Hom(H_m,H_{m+1}^\vee\otimes\wedge^2V^\vee)$. For a given $\mathbf{k}$-vector space $U$ (respectively,
a direct sum $U\oplus U'$ of two $\mathbf{k}$-vector spaces) we will, abusing notations, denote by the same letter $U$
(respectively, by $U\oplus U'$) the corresponding affine space $\mathbf{V}(U^\vee)={\rm Spec}(Sym^*U^\vee)$ (respectively,
the direct product of affine spaces $\mathbf{V}(U^\vee)\times\mathbf{V}(U'^\vee)$).

All the schemes considered in the paper are Noetherian. By an irreducible scheme we understand a scheme whose underlying
topological space is irreducible. By an integral scheme we
understand an irreducible reduced scheme. Also, by the dimension of a given scheme we understand below the
maximum of dimensions of its irreducible components. By a general point of an irreducible (but not necessarily reduced)
scheme $\mathcal{X}$ we mean any closed point belonging to some dense open subset of $\mathcal{X}$. An irreducible
scheme is called generically reduced if it is reduced at a general point.

\vspace{0.5cm}
\section{Some generalities on instantons. Set $MI_n$}\label{general}
\vspace{0.5cm}

In this Section we recall some well known facts about mathematical instanton bundles -
see, e.g., \cite{CTT}.

For a given $n$-instanton $E$, the conditions (\ref{Chern classes}),
(\ref{vanishing condns}), Riemann-Roch and Serre duality imply
\begin{equation}\label{dimension1}
h^1(E(-1))=h^2(E(-3))=n,\ \ \ h^1(E\otimes\Omega^1_{\mathbb{P}^3})=
h^2(E\otimes\Omega^2_{\mathbb{P}^3})=2n+2,\ \ \
\end{equation}
$$
h^1(E)=h^2(E(-4))=2n-2.
$$
\begin{equation}\label{dimension2}
h^i(E)=h^i(E(-1))=h^{3-i}(E(-3))=h^{3-i}(E(-4))=0,\ \ i\ne1,\ \ \ h^i(E(-2))=0,\ \ i\ge0.
\end{equation}

Furthermore, the condition $c_1(E)=0$ yields an isomorphism
$\wedge^2E\overset{\simeq}\to\mathcal{O}_{\mathbb{P}^3}$, hence a symplectic isomorphism
$j:E\overset{\simeq}\to E^\vee$ defined uniquely up to a scalar.
Consider a triple $(E,f,j)$ where $E$ is an $n$-instanton, $f$ is an isomorphism
$H_n\overset{\simeq}\to H^2(E(-3))$
and $j:E\overset{\simeq}\to E^\vee$ is a symplectic structure on $E$.
Note that, since $E$ as a stable rank-2 bundle, it is a simple bundle,
i. e. any automorphism $\varphi$ of $E$
has the form $\varphi=\lambda{\rm id}$ for some $\lambda\in\mathbf{k}^*$.
Imposing the condition that $\varphi$ is compatible with the symplectic structure $j$,
i. e. $\varphi^\vee\circ j\circ\varphi=j$, we obtain $\lambda=\pm1$. This leads to the
following definition of equivalence of triples $(E,f,j)$. We call two such triples
$(E,f,j)$ and $(E'f',j')$ equivalent if there is an isomorphism
$g:E\overset{\simeq}\to E'$ such that $g_*\circ f=\lambda f'$ with $\lambda\in\{1,-1\}$ and
$j=g^\vee\circ j'\circ g$,
where
$g_*:H^2(E(-3))\overset{\simeq}\to H^2(E'(-3))$
is the induced isomorphism. We denote by $[E,f,j]$ the equivalence class of a triple $(E,f,j)$.
From this definition one easily deduces that the set $F_{[E]}$ of all equivalence classes
$[E,f,j]$ with given $[E]$ is a homogeneous space of the group $GL(H_n)/\{\pm{\rm id}\}$.

Each class $[E,f,j]$ defines a point
\begin{equation}\label{hypernet}
A=A([E,f,j])\in S^2H_n^\vee\otimes \wedge^2V^\vee
\end{equation}
in the following way. Consider the exact sequences
\begin{equation}\label{Koszul triples}
0\to\Omega^1_{\mathbb{P}^3}\overset{i_1}\to V^\vee\otimes\mathcal{O}_{\mathbb{P}^3}(-1)\to
\mathcal{O}_{\mathbb{P}^3}\to0,
\end{equation}
$$
0\to\Omega^2_{\mathbb{P}^3}\to\wedge^2V^\vee\otimes\mathcal{O}_{\mathbb{P}^3}(-2)\to
\Omega^1_{\mathbb{P}^3}\to0,
0\to\wedge^4V^\vee\otimes\mathcal{O}_{\mathbb{P}^3}(-4)\to
\wedge^3V^\vee\otimes\mathcal{O}_{\mathbb{P}^3}(-3)\overset{i_2}\to\Omega^2_{\mathbb{P}^3}\to0,
$$
induced by the Koszul complex of
$V^\vee\otimes\mathcal{O}_{\mathbb{P}^3}(-1)\overset{ev}
\twoheadrightarrow\mathcal{O}_{\mathbb{P}^3}$.
Twisting these sequences by $E$ and passing to cohomology in view of (\ref{vanishing condns})-(\ref{dimension2})
gives the equalities
$0=h^0(E\otimes\Omega_{\mathbb{P}^3})=h^3(E\otimes\Omega^2_{\mathbb{P}^3})=h^2(E\otimes\Omega_{\mathbb{P}^3})$
and the diagram with exact rows
\begin{equation}\label{A'}
\xymatrix{0\ar[r] &
H^2(E(-4))\otimes\wedge^4V^\vee\ar[r] &
H^2(E(-3))\otimes\wedge^3V^\vee\ar[r]^{\ \ \ \ i_2}\ar[d]^{A'} &
H^2(E\otimes\Omega^2_{\mathbb{P}^3})\ar[r]& 0 \\
0 & H^1(E))\ar[l] & H^1(E(-1))\otimes V^\vee\ar[l] &
H^1(E\otimes\Omega_{\mathbb{P}^3})\ar[l]_{\ \ \ \ i_1}\ar[u]^{\cong}_{\partial}& 0,\ar[l]}
\end{equation}
where $A':=i_1\circ\partial^{-1}\circ i_2$. The Euler exact sequence (\ref{Koszul triples})
yields a canonical isomorphism
$\omega_{\mathbb{P}^3}\overset{\simeq}\to\wedge^4V^\vee\otimes\mathcal{O}_{\mathbb{P}^3}(-4)$,
and fixing an isomorphism
$\tau:\mathbf{k}\overset{\simeq}\to\wedge^4V^\vee$
induces isomorphisms
$\tilde{\tau}:V\overset{\simeq}\to\wedge^3V^\vee$
and
$\hat{\tau}:\omega_{\mathbb{P}^3}\overset{\simeq}\to\mathcal{O}_{\mathbb{P}^3}(-4)$.
Now the point $A$ in (\ref{hypernet}) is defined as the composition
\begin{equation}\label{An}
A:H_n\otimes V\overset{\tilde{\tau}}{\overset{\simeq}\to}
H_n\otimes\wedge^3V^\vee\overset{f}{\overset{\simeq}\to}
H^2(E(-3))\otimes\wedge^3V^\vee\overset{A'}\to H^1(E(-1))\otimes V^\vee
\overset{j}{\overset{\simeq}\to}
\end{equation}
$$
\overset{j}{\overset{\simeq}\to} H^1(E^\vee(-1))\otimes V^\vee
\overset{SD}{\overset{\simeq}\to} H^2(E(1)\otimes\omega_{\mathbb{P}^3})^\vee\otimes V^\vee
\overset{\hat{\tau}}{\overset{\simeq}\to} H^2(E(-3))^\vee\otimes V^\vee
\overset{f^\vee}{\overset{\simeq}\to}H_n^\vee\otimes V^\vee,
$$
where $SD$ is the Serre duality isomorphism. One checks that $A$ is a skew symmetric map
depending only on the class $[E,f,j]$ and not depending on the choice of $\tau$, and that this
point $A\in\wedge^2(H_n^\vee\otimes V^\vee)$ lies in the direct summand
$\mathbf{S}_n=S^2H_n^\vee\otimes \wedge^2V^\vee$
of the canonical decomposition
\begin{equation}\label{can decomp}
\wedge^2(H_n^\vee\otimes V^\vee)=S^2H_n^\vee\otimes \wedge^2V^\vee\oplus
\wedge^2H_n^\vee\otimes S^2V^\vee.
\end{equation}
Here $\mathbf{S}_n$ is the space of nets of quadrics in $H_n$. Following \cite{B3},
\cite{Tju2} and \cite{Tju1} we call $A$ the $n$-{\it instanton net of quadrics} corresponding to
the data $[E,f,j]$.

Denote $W_A:=H_n\otimes V/\ker A$. Using the above chain of isomorphisms we
can rewrite the diagram (\ref{A'}) as
\begin{equation}\label{qA}
\xymatrix{0\ar[r] & \ker A\ar[r] &
H_n\otimes V\ar[r]^{\ \ c_A}\ar[d]^{A} & W_{A}\ar[r]\ar[d]_{\cong}^{q_{A}} & 0 \\
0 & \ker A^\vee\ar[l] & H_n^\vee\otimes V^\vee\ar[l] &
W_{A}^\vee\ar[l]_{\ \ \ \ \ c_A^\vee}& 0.\ar[l]}
\end{equation}
Here in view of (\ref{dimension1}) $\dim W_{A}=2n+2$ and
$q_{A}:W_{A}\overset{\simeq}\to W_{A}^\vee$
is the induced skew-symmetric isomorphism. An important property of $A=A([E,f,j])$ is that
the induced morphism of sheaves
\begin{equation}\label{an^vee}
a_A^\vee:{W}^\vee_{A}\otimes\mathcal{O}_{\mathbb{P}^3}
\overset{c_A^\vee}\to
H_n^\vee\otimes V^\vee\otimes\mathcal{O}_{\mathbb{P}^3}
\overset{ev}\to H_n^\vee\otimes\mathcal{O}_{\mathbb{P}^3}(1)
\end{equation}
is an epimorphism such that the composition
$H_n\otimes\mathcal{O}_{\mathbb{P}^3}(-1)\overset{a_A}\to
W_{A}\otimes\mathcal{O}_{\mathbb{P}^3}\overset{q_A}\to
W^\vee_{A}\otimes\mathcal{O}_{\mathbb{P}^3}
\overset{a^\vee_{A}}\to H_n^\vee\otimes\mathcal{O}_{\mathbb{P}^3}(1)$
is zero, and $E=\ker(a_A^\vee\circ q_{A})/\im a_A$. Thus $A$ defines a monad
\begin{equation}\label{Monad An}
\mathcal{M}_{A}:\ \ 0\to H_n\otimes\mathcal{O}_{\mathbb{P}^3}(-1)\overset{a_A}\to
W_{A}\otimes\mathcal{O}_{\mathbb{P}^3}\overset{a_A^\vee\circ q_{A}}
\to H_n^\vee\otimes\mathcal{O}_{\mathbb{P}^3}(1)\to0
\end{equation}
with the cohomology sheaf $E$,
\begin{equation}\label{coho sheaf}
E=E(A):=\ker(a_A^\vee\circ q_{A})/\im a_A.
\end{equation}
Note that passing to cohomology in the monad $\mathcal{M}_{A}$ twisted by
$\mathcal{O}_{\mathbb{P}^3}(-3)$ and using (\ref{coho sheaf}) yields the isomorphism
$f:H_n\overset{\simeq}\to H^2(E(-3))$.
Furthermore, the simplecticity of the form $q_{A}$ in the monad $\mathcal{M}_{A}$ implies
that there is a canonical isomorphism of $\mathcal{M}_{A}$ with its dual monad, and this isomorphism induces the
symplectic isomorphism $j:E\overset{\simeq}\to E^\vee$. Thus, the data $[E,f,j]$ are recovered
from the net $A$. This leads to the following description of the moduli space $I_n$. Consider
the {\it set of $n$-instanton nets of quadrics}
\begin{equation}\label{space of nets}
MI_n:=\left\{A\in \mathbf{S}_n\ \left|\
\begin{matrix}
(i)\ \rk(A:H_n\otimes V\to H_n^\vee\otimes V^\vee)=2n+2,
\ \ \ \ \ \ \ \ \ \ \ \ \ \ \ \ \cr
(ii)\ {\rm the\ morphism}\ a_A^\vee:W^\vee_{A}\otimes\mathcal{O}_{\mathbb{P}^3}\to
H_n^\vee\otimes\mathcal{O}_{\mathbb{P}^3}(1)\ \ \ \ \ \ \cr
{\rm defined\ by}\ A\ {\rm in}\ (\ref{an^vee})\ {\rm is\ surjective},
\ \ \ \ \ \ \ \ \ \ \ \ \ \ \ \ \ \ \ \cr
\ (iii)\ h^0(E_2(A))=0,\ {\rm where}\  E_2(A):=
\ker(a^\vee_A\circ q_{A})/\im a_A\ \cr
\ {\rm and}\ q_{A}:W_{A}\overset{\simeq}\to W_{A}^\vee\ {\rm is\ a\ symplectic\
isomorphism} \cr
{\rm defined\ by}\ A\ {\rm in}\ (\ref{qA})\ \ \ \ \ \ \ \ \ \ \ \ \ \ \ \ \ \ \ \ \
\ \ \ \ \ \ \ \ \ \ \ \ \ \ \ \ \ \cr
\end{matrix}\right.
\right\}
\end{equation}
The conditions (i)-(iii) here are called {\it Barth's coditions}. These conditions show that
$MI_n$ is naturally endowed with a structure of a locally closed subscheme of the vector space
$\mathbf{S}_n$. Moreover, the above description shows that there is a morphism
$\pi_n: MI_n\to I_n:\ A\mapsto[E(A)]$, and it is known that this morphism is a principal
$GL(H_n)/\{\pm{\rm id}\}$-bundle in the \'etale topology - cf. \cite{CTT}. Here by construction
the fibre $\pi_n^{-1}([E])$ over an arbitrary point $[E]\in I_n$ coincides with the homogeneous
space $F_{[E]}$ of the group $GL(H_n)/\{\pm{\rm id}\}$ described above. Hence the
irreducibility of $(I_n)_{red}$ is equivalent to the irreducibility of the scheme $(MI_n)_{red}$.

The definition (\ref{space of nets}) yields the following.
\begin{theorem}\label{rank condns}
For each $n\ge1$, the space of $n$-instanton nets of quadrics $MI_n$ is a locally closed subscheme
of the vector space $\mathbf{S}_n$ given locally at any point
$A\in MI_n$ by
\begin{equation}\label{eqns(i)}
\binom{2n-2}{2}=2n^2-5n+3
\end{equation}
equations obtained as the rank condition (i) in (\ref{space of nets}).
\end{theorem}

Note that from (\ref{eqns(i)}) it follows that
\begin{equation}\label{dim MIn ge...}
\dim_{[A]}MI_n\ge\dim \mathbf{S}_n-(2n^2-5n+3)=n^2+8n-3
\end{equation}
at any point $A\in MI_n$. On the other hand, by deformation theory for any $n$-instanton $E$ we
have $\dim_{[E]}I_n\ge8n-3$. This agrees with (\ref{dim MIn ge...}), since
$MI_n\to I_n$ is a principal $GL(H_n)/\{\pm{\rm id}\}$-bundle in the \'etale topology.

Let $\mathcal{S}_n=\{[E]\in I_n|$ there exists a line $l\in\mathbb{P}^3$ of maximal jump for
$E$, i.e. such a line $l$ that $h^0(E(-n)|_l)\ne0\}$. It is known \cite{S} that $\mathcal{S}_n$
is a closed subset of $I_n$ of dimension $6n+2$, and $I_n$ is smooth along $\mathcal{S}_n$.
Thus, since $\dim_{[E]}I_n\ge8n-3$ at any $[E]\in I_n$, it follows that
\begin{equation}\label{def of I'n}
I'_n:=I_n\smallsetminus\mathcal{S}_n
\end{equation}
is an open subset of $I_n$ and $(I'_n)_{red}$ is dense open in $(I_n)_{red}$; respectively,
\begin{equation}\label{def of MI'n}
MI'_n:=\pi_n^{-1}(I'_n)
\end{equation}
is an open subset of $MI_n$ and we have a dense open embedding
\begin{equation}\label{MI'n dense}
\xymatrix{(MI'_n)_{red}\ar@{^{(}->}[rr]^-{\rm dense\ open} & & (MI_n)_{red}}.
\end{equation}
For technical reasons we will below restrict ourselves to $MI'_n$ instead of $MI_n$.
\begin{remark}\label{smooth pts}
There exist smooth points of $I_n$ - see, e.g., \cite{NT}. Hence, there exist smooth points in $MI_n$.
\end{remark}

\vspace{0.5cm}


\section{Decomposition $H_{2m+1}\simeq H_{m+1}\oplus H_m$ and related constructions}\label{decomp}

\vspace{0.5cm}

\begin{sub}{\bf One result of general position for $(2m+1)$-instanton nets.}
\label{gen pos1}
\rm

Fix a positive integer $m\ge3$ and, for a given $(2m+1)$-instanton vector bundle
$[E]\in I'_{2m+1}$, fix an isomorphism $f:H_{2m+1}\overset{\simeq}\to H^2(E(-3))$
and set
\begin{equation}\label{H4k+4}
H_{4m}:=H^2(E(-4)),\ \ \ W_{4m+4}:=H^1(E\otimes\Omega_{\mathbb{P}^3})^\vee.
\end{equation}
(Here we keep in mind the equalities (\ref{dimension1}) for $n=2m+1$.)
In this notation, the lower exact triple in (\ref{A'}) can be rewritten as:
\begin{equation}\label{mult2}
0\to W^\vee_{4m+4}\to H^\vee_{2m+1}\otimes V^\vee\overset{mult}\to H^\vee_{4m}\to0
\end{equation}

We formulate now the following result of general position for $(2m+1)$-instanton nets
of quadrics which will be important for further study.

\begin{theorem}\label{generic Vm}
Let $m\ge3$ and let $E$ be a $(2m+1)$-instanton, $[E]\in I'_{2m+1}$, supplied with an
isomorphism $f:H_{2m+1}\overset{\simeq}\to H^2(E(-3))$
and set $W_{4m+4}=H^1(E\otimes\Omega_{\mathbb{P}^3})^\vee$, so that there is the injection
$W^\vee_{4m+4} \hookrightarrow H^\vee_{2m+1}\otimes V^\vee$ defined in (\ref{mult2}).
Then for a generic $m$-dimensional subspace $V_m$ of $H^\vee_{2m+1}$ one has
$$
W^\vee_{4m+4}\cap V_m\otimes V^\vee=\{0\}.
$$
\end{theorem}

The proof of this Theorem has rather technical character, and we leave it to the end of the paper - see Appendix
(section \ref{Appendix}).

\end{sub}

\begin{sub}{\bf Decomposition $H_{2m+1}\simeq H_{m+1}\oplus H_m$.}
\label{decomp xi}
\rm

Fix an isomorphism
\begin{equation}\label{xi}
\xi:H_{m+1}\oplus H_m\overset{\simeq}\to H_{2m+1}
\end{equation}
and let
\begin{equation}\label{injns}
H_{m+1}\overset{i_{m+1}}\hookrightarrow H_{m+1}
\oplus H_m\overset{i_m}\hookleftarrow H_m
\end{equation}
be the injections of direct summands.
For a given $(2m+1)$-instanton vector bundle $E$, $[E]\in I'_{2m+1}$, fix an isomorphism
$f:H_{2m+1}\overset{\simeq}\to H^2(E(-3))$ and a symplectic structure
$j:E\overset{\simeq}\to E^\vee$. The data $[E,f,j]$ define a net of quadrics
$A\in MI'_{2m+1}$ (see section \ref{general}), and the exact triple (\ref{mult2}) is
naturally identified with the dual to the triple
$0\to\ker A\to H_{2m+1}\otimes V\to W_{A}\to0$ and fits in diagram
(\ref{qA}) for $n=2m+1$
\begin{equation}\label{qA 2m+1}
\xymatrix{0\ar[r] & \ker A\ar[r] &
H_{2m+1}\otimes V\ar[r]^{\ \ c_A}\ar[d]^{A} & W_{A}
\ar[r]\ar[d]_{\cong}^{q_{A}} & 0 \\
0 & \ker A^\vee\ar[l] & H_{2m+1}^\vee\otimes V^\vee\ar[l] &
W_{A}^\vee\ar[l]_{\ \ \ \ \ c_A^\vee}& 0.\ar[l]}
\end{equation}
Consider the composition
\begin{equation}\label{j xi}
j_{\xi,A}:H_{m+1}\otimes V\overset{i_{m+1}}\hookrightarrow
H_{m+1}\otimes V\oplus H_m\otimes V
\overset{\overset{\xi}\simeq}\to H_{2m+1}\otimes V\overset{c_A}\to W_A.
\end{equation}
Under these notations Theorem \ref{generic Vm} can be reformulated in the following way:

\vspace{0.2cm}
(*) {\it Assume $m\ge3$ and let $A$ be an arbitrary $(2m+1)$-net from $MI'_{2m+1}$. Then for a
generic isomorphism
$\xi:H_{2m+1}\overset{\simeq}\to H_{m+1}\oplus H_m$
one has
\begin{equation}\label{ker A cap...=0}
\ker A\cap(\xi\circ i_{m+1})(H_{m+1}\otimes V)=\{0\}.
\end{equation}
Equivalently, $j_{\xi,A}:\ H_{m+1}\otimes V\to W_{A}$ is an isomorphism.}
\vspace{0.2cm}

Consider the direct sum decomposition corresponding to the isomorphism (\ref{xi})
\begin{equation}\label{direct sum}
\widetilde{\xi}: \mathbf{S}_{m+1}\ \ \oplus\ \ \mathbf{\Sigma}_{m+1}\ \ \oplus\ \ \mathbf{S}_m\overset{\sim}\to\mathbf{S}_{2m+1}
\end{equation}
and let
\begin{equation}\label{xi 1,2,3}
\mathbf{S}_{2m+1}\twoheadrightarrow\mathbf{S}_{m+1}:A\mapsto A_1(\xi),\ \ \ \
\mathbf{S}_{2m+1}\twoheadrightarrow\mathbf{\Sigma}_{m+1}:A\mapsto A_2(\xi),\ \ \ \
\mathbf{S}_{2m+1}\twoheadrightarrow\mathbf{S}_m:A\mapsto A_3(\xi)
\end{equation}
be the projections onto direct summands. By definition, $A_1(\xi)$ considered as a skew-symmetric
homomorphism
$H_{m+1}\otimes V\to H_{m+1}^\vee\otimes V^\vee$
coincides with the composition
\begin{equation}\label{xi 1}
A_1(\xi):\ H_{m+1}\otimes V\overset{j_{\xi,A}}\to W_{A}
\overset{q_{A}}{\underset{\simeq}\to}W^\vee_{A}
\overset{j^\vee_{\xi,A}}\to H_{m+1}^\vee\otimes V^\vee.
\end{equation}
This means that assertion (*) can be reformulated as:

\vspace{0.2cm}
(**) {\it Assume $m\ge3$ and let $A$ be an arbitrary $(2m+1)$-net from $MI'_{2m+1}$.
Then for a generic isomorphism $\xi$ in (\ref{xi}) the skew-symmetric homomorphism
$A_1(\xi):\ H_{m+1}\otimes V\to H_{m+1}^\vee\otimes V^\vee$
is invertible.}
\vspace{0.2cm}

Now, using the notation (\ref{xi 1,2,3}), we can represent the net $A\in\mathbf{S}_{2m+1}$ considered as a homomorphism
$A:H_{m+1}\otimes V\oplus H_m\otimes V\to H_{m+1}^\vee\otimes V^\vee\oplus H_m^\vee\otimes V^\vee$
by the $(8m+4)\times(8m+4)$-matrix of homomorphisms
$$
A=\left(\begin{array}{cc}
A_1(\xi) & A_2(\xi) \\
-A_2(\xi)^\vee & A_3(\xi)
\end{array}\right).
$$
This matrix is of rank $4m+4$ according to Barth's condition (i) in (\ref{space of nets}).
On the other hand, by (**) we have $\rk A_1(\xi)=4m+4$, i.e.
ranks of $A$ and of its submatrix $A_1(\xi)$ coincide. This yields, after multiplying the matrix $A$
by the invertible matrix of homomorphisms
$$
\left(\begin{array}{cc}
A_1(\xi)^{-1} & \mathbf{0} \\
A_2(\xi)^\vee\circ A_1(\xi)^{-1} & {\rm id}_{H_m^\vee\otimes V^\vee}
\end{array}\right)
$$
from the left, the following relation between the matrices $A_1(\xi),\ A_2(\xi)$ and $A_3(\xi)$:
\begin{equation}\label{xi3(A)}
A_3(\xi)=-A_2(\xi)^\vee\circ A_1(\xi)^{-1}\circ A_2(\xi),
\end{equation}
\begin{remark} This relation means that $A_3(\xi)$ is uniquely determined by $A_1(\xi)$ and $A_2(\xi)$. We will use this important
observation systematically in the sequel.
\end{remark}

For $m\ge1$ let ${\rm Isom}_{2m+1}$ be the set of all isomorphisms $\xi$ in (\ref{xi})
and set
\begin{equation}\label{MI2m+1(xi)}
MI_{2m+1}(\xi):=\{A\in MI'_{2m+1}\ |\ {\rm the\ skew-symmetric\ homomorphism}\ A_1(\xi)
\ {\rm in}\ (\ref{xi 1})\
\end{equation}
$$
{\rm is\ invertible}\},\ \ \ \xi\in{\rm Isom}_{2m+1}.
$$
In these notations we have the following result.
\vspace{0.3cm}
\noindent
\begin{theorem}\label{Theorem 4.2'}
For $m\ge3$ the following statements hold.

(i) There exists a dense subset ${\rm Isom}_{2m+1}^0$ of ${\rm Isom}_{2m+1}$
such that the sets $MI_{2m+1}(\xi),\ \xi\in{\rm Isom}_{2m+1}^0$, constitute an open cover of
$MI'_{2m+1}$ .

(ii) There exists a dense open subset ${\rm Isom}_{2m+1}^{00}$ of ${\rm Isom}_{2m+1}$ contained in ${\rm Isom}_{2m+1}^0$
such that the sets $MI_{2m+1}(\xi),\ \xi\in{\rm Isom}_{2m+1}^{00}$, are dense open subsets of
$MI'_{2m+1}$.

(iii) For any $\xi\in{\rm Isom}_{2m+1}^0$ and any $A\in MI_{2m+1}(\xi)$ the relation (\ref{xi3(A)})
is true.
\end{theorem}
\begin{proof}
(i)-(ii) Let $MI'_{2m+1}=M_1\cup...\cup M_s$ be a decomposition of $MI'_{2m+1}$ into irreducible components. Consider
the set $U:=\{(A,\xi)\in MI'_{2m+1}\times{\rm Isom_{2m+1}}\ |\ A_1(\xi):\ H_{m+1}\otimes V\to
H_{m+1}^\vee\otimes V^\vee$ is invertible $\}$ with
projections $MI'_{2m+1}\overset{p}\leftarrow U\overset{q}\to{\rm Isom_{2m+1}}$, and let
$U_i:=U\cap M_i\times{\rm Isom_{2m+1}}$ with the induced projections
$M_i\overset{p_i}\leftarrow U_i\overset{q_i}\to{\rm Isom_{2m+1}},\ \ i=1,...,s.$
By definition, $U$ is open in $ MI'_{2m+1}\times{\rm Isom_{2m+1}}$, hence each $U_i$ is open in
$M_i\times{\rm Isom_{2m+1}}$. Moreover, the property (**) implies that $p_i(U_i)=M_i$, so that $U_i$ is nonempty,
hence dense in $M_i\times{\rm Isom_{2m+1}}$ since both $M_i$ and ${\rm Isom_{2m+1}}$ are irreducible. (Note that
${\rm Isom_{2m+1}}$ is irreducible as a principal homogeneous space of the group $GL(2m+1)$.) Hence $q_i(U_i)$ contains
a dense open subset, say, $W_i$ of ${\rm Isom_{2m+1}}$. Set ${\rm Isom}_{2m+1}^0:=\underset{1\le i\le s}\cup q_i(U_i)$
and ${\rm Isom}_{2m+1}^{00}:=\underset{1\le i\le s}\cap W_i$. By construction, the sets
$MI_{2m+1}(\xi)\simeq q^{-1}(\xi),\ \xi\in{\rm Isom}_{2m+1}^0$, constitute an open cover of $MI'_{2m+1}$.
Respectively, for any $\xi\in{\rm Isom}_{2m+1}^{00}$ and each $i$, $1\le i\le s$, the set $q_i^{-1}(\xi)$ is nonempty
open, hence dense subset in $M_i$. This yields that, for $\xi\in{\rm Isom}_{2m+1}^{00}$, the set
$MI'_{2m+1}(\xi))\simeq q^{-1}(\xi)=\underset{1\le i\le s}\cup q_i^{-1}(\xi)$
is dense open in $MI'_{2m+1}$.

(iii) This follows from (\ref{xi3(A)}) and (**).
\end{proof}

\vspace{0.3cm}
We will need below the following lemma.

\begin{lemma}\label{lemma on B,C}
For $\xi\in{\rm Isom}_{2m+1}^0$ and $A\in MI_{2m+1}(\xi)$, set
\begin{equation}\label{B,C}
B:=A_1(\xi),\ \ \ C:=A_2(\xi).
\end{equation}
Then the following statements hold.

(i) Consider a subbundle morphism
\begin{equation}\label{alpha xi,A}
\alpha_{\xi,A}:=j^{-1}_{\xi,A}\circ a_A\circ\xi:
(H_{m+1}\oplus H_m)\otimes\mathcal{O}_{\mathbb{P}^3}(-1)\to
H_{m+1}\otimes V\otimes\mathcal{O}_{\mathbb{P}^3}.
\end{equation}
Then there exists an epimorphism
\begin{equation}\label{lambda xi,A}
\lambda_{\xi,A}:\coker(B\circ\alpha_{\xi,A})\twoheadrightarrow
H_{m+1}^\vee\otimes\mathcal{O}_{\mathbb{P}^3}(1).
\end{equation}
making commutative the diagram
\begin{equation}\label{diag lambda}
\xymatrix{
H_{m+1}^\vee\otimes V^\vee\otimes\mathcal{O}_{\mathbb{P}^3} \ar[r]^-{can}
\ar[rd]_{u^\vee} & \coker(B\circ\alpha_{\xi,A})\ar[d]^{\lambda_{\xi,A}} \\
& H_{m+1}^\vee\otimes\mathcal{O}_{\mathbb{P}^3}(1),
}
\end{equation}
where $can$ is the canonical surjection.

(ii) Consider the commutative diagram
\begin{equation}\label{diag tau}
\xymatrix{
& H_m\otimes\mathcal{O}_{\mathbb{P}^3}(-1) & & & \\
0\ar[r] & (H_{m+1}\oplus H_m)\otimes\mathcal{O}_{\mathbb{P}^3}(-1)
\ar[r]^-{B\circ\alpha_{\xi,A}}\ar@{->>}[u] &
H_{m+1}^\vee\otimes V^\vee\otimes\mathcal{O}_{\mathbb{P}^3}\ar[r]^-{can} &
\coker(B\circ\alpha_{\xi,A}) \ar[r] & 0 \\
0\ar[r] & H_{m+1}\otimes\mathcal{O}_{\mathbb{P}^3}(-1)
\ar[r]^{B\circ u\ \ \ }\ar@{>->}[u]_{i_{m+1}} &
H_{m+1}^\vee\otimes V^\vee\otimes\mathcal{O}_{\mathbb{P}^3}
\ar[r]^{\ \ \ v\circ B^{-1}}\ar@{=}[u] &
H_{m+1}\otimes T_{\mathbb{P}^3}(-1) \ar[r]\ar@{->>}[u]_{\epsilon_{\xi,A}} & 0 \\
 & & & H_m\otimes\mathcal{O}_{\mathbb{P}^3}(-1),\ar@{>->}[u]_{\tau_{\xi,A}}
}
\end{equation}
where $\tau_{\xi,A}$ and $\epsilon_{\xi,A}$ are the induced morphisms. Then the morphism
$\tau_{\xi,A}$ is a subbundle morphism fitting in a commutative diagram
\begin{equation}\label{diag C,tau}
\xymatrix{
H_{m+1}^\vee\otimes V^\vee\otimes\mathcal{O}_{\mathbb{P}^3}\ar[r]^{\ \ v\circ B^{-1}} &
H_{m+1}\otimes T_{\mathbb{P}^3}(-1) \\
H_m\otimes\mathcal{O}_{\mathbb{P}^3}(-1)\ar[u]_{C\circ u}\ar@{=}[r]
& H_m\otimes\mathcal{O}_{\mathbb{P}^3}(-1).\ar@{>->}[u]_{\tau_{\xi,A}}
}
\end{equation}

\end{lemma}

\begin{proof}
(i) Consider the commutative diagram
\begin{equation}\label{diag with 2 monads}
\xymatrix{
H_{2m+1}\otimes\mathcal{O}(-1)\ \ar@{>->}[r]^{a_A} &
W_A\otimes\mathcal{O}\ar[r]^{q_A}_{\simeq} & W_A^\vee\otimes\mathcal{O}
\ar@{->>}[r]^{a_A^\vee} \ar[d]^{j_{\xi,A}^\vee}_{\simeq}
& H_{2m+1}^\vee\otimes\mathcal{O}(1) \ar[d]^{\xi^\vee}_{\simeq}\\
(H_{m+1}\oplus H_m)\otimes\mathcal{O}(-1)\
\ar@{>->}[r]^{\ \ \ \alpha_{\xi,A}} \ar[u]^{\xi}_{\simeq}&
H_{m+1}\otimes V\otimes\mathcal{O}\ar[r]^{B}_{\simeq} \ar[u]^{j_{\xi,A}}_{\simeq}&
H_{m+1}^\vee\otimes V^\vee\otimes\mathcal{O}
\ar@{->>}[r]^{\alpha_{\xi,A}^\vee}\ar[dr]_{u^\vee} &
(H_{m+1}\oplus H_m)^\vee\otimes\mathcal{O}(1)
\ar@{->>}[d]^{i_{m+1}^\vee}\\
H_{m+1}\otimes\mathcal{O}(-1)\ar@{>->}[u]^{i_{m+1}}\ar[ur]_u & & &
H_{m+1}^\vee\otimes\mathcal{O}(1)
}
\end{equation}
Here the upper triple is the monad (\ref{Monad An}) for $n=2m+1$. Whence the statement (i)
follows.

(ii) Standard diagram chasing using (\ref{xi3(A)}), (\ref{B,C}) and diagram (\ref{diag tau}).
\end{proof}

\vspace{0.5cm}

\end{sub}

\begin{sub}\label{RemTH}{\bf Remarks on t'Hooft instantons.}
\label{THinstantons}
\rm

Consider the set
$$
I^{tH}_{2m+1}:=\{[E]\in I_{2m+1}\ |\ h^0(E(1))\ne0\},
$$
of {\it t'Hooft instanton bundles} and the corresponding set of {\it t'Hooft instanton nets}
$$
MI^{tH}_{2m+1}:=\pi_n^{-1}(I^{tH}_{2m+1}).
$$
We collect some well-known facts about $I^{tH}_{2m+1}$ in the following Lemma - see \cite{BT},
\cite{NT}, \cite[Prop. 2.2]{Tju1}.

\begin{lemma}\label{tHooft}
Let $m\ge1$. Then the following statements hold.

(i) $I^{tH}_{2m+1}$ is an irreducible $(10m+9)$-dimensional subvariety of $I_{2m+1}$.
Respectively, $MI^{tH}_{2m+1}$ is an irreducible $(4m^2+14m+10)$-dimensional subvariety of
$I_{2m+1}$.

(ii) $I^{tH*}_{2m+1}:=I^{tH}_{2m+1}\cap I'_{2m+1}$ is a smooth dense open subset of
$I^{tH}_{2m+1}$ and
\begin{equation}\label{h0=1}
h^0(E(1))=1,\ \ \ \ [E]\in I^{tH*}_{2m+1}.
\end{equation}
(iii) $MI^{tH*}_{2m+1}$ is a smooth dense open subset of the set
$$
TH_{2m+1}:=\{A\in\mathbf{S}_{2m+1}|A=\sum_{i=1}^{2m+2}h^2\otimes w,\
{\rm where}\ h\in H_{2m+1}^\vee,w\in\wedge^2V^\vee,\ w\wedge w=0\}.
$$
\end{lemma}

We are going to extend the statement of Theorem \ref{Theorem 4.2'} to the cases $m=1$ and 2.
To this end, for $m=1,2$ and $\xi\in{\rm Isom}_{2m+1}$ consider the sets $MI_{2m+1}(\xi)$
defined in (\ref{MI2m+1(xi)}) and set
\begin{equation}\label{MI' for m=1,2}
MI''_{2m+1}:=\underset{\xi\in{\rm Isom}_{2m+1}}\cup MI_{2m+1}(\xi),\ \ \ m=1,2.
\end{equation}
For $m=1,2$, fix an isomorphism $\xi^0\in{\rm Isom}_{2m+1}$,
$\xi^0:H_{m+1}\oplus H_m\overset{\sim}\longrightarrow H_{2m+1}$ and fix
a basis $\{h_1,...,h_{2m+1}\}$ in $H_{2m+1}^\vee$ such that
$\{h_1,...,h_m\}$ in $H_{2m+1}^\vee$ and $\{h_{m+2},...,h_{2m+1}\}$ in $H_{2m+1}^\vee$;
respectively, let $e_1,...,e_4$ be some fixed basis in $V^\vee$. Consider the nets
$A^{(m)}\in TH_{2m+1},\ \ m=1,2,$ defined as follows
\begin{equation}\label{A(1),A(2)}
A^{(1)}=h_1^2\otimes(e_1\wedge e_2+e_3\wedge e_4)+h_2^2\otimes(e_1\wedge e_3+e_4\wedge e_2),
\end{equation}
$$
A^{(2)}=h_1^2\otimes(e_1\wedge e_2+e_3\wedge e_4)+h_2^2\otimes(e_1\wedge e_3+e_4\wedge e_2)+
h_3^2\otimes(e_1\wedge e_4+e_2\wedge e_3).
$$
It is an exercise to show that, in the notation of (\ref{xi 1,2,3}), the homomorphisms
$$
A^{(m)}_1(\xi^0):H_{m+1}\otimes V\to H_{m+1}^\vee\otimes V^\vee,\ \ \ m=1,2,
$$
are invertible. On the other hand, for a given $\xi\in{\rm Isom}_{2m+1}$, the condition that a
homomorphism $A_1(\xi):H_{m+1}\otimes V\to H_{m+1}^\vee\otimes V^\vee$ is
invertible is an open condition on the net $A\in TH_{2m+1}$, respectively, on the net $A\in \mathbf{S}_{2m+1}$.
Since the sets $MI'_{2m+1},\ m=1,2,$ are irreducible (see \cite{CTT}), this together with Lemma \ref{tHooft}
yields the following corollary.

\begin{corollary}\label{tHooft2}

(i) For $m=1,2$ the set $MI''_{2m+1}$ is a dense open subset of $MI'_{2m+1}$ and of $MI_{2m+1}$,
and the statement of Theorem \ref{Theorem 4.2'} extends to the cases $m=1$ and 2, if we substitute
$MI'_{2m+1}$ by $MI''_{2m+1}$ and take for ${\rm Isom}_{2m+1}^0={\rm Isom}_{2m+1}^{00}$ any nonempty open subset of
${\rm Isom}_{2m+1}$ contained in the set $\{\xi\in{\rm Isom}_{2m+1}\ |\ MI_{2m+1}(\xi)\ne\emptyset\}$.

(ii) Let $m\ge1$. The set
$$
MI^{tH**}_{2m+1}:=\left\{\begin{matrix}
MI''_{2m+1}\cap MI^{tH*}_{2m+1},\ \ \ m=1,2,\ \cr
MI^{tH*}_{2m+1},\ \ \ \ \ \ \ \ \ \ \ \ \ m\ge3,\cr
\end{matrix}
\right.
$$
is a dense open subset of $MI^{tH*}_{2m+1}$, respectively, of $MI^{tH}_{2m+1}$.

(iii) For $m\ge1$ let
$$
MI^{tH}_{2m+1}(\xi):=MI^{tH**}_{2m+1}\cap MI_{2m+1}(\xi),\ \ \ \xi\in{\rm Isom}_{2m+1}.
$$
The set
\begin{equation}\label{MItH2m+1(xi)}
{\rm Isom}_{2m+1}^{tH}:=\{\xi\in{\rm Isom}_{2m+1}\ |\ MI^{tH}_{2m+1}(\xi)\ne\emptyset\}
\end{equation}
is a dense open subset of ${\rm Isom}_{2m+1}$ such that $MI^{tH**}_{2m+1}$ is covered by dense open subsets
$MI^{tH}_{2m+1}(\xi),\ \ \xi\in{\rm Isom}_{2m+1}^{tH}$.
\end{corollary}

\begin{remark}\label{isom-tH}
From the definition of the sets ${\rm Isom}_{2m+1}^0$, $MI^{tH}_{2m+1}(\xi)$ and ${\rm Isom}_{2m+1}^{tH}$ it follows
immediately that ${\rm Isom}_{2m+1}^{tH}\subset{\rm Isom}_{2m+1}^0$ and $MI^{tH}_{2m+1}(\xi)\subset MI_{2m+1}(\xi)$ for
$\xi\in{\rm Isom}_{2m+1}^{tH}$.
\end{remark}

Now (\ref{MI'n dense}), Theorem \ref{Theorem 4.2'} and Corollary \ref{tHooft2} yield
\begin{corollary}\label{density} Let $m\ge1$. Then for any $\xi\in{\rm Isom}_{2m+1}^0$
(respectively, for any $\xi\in{\rm Isom}_{2m+1}^{00}$) the scheme
$(MI_{2m+1}(\xi))_{red}$ is open (respectively, dense open) in $(MI_{2m+1})_{red}$.
In particular,
\begin{equation}\label{dim MI2m+1(xi)}
\dim_A MI_{2m+1}(\xi)=\dim_A MI_{2m+1},\ \ \ A\in MI_{2m+1}(\xi),\ \ \ \ \xi\in{\rm Isom}_{2m+1}^{00}.
\end{equation}
\end{corollary}

\end{sub}

\vspace{1cm}

\section{Invertible nets of quadrics from $\mathbf{S}_{m+1}$ and symplectic rank-$(2m+2)$ bundles}\label{symplectic E2m+2}

\vspace{0.5cm}

\begin{sub}\label{rho}{\bf Construction of symplectic rank-$(2m+2)$ bundles from invertible nets of quadrics from
$\mathbf{S}_{m+1}$.}

\rm
\vspace{0.5cm}

In this subsection we show that each invertible net of quadrics $B\in\mathbf{S}_{m+1}$ naturally leads to a construction
of a symplectic rank-$(2m+2)$ vector bundle $E_{2m+2}(B)$ on $\mathbb{P}^3$.
Let us introduce more notation. Set
\begin{equation}\label{Nm+1}
\mathbf{S}^0_{m+1}:=\{B\in\mathbf{S}_{m+1}\ |\
B:H_{m+1}\otimes V\to H_{m+1}^\vee\otimes V^\vee
\ {\rm is\ an\ invertible\ homomorphism}\}.
\end{equation}
The set $\mathbf{S}^0_{m+1}$ is a dense open subset of the vector space $\mathbf{S}_{m+1}$,
and it is easy to see that for any $B\in\mathbf{S}^0_{m+1}$ the following conditions are satisfied.\\
(1) The morphism
$\widetilde{B}:H_{m+1}\otimes\mathcal{O}_{\mathbb{P}^3}(-1)\to
H_{m+1}^\vee\otimes\Omega_{\mathbb{P}^3}(1)$
induced by the homomorphism
$B:H_{m+1}\otimes V\to H_{m+1}^\vee\otimes V^\vee$
is a subbundle morphism, i.e.
\begin{equation}\label{E2m+2}
E_{2m+2}(B):=\coker(\widetilde{B})
\end{equation}
is a vector bundle of rank $2m+2$ on $\mathbb{P}^3$. This follows from the diagram
\begin{equation}\label{diag E2m+2}
\xymatrix{
& & 0\ar[d] & 0\ar[d] & & \\
& 0\ar[r] & H_{m+1}\otimes\mathcal{O}_{\mathbb{P}^3}(-1)
\ar[r]^{\widetilde{B}}\ar[d]^u
& H_{m+1}^\vee\otimes\Omega_{\mathbb{P}^3}(1)\ar[r]^{\ \ \ e}\ar[d]^{v^\vee} &
E_{2m+2}(B)\ar[r] &  0\\
& &H_{m+1}\otimes V\otimes\mathcal{O}_{\mathbb{P}^3}
\ar[r]^{B\ }_-{\simeq}\ar[d]^v
& H_{m+1}^\vee\otimes V^\vee\otimes\mathcal{O}_{\mathbb{P}^3}\ar[d]^{u^\vee} &  & \\
&0\to E_{2m+2}(B)^\vee\ar[r] &
H_{m+1}\otimes T_{\mathbb{P}^3}(-1)\ar[r]^-{\widetilde{B}^\vee}\ar[d] &
H_{m+1}^\vee\otimes\mathcal{O}_{\mathbb{P}^3}(1)\ar[r]\ar[d] & 0 & \\
& & 0 & 0 & &
}
\end{equation}
(2) The homomorphism
${}^\sharp B: H_{m+1}\to H_{m+1}^\vee\otimes\wedge^2V^\vee$
induced by
$B:H_{m+1}\otimes V\to H_{m+1}^\vee\otimes V^\vee$
is injective. This follows from the commutative diagram extending the upper horizontal triple in
(\ref{diag E2m+2})

\begin{equation}\label{extend}
\xymatrix{
& & 0\ar[d] & 0\ar[d] & \\
& & H_{m+1}^\vee\otimes T_{\mathbb{P}^3}(-2)\ar@{=}[r]\ar[d] &
H_{m+1}^\vee\otimes T_{\mathbb{P}^3}(-2)\ar[d] & \\
0\ar[r] & H_{m+1}\otimes\mathcal{O}_{\mathbb{P}^3}
\ar[r]^-{{}^\sharp B}\ar@{=}[d] &
H_{m+1}^\vee\otimes\wedge^2V^\vee\otimes\mathcal{O}_{\mathbb{P}^3}
\ar[r]^-{can}\ar[d]^w &
H^0(E_{2m+2}(B)(1))\otimes\mathcal{O}_{\mathbb{P}^3}\ar[r]\ar[d]^{ev} & 0 \\
0\ar[r] & H_{m+1}\otimes\mathcal{O}_{\mathbb{P}^3}\ar[r]^{\widetilde{B}} &
H_{m+1}^\vee\otimes\Omega_{\mathbb{P}^3}(2)\ar[r]^e\ar[d] &
E_{2m+2}(B)(1)\ar[r]\ar[d] & 0 \\
& & 0 & \ 0, &}
\end{equation}
where $w$ is the morphism induced by the morphism $v$ from the Euler exact sequence in
(\ref{diag E2m+2}). From this diagram we obtain an isomorphism
\begin{equation}\label{H0(E2m+2(1))}
\coker({}^\sharp B)\simeq H^0(E_{2m+2}(B)(1)).
\end{equation}

(3) Diagram (\ref{diag E2m+2}) and the Five-Lemma yield an isomorphism
\begin{equation}\label{theta}
\theta:E_{2m+2}(B)\overset{\sim}\to E_{2m+2}(B)^\vee
\end{equation}
which is in fact symplectic,
$$
\theta^\vee=-\theta,
$$
since the homomorphism
$B:H_{m+1}\otimes V\to H_{m+1}^\vee\otimes V^\vee$ is skew-symmetric.
The isomorphism $\theta$ together with the upper triple from (\ref{diag E2m+2})
and its dual fits in the commutative diagram
\begin{equation}\label{diag display E2m+2 }
\xymatrix{
& & 0\ar[d] & 0\ar[d] & \\
0\ar[r] & H_{m+1}\otimes\mathcal{O}_{\mathbb{P}^3}(-1)\ar[r]^{\widetilde{B}}
\ar@{=}[d] &
H_{m+1}^\vee\otimes\Omega_{\mathbb{P}^3}(1)\ar[r]^e\ar[d]^{v^\vee} &
E_{2m+2}(B)\ar[r]\ar[d]^{e^\vee\circ\theta} & 0 \\
0\ar[r] & H_{m+1}\otimes\mathcal{O}_{\mathbb{P}^3}(-1)\ar[r]^-{B\circ u} &
H_{m+1}^\vee\otimes V^\vee\otimes\mathcal{O}_{\mathbb{P}^3}\ar[r]
^{\ \ v\circ B^{-1}}\ar[d]^{u^\vee} &
H_{m+1}\otimes T_{\mathbb{P}^3}(-1)\ar[r]\ar[d]^{\widetilde{B}^\vee} & 0 \\
& & H_{m+1}^\vee\otimes\mathcal{O}_{\mathbb{P}^3}(1)\ar@{=}[r]\ar[d]
& \ H_{m+1}^\vee\otimes\mathcal{O}_{\mathbb{P}^3}(1)\ar[d] & \\
& & 0 & \  0. &}
\end{equation}
Note that the upper horizontal triple in (\ref{diag E2m+2}) immediately implies
\begin{equation}\label{vanish cohom of E2m+2}
h^0(E_{2m+2}(B))=h^i(E_{2m+2}(B)(-2))=0,\ \ \ i\ge0.
\end{equation}

\vspace{0.5cm}

\end{sub}

\begin{sub}\label{Rel E2 and E2m+2}{\bf Relation between instantons and rank-$(2m+2)$ symplectic bundles.}
\label{Rel inst sympl}
\rm
\vspace{0.5cm}

For $m\ge1$ let $\xi\in{\rm Isom}_{2m+1}^0$ and $A\in MI_{2m+1}(\xi)$.
In this subsection we relate an instanton vector bundle $E(A)$ to a symplectic rank-$(2m+2)$ vector bundle
$E_{2m+2}(B)$ for $B=A_1(\xi)$. We will show that $E(A)$ is a cohomology sheaf of the monad (\ref{monad for E2(xi,A)})
defined by the data $(\xi,A)$ with $E_{2m+2}(B)$ in the middle - see Lemma \ref{E(xi,A)=E(A)}.

In fact, since $\xi\in{\rm Isom}_{2m+1}^0$, the homomorphism $B:H_{m+1}\otimes V\to H_{m+1}^\vee\otimes V^\vee$
by definition lies in $ \mathbf{S}^0_{m+1}$. Hence by Lemma \ref{lemma on B,C} the
diagram (\ref{diag C,tau}) holds. This diagram together with
(\ref{diag display E2m+2 }) implies $\widetilde{B}^\vee\circ\tau_{\xi,A}=0$
(note that in (\ref{diag C,tau}) ${\rm im}(C\circ u)\subset H_{m+1}^\vee\otimes\Omega_{\mathbb{P}^3}(1)$ since
$C\in\mathbf{\Sigma}_{m+1}$), so that there exists a morphism
\begin{equation}\label{rho xi,A}
\rho_{\xi,A}:H_m\otimes\mathcal{O}(-1)\to E_{2m+2}(B)
\end{equation}
such that
$\tau_{\xi,A}=e^\vee\circ\theta\circ\rho_{\xi,A}$. Since $\tau_{\xi,A}$ is a subbundle morphism,
$\rho_{\xi,A}$ is also a subbundle morphism. Moreover, diagrams (\ref{diag C,tau}) and
(\ref{diag display E2m+2 }) yield a commutative diagram

\begin{equation}\label{diag rho}
\xymatrix{
H_{m+1}^\vee\otimes\Omega_{\mathbb{P}^3}(1)\ar[rr]^e\ar[dd]^{v^\vee} & &
E_{2m+2}(B)\ar[dd]^{e^\vee\circ\theta} \\
& H_m\otimes\mathcal{O}(-1)\ar@{>->}[ul]_{{}^\sharp C}\ar@{>->}[dl]_{\widetilde{C}}
\ar@{>->}[ur]^{\rho_{\xi,A}}\ar@{>->}[dr]^{\tau_{\xi,A}} & \\
H_{m+1}^\vee\otimes V^\vee\otimes\mathcal{O}\ar[rr]^{\ \ v\circ B^{-1}} & &
H_{m+1}\otimes T_{\mathbb{P}^3}(-1). &}
\end{equation}

Diagrams (\ref{diag display E2m+2 }) and (\ref{diag rho}) yield a commutative diagram

\begin{equation}\label{diag DC}
\xymatrix{
H_m\otimes\mathcal{O}(-1)\ar[rrr]^{\widetilde{C}}
\ar@{>->}[dr]_{\rho_{\xi,A}}\ar@{>->}[drr]^{{}^\sharp C}\ar[ddd]_{D_C}
& & & H_{m+1}^\vee\otimes V^\vee\otimes\mathcal{O}\ar[ddd]^{B^{-1}}_{\simeq} \\
& E_{2m+2}(B)\ar[d]^{\theta}_{\simeq} &
H_{m+1}^\vee\otimes \Omega_{\mathbb{P}^3}(1)\ar@{->>}[l]^-{e}\ar[d]^{e^\vee\circ\theta\circ e}
\ar@{>->}[ur]^{v^\vee} & \\
& E_{2m+2}(B)^\vee\ar@{>->}[r]^-{e^\vee}\ar@{->>}[dl]_{\rho_{\xi,A}^\vee} &
H_{m+1}\otimes T_{\mathbb{P}^3}(-1)\ar@{->>}[dll]^{{}^\sharp C^\vee} & \\
H_m^\vee\otimes\mathcal{O}(1) & & &
H_{m+1}\otimes V\otimes\mathcal{O}\ar[lll]^{\widetilde{C}^\vee}\ar@{->>}[ul]_v,
}
\end{equation}
where
$D_C:=-\widetilde{C}^\vee\circ B^{-1}\circ\widetilde{C}=
-u^\vee\circ(C^\vee\circ B^{-1}\circ C)\circ u$
is the zero map. In fact, by (\ref{xi3(A)}) and (\ref{B,C}) we have $D_C=p_2(A_3(\xi))$,
where
$p_2:\wedge^2(H_n^\vee\otimes V^\vee)\to\wedge^2H_n^\vee\otimes S^2V^\vee$
is the projection onto the second direct summand of the decomposition (\ref{can decomp}). Since
by (\ref{xi 1,2,3}) $A_3(\xi)$ lies in the first direct summand of (\ref{can decomp}) it follows
that $D_C=0$. We thus obtain a monad
\begin{equation}\label{monad for E2(xi,A)}
0\to H_m\otimes\mathcal{O}(-1)\overset{\rho_{\xi,A}}\longrightarrow E_{2m+2}(B)
\overset{\rho_{\xi,A}^\vee\circ\theta}\longrightarrow H_m^\vee\otimes\mathcal{O}(1)\to0
\end{equation}
with cohomology sheaf
\begin{equation}\label{E2(xi,A)}
E_2(\xi,A):=\ker(\rho_{\xi,A}^\vee\circ\theta)/\im\rho_{\xi,A}
\end{equation}
which is a vector bundle since $\rho_{\xi,A}$ is a subbundle morphism. Furthermore, by
(\ref{vanish cohom of E2m+2}) it follows from the monad
(\ref{monad for E2(xi,A)}) that $E_2(\xi,A)$ is a $(2m+1)$-instanton,
\begin{equation}\label{again E2(xi,A)}
[E_2(\xi,A)]\in I_{2m+1}.
\end{equation}
\begin{lemma}\label{E(xi,A)=E(A)}
$E_2(\xi,A)\simeq E(A)$, where the sheaf $E(A)$ is defined in (\ref{coho sheaf}).
\end{lemma}
\begin{proof}
Diagram chasing using (\ref{xi3(A)}), (\ref{diag tau})-(\ref{diag with 2 monads}), (\ref{diag E2m+2})-(\ref{extend}) and
(\ref{diag display E2m+2 }).
\end{proof}

\end{sub}

\vspace{1cm}

\section{Scheme $X_m$. An isomorphism between $X_m$ and an open subset of the
space $(MI_{2m+1})_{red}$}\label{Space Xm}

\vspace{1cm}

In this section we introduce a locally closed subset $X_m$ of the vector space $\mathbf{S}_{m+1}\oplus\mathbf{\Sigma}_{m+1}$ and
prove in Theorem \ref{Xm isom MI2m+1} below that this subset, considered as a reduced scheme, is isomorphic to the reduced scheme
$(MI_{2m+1}(\xi))_{red}$ for any $\xi\in{\rm Isom}_{2m+1}^0$. The set $X_m$ is defined as follows:
\begin{equation}\label{Xm}
X_m:=\left\{(B,C)\in\mathbf{S}_{m+1}^0\times\mathbf{\Sigma}_{m+1}\ \left|\
\begin{matrix}
(i)\ (C^\vee\circ B^{-1}\circ C:H_m\otimes V\to H_m^\vee\otimes V^\vee)\in
\mathbf{S}_m,\ \ \ \ \ \ \ \ \ \ \ \ \ \ \ \\ \cr
(ii)\ {\rm the\ map}\ (H_{m+1}\oplus H_m)\otimes\mathcal{O}
\overset{(B,C)\circ u}\longrightarrow
H_{m+1}^\vee\otimes V^\vee\otimes\mathcal{O}(1)\ \ \\
{\rm is\ a\ subbundle\ morphism}, \cr
(iii)\ {\rm the\ composition}\ \hat{C}:H_m\overset{{}^\sharp C}\to
H_{m+1}^\vee\otimes\wedge^2V^\vee\overset{can}\twoheadrightarrow
\ \ \ \ \ \ \ \ \ \ \ \ \ \ \cr
\ \ \ \ \ H_{m+1}^\vee\otimes\wedge^2V^\vee/
\im({}^\sharp B)
\simeq H^0(E_{2m+2}(B)(1))\ {\rm yields}\cr
{\rm a\ subbundle\ morphism}\ \ \cr
H_m\otimes\mathcal{O}_{\mathbb{P}^3}(-1)\overset{\rho_{B,C}}\to E_{2m+2}(B),\cr
{\rm i.e.}\ \rho_{B,C}^\vee\ {\rm is\ surjective\ and}\ E_2(B,C):=
\Ker({}^t\rho_{B,C})/\im(\rho_{B,C})\cr
{\rm is\ locally\ free}\ \ \ \ \ \ \ \ \ \ \ \ \ \ \ \ \cr
\end{matrix}
\right.
\right\}.
\end{equation}
By definition $X_m$ is a locally closed subset of
$\mathbf{S}_{m+1}^0\times\mathbf{\Sigma}_{m+1}$.
Hence it is naturally endowed with the structure of a reduced scheme.

Note that in the condition (iii) of the definition of $X_m$ we set
${}^t\rho_{B,C}:=\rho_{B,C}^\vee\circ\theta$,
where
$\theta:E_{2m+2}(B)\overset{\sim}\to E^\vee_{2m+2}(B)$
is the natural symplectic structure on $E_{2m+2}(B)$ defined in (\ref{theta}).

\begin{theorem}\label{Xm isom MI2m+1}
Let $m\ge1$ and let $\xi\in{\rm Isom}_{2m+1}^0$.

(i) There is an isomorphism of reduced schemes
\begin{equation}\label{isomorphism f}
f_m:\ (MI_{2m+1}(\xi))_{red}\overset{\simeq}\to X_m:
\ A\mapsto(A_1(\xi),A_2(\xi)).
\end{equation}

(ii) The inverse isomorphism is given by the formula
\begin{equation}\label{isomorphism f^-1}
g_m:\ X_m\overset{\simeq}\to (MI_{2m+1}(\xi))_{red}:
\ (B,C)\mapsto\ \widetilde{\xi}(B,\ C,\ -C^\vee\circ B^{-1}\circ C).
\footnote{Here we use the decomposition (\ref{direct sum}) fixed by the choice of $\xi$.}
\end{equation}
\end{theorem}
\begin{proof}
(i) We first show that the image of the map
$f_m:\ (MI_{2m+1}(\xi))_{red}\to\mathbf{S}_{m+1}^0\times\Sigma_{m,m+1}^{in}$
lies in $X_m$, i.e. satisfies the conditions (i)-(iii) in the definition of $X_m$.
Indeed, the condition (i) is automatically satisfied, since (\ref{xi 1,2,3}) and (\ref{xi3(A)})
give $-C^\vee\circ B^{-1}\circ C=-A_2(\xi)^\vee\circ A_1(\xi)^{-1}\circ A_2(\xi)=A_3(\xi)\in
S^2H_m^\vee\otimes\wedge^2V^\vee.$ Next, the morphism $\rho_{B,C}$ defined in (\ref{Xm}.iii)
above coincides by its definition with the morphism $\rho_{\xi,A}$ defined in (\ref{rho xi,A}).
In fact, the upper triangle of the diagram (\ref{diag rho}) twisted by $\mathcal{O}(1)$ and the
lower part of the diagram (\ref{extend})
fit in the diagram
\begin{equation}\label{combine}
\xymatrix{
0\to H_{m+1}\otimes\mathcal{O}\ar[r]^-{{}^\sharp B}\ar@{=}[dd] &
H_{m+1}^\vee\otimes\wedge^2V^\vee\otimes\mathcal{O}\ar[rr]^{can}\ar[dd]_w
& & H^0(E_{2m+2}(B)(1))\otimes\mathcal{O}\ar[r]\ar[dd]^{ev} &0\\
& & H_m\otimes\mathcal{O}\ar@{>->}[ul]_-{{}^\sharp C}\ar@{>->}[ur]^-{\widehat{C}}
\ar@{>->}[dr]^{\rho_{\xi,A}}\ar@{>->}[dl]_{\widetilde{C}} & & \\
0\to H_{m+1}\otimes\mathcal{O}\ar[r]^{\widetilde{B}} &
H_{m+1}^\vee\otimes\Omega(2)\ar[rr]^e & & E_{2m+2}(B)(1)\ar[r] &0,
}
\end{equation}
where the composition $\widehat{C}=can\circ C$ is defined in the condition (iii) of the
definition of $X_m$. Whence
\begin{equation}\label{rho...=rho...}
\rho_{B,C}=\rho_{\xi,A}.
\end{equation}
Since $\rho_{\xi,A}$ is a subbundle
morphism, the condition (iii) is satisfied and, moreover, $\widehat{C}$ is a subbundle
morphism as well. Thus, the lower part of the diagram (\ref{combine}) shows that the morphism
$(\widetilde{B},\widetilde{C}):(H_{m+1}\oplus H_m)\otimes\mathcal{O}\to
H_{m+1}^\vee\otimes\Omega(2)$
is  a subbundle morphism. Hence its composition with the subbundle morphism
$v^\vee:H_{m+1}^\vee\otimes\Omega(2)\hookrightarrow
H_{m+1}^\vee\otimes V\otimes\mathcal{O}(1)$
is a subbundle morphism as well. By definition, this composition coincides with
$(B,C)\circ u.$
Hence the condition (ii) in the definition of $X_m$ is satisfied.

This shows that $f_m((MI_{2m+1}(\xi))_{red})$ lies in $X_m$.
Finally, the equality $g_m\circ f_m=id$ follows directly from (\ref{xi 1,2,3}) and (\ref{xi3(A)}).

(ii) We first prove that the image of the map
\begin{equation}\label{im g}
g_m:X_m\to\mathbf{S}_{2m+1}:\ (
B,C)\mapsto(B,\ C,\ C^\vee\circ B^{-1}\circ C)\ \footnote{
We identify here the triple
$(B,\ C,\ C^\vee\circ B^{-1}\circ C)$ with a point in
$S^2H_{2m+1}^\vee\otimes\wedge^2V^\vee$
via the decomposition (\ref{direct sum}).}
\end{equation}
lies in $(MI_{2m+1}(\xi))_{red}$. In fact, the subbundle morphism
$\mathcal{A}:=(B,C)\circ u:(H_{m+1}\oplus H_m)\otimes\mathcal{O}\to
H_{m+1}^\vee\otimes V^\vee\otimes\mathcal{O}(1)$ and its dual extend to the
right and left exact sequence
\begin{equation}\label{A,D}
0\to(H_{m+1}\oplus H_m)\otimes\mathcal{O}(-1)\overset{\mathcal{A}}\to
H_{m+1}^\vee\otimes V^\vee\otimes\mathcal{O}\overset{\mathcal{A}^\vee\circ B^{-1}}\to
(H_{m+1}\oplus H_m)^\vee\otimes\mathcal{O}(1)\to0.
\end{equation}
Furthermore, by definition $\mathcal{A}^\vee\circ B^{-1}\circ\mathcal{A}=u^\vee\circ A\circ u$, where
$A$ is the matrix
$\left(\begin{array}{cc}
B & C \\
-C^\vee & -C^\vee\circ B^{-1}\circ C
\end{array}\right)$.
Since the condition (i) of (\ref{Xm}) is satisfied, under the direct sum decomposition (\ref{direct sum})
this matrix $A$ can be treated as an element of $\mathbf{S}_{2m+1}$.
Hence $u^\vee\circ A\circ u=0$, i.e. (\ref{A,D}) is a monad. We will show that its cohomology bundle
$$
E(B,C):=\ker(\mathcal{A}^\vee\circ B^{-1})/\im\mathcal{A}
$$
is an $(2m+1)$-instanton, and this will give the desired inclusion
$g(X_m)\subset (MI_{2m+1}(\xi))_{red}.$
For this, consider  the diagram (\ref{diag tau}) in
which we substitute $B\circ\alpha_{\xi,A}$ by $\mathcal{A}$,
denote $\mathcal{G}:=\coker\mathcal{A}$, and change the notation for $\tau_{\xi,A}$ and
$\epsilon_{\xi,A}$, respectively, to $\tau_{B,C}$ and $\epsilon_{B,C}$:
\begin{equation}\label{diag tau D,C}
\xymatrix{
& H_m\otimes\mathcal{O}_{\mathbb{P}^3}(-1) & & & \\
0\ar[r] & (H_{m+1}\oplus H_m)\otimes\mathcal{O}_{\mathbb{P}^3}(-1)
\ar[r]^{\ \ \mathcal{A}}\ar@{->>}[u] &
H_{m+1}^\vee\otimes V^\vee\otimes\mathcal{O}_{\mathbb{P}^3}\ar[r]^{\ \ \ can} &
\mathcal{G}\ar[r] & 0 \\
0\ar[r] & H_{m+1}\otimes\mathcal{O}_{\mathbb{P}^3}(-1)
\ar[r]^{B\circ u\ \ \ }\ar@{>->}[u]_{i_{m+1}} &
H_{m+1}^\vee\otimes V^\vee\otimes\mathcal{O}_{\mathbb{P}^3}
\ar[r]^{\ \ \ v\circ B^{-1}}\ar@{=}[u] &
H_{m+1}\otimes T_{\mathbb{P}^3}(-1) \ar[r]\ar@{->>}[u]_{\epsilon_{B,C}} & 0 \\
 & & & H_m\otimes\mathcal{O}_{\mathbb{P}^3}(-1).\ar@{>->}[u]_{\tau_{B,C}}
}
\end{equation}
In these notations the diagram (\ref{diag display E2m+2 }) becomes the display of the antiselfdual
monad
\begin{equation}\label{monad for E2m+2}
0\to H_{m+1}\otimes\mathcal{O}(-1)\overset{B\circ u}\to
H_{m+1}^\vee\otimes V^\vee\otimes\mathcal{O}\overset{u^\vee}\to
H_{m+1}^\vee\otimes\mathcal{O}(1)\to0
\end{equation}
with the symplectic cohomology sheaf $E_{2m+2}(B)$:
\begin{equation}\label{E2m+2(D-1)}
E_{2m+2}(B)=\ker(u^\vee)/\im(B\circ u).
\end{equation}
Moreover, as in (\ref{rho xi,A}) and (\ref{diag rho}) we obtain a subbundle morphism
\begin{equation}\label{rho D,C}
\rho_{B,C}:H_m\otimes\mathcal{O}(-1)\to E_{2m+2}(B)
\end{equation}
such that
\begin{equation}\label{tau,rho D,C}
\tau_{B,C}=e^\vee\circ\theta\circ\rho_{B,C},
\end{equation}
where $\theta:E_{2m+2}(B)\overset{\simeq}\to E_{2m+2}(B)$ is a symplectic structure on
$E_{2m+2}(B)$. In addition, as in (\ref{vanish cohom of E2m+2}) we have
\begin{equation}\label{vanish cohom of E2m+2(D-1)}
h^0(E_{2m+2}(B))=h^i(E_{2m+2}(B)(-2))=0,\ \ \ i\ge0.
\end{equation}
Furthermore, the antiselfdual monads (\ref{A,D}) and (\ref{monad for E2m+2}) recover the
antiselfdual monad (\ref{monad for E2(xi,A)}) which in view of (\ref{rho...=rho...}) becomes
\begin{equation}\label{monad for E(D,C)}
0\to H_m\otimes\mathcal{O}(-1)\overset{\rho_{B,C}}\longrightarrow E_{2m+2}(B)
\overset{\rho_{B,C}^\vee\circ\theta}\longrightarrow H_m^\vee\otimes\mathcal{O}(1)\to0.
\end{equation}
with the cohomology sheaf $E(B,C)$,
\begin{equation}\label{E(D,C)=coho}
E(B,C)=\ker(\rho_{B,C}^\vee\circ\theta)/\im(\rho_{B,C}).
\end{equation}
Now (\ref{vanish cohom of E2m+2(D-1)}) and (\ref{monad for E(D,C)}) yield
$h^0(E(B,C))=h^i(E(B,C)(-2))=0,\ \ \ i\ge0$, i.e. $E(B,C)$ is an $(2m+1)$-instanton.

Thus $\im g_m\subset I_{2m+1}(\xi)$. The fact that $f_m\circ g_m=id$ follows directly from
(\ref{isomorphism f}) and (\ref{isomorphism f^-1}).
\end{proof}

\begin{remark}\label{cap=0}
Note that, since the morphism $\widehat{C}$ in the diagram (\ref{combine}) is injective, it follows from this diagram that,
for any $m\ge1,\ \xi\in{\rm Isom}_{2m+1}^0$ and any $A\in MI_{2m+1}(\xi)$, the monomorphisms
$H_{m+1}\overset{{}^\sharp A_1(\xi)}\hookrightarrow H_{m+1}^\vee\otimes
\wedge^2V^\vee\overset{{}^\sharp A_2(\xi)}\hookleftarrow H_m$
satisfy the condition
$\im({}^\sharp A_1(\xi))\cap\im({}^\sharp A_2(\xi))=\{0\}$, i. e.
$\dim\Span(\im({}^\sharp A_1(\xi)),\ \im({}^\sharp A_2(\xi)))=2m+1$.
\end{remark}

\vspace{1cm}

\section{Scheme $Z_m$.
Reduction of the irreducibility of $X_m$ to the irreducibility of $Z_m$.
Proof of main theorem}\label{Scheme Zm}

\vspace{0.5cm}

\begin{sub}{\bf Scheme $\widehat{Z}_m$ and its open subset $Z_m$.}\label{Definiton of Zm}
\rm
In this subsection we introduce a new set $Z_m$ as a locally closed subset of a certain vector space (see (\ref{Zm}))
and endow it
with a natural scheme structure. We then formulate Theorem \ref{Irreducibility of Zm} on the irreducibility of $Z_m$.
This Theorem plays a key role in the proof of irreducibility of $I_{2m+1}$ which we give in subsection
\ref{Proof of main Thm}. The proof of Theorem \ref{Irreducibility of Zm} will be given in the next section.

Set
\begin{equation}\label{Sm,Lambdam}
\mathbf{\Lambda}_m:=\wedge^2H_m^\vee\otimes S^2 V^\vee,\ \ \
\mathbf{\Phi}_m:=\Hom(H_m,H_m^\vee\otimes\wedge^2 V^\vee),
\end{equation}
and
\begin{equation}\label{Sm0vee}
(\mathbf{S}_m^\vee)^0:=\{D\in\mathbf{S}_m^\vee\ |\ D: H_m^\vee\otimes V^\vee\to H_m\otimes V\ {\rm is\ invertible}\}.
\end{equation}
Note that $(\mathbf{S}_m^\vee)^0$ is a dense open subset of $\mathbf{S}_m^\vee$ and there is a canonical isomorphism
\begin{equation}\label{isom Sm0}
\mathbf{S}_m^0\overset{\simeq}\to(\mathbf{S}_m^\vee)^0:\ A\mapsto A^{-1}.
\end{equation}

Consider the sets
\begin{equation}\label{tildeZm}
\widehat{Z}_m:=\left\{(D,\phi)\in\mathbf{S}_m^\vee\times\mathbf{\Phi}_m\ \left|\
\begin{matrix}
\Theta(D,\phi):=\phi^\vee\circ D\circ \phi:H_m\otimes V\to\cr
\to H_m^\vee\otimes V^\vee\ {\rm satisfies\ the\ condition}\cr
\Theta(D,\phi)\in \mathbf{S}_m\ \cr
\end{matrix}\right.
\right\}.
\end{equation}
and
\begin{equation}\label{Zm}
Z_m:=\widehat{Z}_m\cap(\mathbf{S}_m^\vee)^0\times\mathbf{\Phi}_m
\end{equation}
(here we understand a point $D\in\mathbf{S}_m^\vee$ as a homomorphism $H_m^\vee\otimes V^\vee\to H_m\otimes V$)
and let $\overline{Z}_m$ be the closure of $Z_m$ in $\mathbf{S}_m^\vee\times\mathbf{\Phi}_m$.
By definition, $Z_m$ is an open subset of $\widehat{Z}_m$, respectively, a dense open subset of $\overline{Z}_m$.

Note that there is a standard decomposition
$$
\wedge^2(H_m^\vee\otimes V^\vee)=\mathbf{S}_m\oplus\mathbf{\Lambda}_m
$$
with induced projection onto the second summand
\begin{equation}\label{qm}
q_m:\ \wedge^2(H_m^\vee\otimes V^\vee)\rightarrow\mathbf{\Lambda}_m
\end{equation}
and the morphism
$$
h:\mathbf{S}_m^\vee\oplus\mathbf{\Phi}_m\to\mathbf{\Lambda}_m:
(D,\phi)\mapsto q_m(\Theta(D,\phi)).
$$
By the definition of $\widehat{Z}_m$ we obtain
\begin{equation}\label{hm^-1(0)}
\widehat{Z}_m=h^{-1}(0).
\end{equation}
Clearly, the point $(0,0)$ belongs to $\widehat{Z}_m$, i. e. $\widehat{Z}_m$ is nonempty.

\textit{Convention}: We endow $\widehat{Z}_m$ with the structure of a scheme-theoretic fibre $h^{-1}(0)$ of the
morphism $h$. Respectively, $Z_m$ inherits the structure of an open subscheme of $\widehat{Z}_m$.

\begin{remark}\label{Z as zeroset}
From (\ref{hm^-1(0)}) it follows that $\widehat{Z}_m$ may be considered as the zero-scheme
$(h^*\mathbf{s}_{taut})_0$ of the section $h^*\mathbf{s}_{taut}$ of the trivial vector bundle
$\mathbf{\Lambda}_m\otimes\mathcal{O}_{\mathbf{S}_m\oplus\mathbf{\Phi}_m}$,
where $\mathbf{s}_{taut}$ is the tautological section of the trivial vector bundle
$\mathbf{\Lambda}_m\otimes\mathcal{O}_{\mathbf{\Lambda}_m}$ of rank $\dim\mathbf{\Lambda}_m=5m(m-1)$ over $\mathbf{\Lambda}_m$.
We thus obtain the following estimate for the
dimension of $\widehat{Z}_m$ at each point $z\in \widehat{Z}_m$,
\begin{equation}\label{dim tildeZm}
\dim_z\widehat{Z}_m=\dim h^{-1}(0)\ge\dim(\mathbf{S}_m^\vee\times\mathbf{\Phi}_m)-
\dim\mathbf{\Lambda}_m=3m(m+1)+6m^2-5m(m-1)
\end{equation}
$$
=4m(m+2).
$$
In particular, if $Z_m$ is nonempty, then
\begin{equation}\label{dim Zm}
\dim_zZ_m\ge4m(m+2),\ \ \ \ z\in Z_m.
\end{equation}
\end{remark}
In the next subsection we will use the following result about $Z_m$.
\begin{theorem}\label{Irreducibility of Zm}
(i) $Z_m$ is an integral locally complete intersection scheme of dimension $4m(m+2)$.

(ii) The natural morphism
$p_m:Z_m\to(\mathbf{S}_m^\vee)^0: \ (D,\phi)\mapsto D$
is surjective.
\end{theorem}
We begin the proof of this theorem in section \ref{Study Zm} and finish in section \ref{gen pos}.

\end{sub}

\vspace{0.2cm}

\begin{sub}{\bf Proof of the main theorem.}\label{Proof of main Thm}
\rm

\vspace{0.2cm}
In this subsection we give the proof of Theorem \ref{Irreducibility}. Set
\begin{equation}\label{tilde Xm}
\widetilde{X}_m:=
\{(D,C)\in (\mathbf{S}_{m+1}^\vee)^0\times\mathbf{\Sigma}_{m+1}\ |\
\ (C^\vee\circ D\circ C:H_m\otimes V\to H_m^\vee\otimes V^\vee)
\in\mathbf{S}_m\}.
\end{equation}
The set $\widetilde{X}_m$ has a natural structure of a closed subscheme of
$(\mathbf{S}_{m+1}^\vee)^0\times\mathbf{\Sigma}_{m+1}$
defined by the equations
\begin{equation}\label{eqns of tilde Xm}
C^\vee\circ D\circ C\in \mathbf{S}_m.
\end{equation}
Since the conditions (ii) and (iii) in the definition (\ref{Xm}) of $X_m$ are open and $X_m$
is nonempty (see Theorem \ref{Xm isom MI2m+1}), it follows immediately in view of (\ref{isom Sm0}) that $X_m$ is a
nonempty open subset of $(\widetilde{X}_m)_{red}$,
\begin{equation}\label{open in tilde Xm}
\emptyset\ne X_m\overset{\rm open}\hookrightarrow(\widetilde{X}_m)_{red}.
\end{equation}

Fix a direct sum decomposition
$$
H_{m+1}\overset{\sim}\to H_m\oplus\mathbf{k}.
$$
Under this isomorphism any homomorphism
\begin{equation}\label{map C}
C\in\mathbf{\Sigma}_{m+1}=\Hom(H_m,H_{m+1}^\vee)\otimes\wedge^2 V^\vee,\ \ \
C:H_m\otimes V\to H_{m+1}^\vee\otimes V^\vee,\ \ \ \
\end{equation}
can be represented as a homomorphism
\begin{equation}\label{decompn C}
C:H_m\otimes V\to
H_m^\vee\otimes V^\vee\ \oplus\ \mathbf{k}^\vee\otimes V^\vee,
\end{equation}
i.e. as a matrix of homomorphisms
\begin{equation}\label{matrix of C}
C= \left(
\begin{array}{c}
\phi \\
\psi
\end{array}
\right),
\end{equation}
where
\begin{equation}\label{phi, b}
\phi\in\Hom(H_m,
H_m^\vee)\otimes\wedge^2 V^\vee=\mathbf{\Phi}_m,\ \ \
\psi\in\mathbf{\Psi}_m:=
\Hom(H_m,(\mathbf{k})^\vee)\otimes\wedge^2 V^\vee.
\end{equation}
Respectively, any homomorphism
$D\in (\mathbf{S}^\vee_{m+1})^0\subset\mathbf{S}^\vee_{m+1}= S^2H_{m+1}\otimes\wedge^2V\subset
\Hom(H_{m+1}^\vee\otimes V^\vee,H_{m+1}\otimes V)$
can be represented as a matrix of homomorphisms

\begin{equation}\label{matrix of D}
D= \left(
\begin{array}{cc}
D_1 & \lambda \\
-\lambda^\vee & \mu
\end{array}
\right),
\end{equation}
where
\begin{equation}\label{A, lambda, mu}
D_1\in\mathbf{S}^\vee_m\subset
\Hom(H_m^\vee\otimes V^\vee,H_m\otimes V),\ \ \
\end{equation}
$$
\lambda\in\mathbf{L}_m:=\Hom(\mathbf{k}^\vee,H_m)\otimes\wedge^2 V, \ \ \
\mu\in\mathbf{M}_m:=\Hom(\mathbf{k}^\vee,\mathbf{k})\otimes\wedge^2 V.
$$
From (\ref{matrix of C}) and (\ref{matrix of D}) it follows that the homomorphism
$$
C^\vee\circ D\circ C:H_m\otimes V\to
H_m^\vee\otimes V^\vee,\ \ \ \ \
C^\vee\circ D\circ C\in\wedge^2(H_m^\vee\otimes V^\vee),
$$
can be represented as
\begin{equation}\label{CDC}
C^\vee\circ D\circ C=\phi^\vee\circ D_1\circ\phi+\phi^\vee\circ\lambda\circ\psi-
\psi^\vee\circ\lambda^\vee\circ\phi+\psi^\vee\circ\mu\circ\psi.
\end{equation}
By (\ref{matrix of C})-(\ref{A, lambda, mu}) we have
$$
\mathbf{S}^\vee_{m+1}\times\mathbf{\Sigma}_{m+1}=\mathbf{S}^\vee_m\times\mathbf{\Phi}_m
\times\mathbf{\Psi}_m\times\mathbf{L}_m\times\mathbf{M}_m,
$$
and there are well-defined morphisms
$$
\tilde{p}_m:\widetilde{X}_m\to\mathbf{L}_m\oplus\mathbf{M}_m:(D_1,\phi,\psi,\lambda,\mu)\mapsto(\lambda,\mu).
$$
and
$$
p_m:=\tilde{p}_m|\overline{X}_m:\overline{X}_m\to\mathbf{L}_m\oplus\mathbf{M}_m,
$$
where $\overline{X}_m$ is the closure of $X_m$ in $(\mathbf{S}_{m+1}^\vee)^0\times\mathbf{\Sigma}_{m+1}$.
We now invoke the following proposition, the proof of which is postponed to Section \ref{Appendix}.
\begin{proposition}\label{nondeg for general}
Let $m\ge1$. Then, for any point $D\in(\mathbf{S}^\vee_{m+1})^0$ and a general choice of the decomposition
$H_{m+1}\overset{\sim}\to H_m\oplus\mathbf{k}$, the induced  homomorphism
$D_1$ in the matrix of homomorphisms $D$ in (\ref{matrix of D}) is nondegenerate.
\end{proposition}

According to this proposition, we fix such a decomposition $H_{m+1}\overset{\sim}\to H_m\oplus\mathbf{k}$ for which
the homomorphism $D_1:H_m^\vee\otimes V^\vee\to H_m\otimes V$ in (\ref{matrix of D})
is nondegenerate, i.e. $D_1\in(\mathbf{S}^\vee_m)^0$.

Let $\mathcal{X}$ be any irreducible component of $X_m$ and let $\overline{\mathcal{X}}$ be its closure in
$\overline{X}_m$. Fix a point
$z=(D_1,\phi,\psi,\lambda,\mu)\in \mathcal{X}$ not lying in the components of $X_m$ different from $\mathcal{X}$.
Consider the morphism
\begin{equation}\label{f^0}
f:\ \mathbb{A}^1\to\overline{\mathcal{X}}:\ t\mapsto(D_1,t^2\phi,t\psi,t\lambda,t^2\mu),\ \ \ f(1)=z.
\end{equation}
(This morphism is well-defined by (\ref{CDC}).) By definition, the point
$f(0)=(D_1,0,0,0,0)$ lies in the fibre $p_m^{-1}(0,0)$. Hence,
$p_m^{-1}(0,0)\cap\overline{\mathcal{X}}\ne\emptyset$.
In other words,
\begin{equation}\label{nonempty fibre}
\rho^{-1}(0,0)\ne\emptyset,\ \ \ \ \text{\it where}\ \ \ \rho:=p_m|\overline{\mathcal{X}}.
\end{equation}
Now from (\ref{CDC}) and the definition of $\widetilde{X}_m$ it follows that
\begin{equation}\label{zero fibre}
\tilde{p}_m^{-1}(0,0)=\{(D_1,\phi,\psi)\in(\mathbf{S}^\vee_m)^0\times\mathbf{\Phi}_m\times\mathbf{\Psi}_m\ |\
\phi^\vee\circ D_1\circ\phi\in\mathbf{S}_m\}.
\end{equation}
Comparing this with the definition (\ref{Zm}) of $Z_m$ we see that, set-theoretically,
$\tilde{p}_m^{-1}(0,0)=Z_m\times\mathbf{\Psi}_m$, so that
\begin{equation}\label{zero fibre=Zm times Psim}
\rho^{-1}(0,0)\overset{\rm sets}\subset p_m^{-1}(0,0)\overset{\rm sets}=
\tilde{p}_m^{-1}(0,0)\overset{\rm sets}=Z_m\times\mathbf{\Psi}_m.
\end{equation}
Respectively, scheme-theoretically we have embeddings of schemes
\begin{equation}\label{zero fibre subset Zm times Psim}
\rho^{-1}(0,0)\overset{\rm schemes}\subset p_m^{-1}(0,0)\overset{\rm schemes}\subset
\tilde{p}_m^{-1}(0,0)\overset{\rm schemes}=Z_m\times\mathbf{\Psi}_m.
\end{equation}
From (\ref{zero fibre=Zm times Psim}) and Theorem \ref{Irreducibility of Zm} it follows, in particular, that
\begin{equation}\label{dim fibre le dim Zm+dim Psim}
\dim\rho^{-1}(0,0)\le\dim p_m^{-1}(0,0)\le\dim Z_m+\dim\mathbf{\Psi}_m=4m(m+2)+6m=4m^2+14m.
\end{equation}
Hence in view of (\ref{nonempty fibre})
\begin{equation}\label{dim X' le}
\dim \overline{\mathcal{X}}\le\dim\rho^{-1}(0,0)+\dim\mathbf{L}_m+\dim\mathbf{M}_m\le4m^2+14m+6m+6=4m^2+20m+6.
\end{equation}
On the other hand, formula (\ref{dim MIn ge...}) for $n=2m+1$, equality (\ref{dim MI2m+1(xi)}) and
Theorem \ref{Xm isom MI2m+1}(ii) show that, for any point $x\in\mathcal{X}$ such that $A:=g_m(x)\in MI_{2m+1}(\xi)$,
\begin{equation}\label{dim X' ge}
4m^2+20m+6=(2m+1)^2+8(2m+1)-3\le\dim_A MI_{2m+1}(\xi)=\dim\overline{\mathcal{X}}.
\end{equation}
Comparing (\ref{dim X' le}) with (\ref{dim X' ge}) we see that all inequalities in
(\ref{dim fibre le dim Zm+dim Psim})-(\ref{dim X' ge}) are equalities. In particular,
\begin{equation}\label{dim fibre=dim X'-dim base}
\dim\rho^{-1}(0,0)=\dim(Z_m\times\mathbf{\Psi}_m)=\dim\overline{\mathcal{X}}-\dim(\mathbf{L}_m\times\mathbf{M}_m).
\end{equation}
Since by Theorem \ref{Irreducibility of Zm} the scheme $Z_m$ is integral and so $Z_m\times\mathbf{\Psi}_m$ is integral
as well, (\ref{zero fibre subset Zm times Psim}) and (\ref{dim fibre=dim X'-dim base}) yield isomorphisms of
integral schemes
\begin{equation}\label{zero fibre = Zm times Psim}
\rho^{-1}(0,0)\overset{\rm schemes}=p_m^{-1}(0,0)\overset{\rm schemes}=
\tilde{p}_m^{-1}(0,0)\overset{\rm schemes}=Z_m\times\mathbf{\Psi}_m.
\end{equation}

Now we formulate the following Lemma, the proof of which we leave to the reader.
\begin{lemma}\label{flat implies irred}
Let $f:X\to Y$ be a morphism of reduced schemes, where $Y$ is a smooth integral scheme. Assume that there exists a
closed point $y\in Y$ such that for any irreducible component $X'$ of $X$ the following conditions are satisfied:

(a) $\dim f^{-1}(y)=\dim X'-\dim Y$,

(b) the scheme-theoretic embedding of fibres $(f|_{X'})^{-1}(y)\subset f^{-1}(y)$ is an isomorphism of integral
schemes.\\
Then

(i) there exists an open subset $U$ of\ $Y$ containing the point $y$ such that the morphism
$f|_{f^{-1}(U)}:f^{-1}(U)\to U$ is flat,

(ii) $X$ is integral and

(iii) $X$ is smooth at any smooth point of $f^{-1}(y)$.
\end{lemma}

Applying the assertions (i)-(ii) of this lemma to
$X=X_m,\ X'=\mathcal{X},\ Y=\mathbf{L}_m\times\mathbf{M}_m,\ y=(0,0),f=p_m$, and using
(\ref{dim fibre=dim X'-dim base}) and (\ref{zero fibre = Zm times Psim}), we obtain that
$X_m$ is an integral scheme of dimension $4m^2+20m+6$.

It follows now from Corollary \ref{density} and  Theorem \ref{Xm isom MI2m+1}
that $(MI_{2m+1})_{red}$ is irreducible of dimension $4m^2+20m+6=n^2+8n-3$ for $n=2m+1$, i.e.
the inequality (\ref{dim MIn ge...}) becomes the strict equality. This together with Theorem
\ref{rank condns} implies that $MI_{2m+1}$ is a locally complete intersection subscheme of
the vector space $\mathbf{S}_{2m+1}$.
We use now the following easy lemma, the proof of which is left to the reader.
\begin{lemma}\label{X integral}
Let $\mathcal{X}$ be an irreducible locally complete intersection subscheme of a smooth integral scheme $\mathcal{Y}$
such that $\mathcal{X}$ is smooth at some point. Then $\mathcal{X}$ is integral.
\end{lemma}

Applying this Lemma to $\mathcal{X}=MI_{2m+1},\ \mathcal{Y}=\mathbf{S}_{2m+1}$ and using Remark \ref{smooth pts}, we
obtain that $MI_{2m+1}$ is integral.
Since $\pi_{2m+1}: MI_{2m+1}\to I_{2m+1}:\ A\mapsto[E(A)]$ is a principal
$GL(H_{2m+1})/\{\pm id\}$-bundle in the \'etale topology (see section \ref{general}),
it follows that $I_{2m+1}$ is integral of dimension $16m+5=8n-3$ for $n=2m+1$.
This finishs the proof of Theorem \ref{Irreducibility}.

\begin{remark}\label{p dominant}
Consider the natural projections
$p_I:X_m\to\mathbf{L}_m\times\mathbf{M}_m\times\mathbf{\Psi}_m$,
$p_{II}:X_m\to\mathbf{S}_m\times\mathbf{L}_m\times\mathbf{M}_m\times\mathbf{\Psi}_m\simeq
\mathbf{S}_{m+1}\times\mathbf{\Psi}_m$
and
$p:X_m\overset{p_{II}}\to\mathbf{S}_{m+1}\times\mathbf{\Psi}_m\overset{pr_1}\to\mathbf{S}_{m+1}$.
From (\ref{zero fibre = Zm times Psim}) it follows that
$p_I^{-1}(0,0,0)\simeq Z_m$.
On the other hand, Theorem \ref{Irreducibility of Zm} shows that the projection
$p':Z_m\overset{p_m}\to(\mathbf{S}_m^\vee)^0\simeq\mathbf{S}_m^0\overset{\rm open}\hookrightarrow\mathbf{S}_m$
is dominant, hence, for a general point $D_1\in\mathbf{S}_m$,
the fibre $p'^{-1}(D_1)$ is an integral scheme of dimension
$\dim Z_m-\dim\mathbf{S}_m=m(m+5)$. This fibre in view of the equality $p_I^{-1}(0,0,0)\simeq Z_m$ coincides
with the fibre $p_{II}^{-1}(D_1,0,0,0)$, and we thus have
$\dim p_{II}^{-1}(D_1,0,0,0)=5m(m+1)=4m^2+20m+6-(3(m+1)(m+2)/2+6m)=\dim X_m-\dim(\mathbf{S}_m\times\mathbf{\Psi_m})$.
Thus, applying Lemma \ref{flat implies irred} to
$X=X'=X_m,\ Y=\mathbf{L}_m\times\mathbf{M}_m,\ y=(D_1,0,0,0),f=p_{II}$,
we obtain that $p_{II}$ is a dominant morphism. A fortiori,
$$
p:X_m\to\mathbf{S}_{m+1}:\ (D,\phi)\to D
$$
is a dominant morphism.

\end{remark}

\end{sub}

\vspace{0.5cm}

\section{Study of $Z_m$. Beginning of the proof of Theorem \ref{Irreducibility of Zm}}\label{Study Zm}

\vspace{0.5cm}

In this section we begin proving Theorem \ref{Irreducibility of Zm} on the irreducibility of $Z_m$.
In subsection \ref{explicit eqns} we first treat the case $m=1$. Next, we obtain explicit
equations of $Z_m$ under a fixed decompomposition of $H_m$ into a direct sum of $H_{m-1}$ and $\mathbf{k}$.
In subsection \ref{induction step} we formulate the main result of this section - Proposition \ref{part of inductn step}
- which is a part of the induction step in the proof of Theorem \ref{Irreducibility of Zm}. (The rest of the proof of
Theorem \ref{Irreducibility of Zm} will be given in the last subsection of Section \ref{gen pos}.) In subsections
\ref{m odd 1}-\ref{m even} we study in detail the explicit equations of $Z_m$ and as a result obtain the proof of
Proposition \ref{part of inductn step}.

\begin{sub} {\bf Explicit equations of $Z_m$ in
$(\mathbf{S}_m^\vee)^0\times\mathbf{\Phi}_m$.}\label{explicit eqns}
\rm
We proceed to the proof of the irreducibility of $Z_m$ by increasing induction on $m$.
For $m=1$ clearly $\mathbf{\Lambda}_m={0}$, so that the equations
$\{\Theta_1(D_1,\phi_1)\in\mathbf{S}_1\}$ of $Z_1$ in
$(\wedge^2(\mathbf{k}^\vee\otimes V^\vee))^0$ are empty,
i.e. scheme-theoretically we have
$$
Z_1=(\wedge^2(\mathbf{k}^\vee\otimes V^\vee))^0\times\mathbf{\Phi}_1\overset{open}\hookrightarrow\mathbb{A}^{12}.
$$
Thus $Z_1$ is integral as a dense open subset of $\mathbb{A}^{12}$.

Now fix an isomorphism
\begin{equation}\label{standard2}
H_{m-1}\oplus\mathbf{k}\overset{\sim}\to H_m:
((a_1,...,a_{m-1}),a_m)\mapsto(a_1,...,a_m).
\end{equation}
Under this isomorphism any homomorphism
\begin{equation}\label{phim}
\phi:H_m\otimes  V\to H_m^\vee\otimes V^\vee,\ \ \ \
\phi\in\mathbf{\Phi}_m=\Hom(H_m,H_m^\vee\otimes\wedge^2 V^\vee).
\end{equation}
can be represented as a homomorphism
\begin{equation}\label{decompn phim}
\phi:H_{m-1}\otimes  V\oplus\mathbf{k}\otimes  V\to
H_{m-1}^\vee\otimes V^\vee\oplus\mathbf{k}^\vee\otimes V^\vee,
\end{equation}
i.e. as a matrix of homomorphisms
\begin{equation}\label{matrix phim}
\phi= \left(
\begin{array}{cc}
\phi_{m-1} & \chi \\
\psi & \theta
\end{array}
\right),
\end{equation}
where
\begin{equation}\label{phi(m-1) etc}
\phi_{m-1}\in\mathbf{\Phi}_{m-1}=\Hom(H_{m-1},
H_{m-1}^\vee\otimes\wedge^2 V^\vee),\ \ \
\psi\in\mathbf{\Psi}_{m-1}:=
\Hom(H_{m-1},\mathbf{k}^\vee\otimes\wedge^2 V^\vee)
,\\
\end{equation}
$$
\chi\in\Hom(\mathbf{k},H_{m-1}^\vee\otimes\wedge^2 V^\vee)=\mathbf{\Psi}_{m-1},
\ \ \ \ \ \ \ \ \
\theta\in\mathbf{B}_\theta:=
\Hom(\mathbf{k},\mathbf{k}^\vee\otimes\wedge^2 V^\vee)=
\mathbf{S}_1.
$$
Respectively, a homomorphism
\begin{equation}\label{Am}
D\in \mathbf{S}_m^\vee\subset\Hom(H_m^\vee\otimes V^\vee,H_m\otimes V)
\end{equation}
can be represented as a matrix of homomorphisms
\begin{equation}\label{matrix Dm}
D= \left(
\begin{array}{cc}
D_{m-1} & a \\
-a^\vee & \alpha
\end{array}
\right),
\end{equation}
where
\begin{equation}\label{D(m-1) etc}
D_{m-1}\in\mathbf{S}^\vee_{m-1}\subset
\Hom(H_{m-1}^\vee\otimes V^\vee,\ H_{m-1}\otimes V),
\end{equation}
$$
a\in\Hom(\mathbf{k}^\vee,H_{m-1}\otimes\wedge^2 V)=
\mathbf{\Psi}_{m-1}^\vee,\ \ \ \ \ \ \ \ \
\alpha\in\mathbf{B}_\alpha:=\Hom(\mathbf{k}^\vee,\mathbf{k}\otimes\wedge^2 V).
$$
Note that the data (\ref{phi(m-1) etc}) and (\ref{D(m-1) etc}) yield isomorphisms
\begin{equation}\label{dir p S,Phi}
\mathbf{S}_m^\vee\overset{\simeq}\to\mathbf{B}_{\alpha}\times\mathbf{\Psi}_{m-1}^\vee\times\mathbf{S}_{m-1}^\vee,\ \ \
\mathbf{\Phi}_m\overset{\simeq}\to\mathbf{\Phi}_{m-1}\times\mathbf{\Psi}_{m-1}\times\mathbf{\Psi}_{m-1}
\times\mathbf{B}_{\theta},
\end{equation}
and hence an isomorphism
\begin{equation}\label{dir p}
\mathbf{S}_m^\vee\times\mathbf{\Phi}_m\overset{\simeq}\to\mathbf{B}_{\theta}\times\mathbf{B}_{\alpha}
\times\mathbf{\Psi}_{m-1}^\vee\times\mathbf{S}_{m-1}^\vee\times\mathbf{\Phi}_{m-1}
\times\mathbf{\Psi}_{m-1}\times\mathbf{\Psi}_{m-1}:
\end{equation}
$$
(D,\phi)\mapsto(\theta,\alpha,a,D_{m-1},\phi_{m-1},\psi,\chi).
$$
From (\ref{matrix phim}) and (\ref{matrix Dm}) it follows that the homomorphism
$$
\Theta(D,\phi):=\phi^\vee\circ D\circ \phi:H_m\otimes V\to
H_m^\vee\otimes V^\vee,\ \ \ \ \
\Theta(D,\phi)\in\wedge^2(H_m^\vee\otimes V^\vee),
$$
can be represented as a matrix of homomorphisms
\begin{equation}\label{matrix Thetam}
\Theta(D,\phi)= \left(
\begin{array}{cc}
\Theta_1(D,\phi) & b(D,\phi) \\
-b(D,\phi)^\vee & \beta(D,\phi)
\end{array}
\right),
\end{equation}
where
\begin{equation}\label{Theta(m-1)etc}
\Theta_1(D,\phi):=\phi_{m-1}^\vee\circ D_{m-1}\circ\phi_{m-1}+\phi_{m-1}^\vee\circ a\circ\psi
-\psi^\vee\circ a^\vee\circ\phi_{m-1}+\psi^\vee\circ\alpha\circ\psi\in
\end{equation}
$$
\in\wedge^2(H_{m-1}^\vee\otimes V^\vee)\subset
\Hom(H_{m-1}^\vee\otimes V^\vee,H_{m-1}\otimes  V),
$$
$$
b(D,\phi):=\phi_{m-1}^\vee\circ D_{m-1}\circ\chi+\phi_{m-1}^\vee\circ a\circ\theta
-\psi^\vee\circ a^\vee\circ\chi+\psi^\vee\circ\alpha\circ\theta\in
$$
$$
\in\Hom(H_{m-1}\otimes  V,\mathbf{k}^\vee\otimes V^\vee),
$$
$$
\beta(D,\phi):=\chi^\vee\circ D_{m-1}\circ\chi+\chi^\vee\circ a\circ\theta
-\theta^\vee\circ a^\vee\circ\chi+\theta^\vee\circ\alpha\circ\theta\in
\mathbf{B}_\theta.
$$
In these notations $Z_m$ can be described as
\begin{equation}\label{again Zm}
Z_m=\left\{(D,\phi)\in (\mathbf{S}_m^\vee)^0\times\mathbf{\Phi}_m\ \ \left|\
\begin{matrix}
(i)\ \Theta_1(D,\phi)\in\mathbf{S}_{m-1}, \cr
(ii)\ b(D,\phi)\in\mathbf{\Psi}_{m-1}\  \cr
\end{matrix}\right.
\right\}.
\end{equation}
(Note that the condition $\beta(D,\phi)\in\mathbf{S}_1$ here is empty.)

We thus have the following explicit equations of
$Z_m$ in the open subset $(\mathbf{S}_m^\vee)^0\times\mathbf{\Phi}_m$
of the variety
$\mathbf{S}_m^\vee\times\mathbf{\Phi}_m$,
where we consider
$\mathbf{S}_m^\vee\times\mathbf{\Phi}_m$ as the direct product
$
\mathbf{B}_{\theta}\times\mathbf{B}_{\alpha}
\times\mathbf{\Psi}_{m-1}^\vee\times\mathbf{S}_{m-1}^\vee\times\mathbf{\Phi}_{m-1}
\times\mathbf{\Psi}_{m-1}\times\mathbf{\Psi}_{m-1}$
via (\ref{dir p}):
\begin{equation}\label{eq1 of Zm}
\Theta_1(D,\phi):=\phi_{m-1}^\vee\circ D_{m-1}\circ\phi_{m-1}+\phi_{m-1}^\vee\circ a\circ\psi
-\psi^\vee\circ a^\vee\circ\phi_{m-1}+\psi^\vee\circ\alpha\circ\psi\in\mathbf{S}_{m-1},
\end{equation}
\begin{equation}\label{eq2 of Zm}
b(D,\phi):=\phi_{m-1}^\vee\circ D_{m-1}\circ\chi+\phi_{m-1}^\vee\circ a\circ\theta
-\psi^\vee\circ a^\vee\circ\chi+\psi^\vee\circ\alpha\circ\theta\in\mathbf{\Psi}_{m-1}.
\end{equation}
These equations will be used systematically in the next subsections.

\end{sub}

\begin{sub} {\bf Part of induction step in the proof of Theorem \ref{Irreducibility of Zm}.}\label{induction step}
\rm

We first introduce some more notation. Set
$$
(\wedge^2V)^0:=\{a\in\wedge^2V\ |\ a:V^\vee\to V\ {\rm is\ an\ isomorphism}\},
$$
$$
(\wedge^2V^\vee)^0:=\{a\in\wedge^2V^\vee\ |\ a:V\to V^\vee\ {\rm is\ an\ isomorphism}\}.
$$
Consider the projective space $P(\wedge^2V^\vee)$ together with the Grassmannian
$G=G(1,3)\subset P(\wedge^2V^\vee)$ embedded by Pl\"ucker. Take any two points $a\in(\wedge^2V)^0$ and
$b\in(\wedge^2V^\vee)^0$ such that the corresponding points $<a^{-1}>$ and $<b>$ in $P(\wedge^2V^\vee)$ are distinct.
The projective line $P^1(a,b):=\Span(<a^{-1}>,<b>)$ joining these points intersects the quadric $G$ in two points, say,
$\{y_1,y_2\}$, not necessarily distinct, and let $\mathbb{P}^1_{(i)}(a,b),\ i=1,2,$ be the two disjoint lines in
$\mathbb{P}^3$
corresponding to the points $y_1,y_2$. Set
\begin{equation}\label{L(a,b)}
L(a,b):=\mathbb{P}^1_{(1)}(a,b)\sqcup\mathbb{P}^1_{(2)}(a,b).
\end{equation}
Next, note that there are natural isomorphisms
$\mathbf{S}_1^\vee\simeq\wedge^2V$
and
$\mathbf{\Phi}_1^\vee\simeq\wedge^2V^\vee$,
and, for any $m\ge2$, the induced isomorphisms
\begin{equation}\label{US UPhi}
U_\mathbf{S}:=\underset{i=1}{\overset{m}{\oplus}}(\mathbf{S}_1^\vee)_{(i)}
\simeq\underset{1}{\overset{m}{\oplus}}\wedge^2V,
\ \ \ U_\mathbf{\Phi}:=\underset{i=1}{\overset{m}{\oplus}}(\mathbf{\Phi}_1)_{(i)}\simeq,
\underset{i=1}{\overset{m}{\oplus}}\wedge^2V^\vee,
\end{equation}
where $(\mathbf{S}_1^\vee)_{(i)}$ and $(\mathbf{\Phi}_1)_{(i)}$
are copies of
$\mathbf{S}_1^\vee$ and $\mathbf{\Phi}_1$,
respectively. Furthermore, any isomorphism
\begin{equation}\label{isom h}
h:\underbrace{H_1\oplus...\oplus H_1}_{m}\overset{\simeq}\to H_m
\end{equation}
induces embeddings
$U_\mathbf{S}\hookrightarrow\mathbf{S}_m^\vee$ and $U_\mathbf{\Phi}\hookrightarrow\mathbf{\Phi}_m$, hence an embedding
\begin{equation}\label{tau h}
\tau_h:U_\mathbf{S}\times U_\mathbf{\Phi}\hookrightarrow\mathbf{S}_m^\vee\times \mathbf{\Phi}_m.
\end{equation}
Note also that the set
\begin{equation}\label{W SPhi}
W_{\mathbf{S\Phi}}:=\{((D_{(1)},...,D_{(m)}),(\phi_{(1)},...,\phi_{(m)}))\in U_\mathbf{S}\times U_\mathbf{\Phi}\ |\
{\rm the\ subsets}\ L(D_{(i)},\phi_{(i)})\ {\rm of}\ \mathbb{P}^3,\
\end{equation}
$$
1\leq i\leq m,\ {\rm are\ well\ defined,\ pairwise\ disjoint\ and\ not\ lying\ on\ a\ quadric}\}
$$
is clearly a dense open subset of $U_\mathbf{S}\times U_\mathbf{\Phi}$.

The aim of the rest of this section is to prove the following proposition which is a part
of the induction step $m-1\rightsquigarrow m$ in the proof of Theorem \ref{Irreducibility of Zm}.

\begin{proposition}\label{part of inductn step}
Let $m\ge2$ and let $Z_{m-1}$ satisfy the statement of Theorem \ref{Irreducibility of Zm}. Then there
exists an irreducible component $Z$ of $Z_m$ such that:

(i) let $Z_m=Z\cup Y$ be the decomposition of $Z_m$ into components; then $Z^0:=Z\smallsetminus(Z\cap Y)$
is an integral locally complete intersection subscheme of $(\mathbf{S}_m^\vee)^0\times\mathbf{\Phi}_m$;

(ii) $\dim Z=4m(m+2)$ and the natural projection
$p_m|_{Z}:Z\to(\mathbf{S}_m^\vee)^0: \ (D,\phi)\mapsto D$
is dominant;

(iii) there exists an isomorphism $h$ in (\ref{isom h}) such that,
in the notations (\ref{tau h}) and (\ref{W SPhi}), $Z\cap\tau_h(W_{\mathbf{S\Phi}})\neq\emptyset$.
\end{proposition}

Before proving this proposition we need some preliminary remarks.

First, consider the case $m=2$. In this case $D_{m-1}=D_1\in\wedge^2V,\ \phi_{m-1}=\phi_1\in\wedge^2V^\vee$
and
$a,\alpha\in\wedge^2V,\ \psi,\chi,\theta\in\wedge^2V^\vee$
so that the equations (\ref{eq1 of Zm}) become empty, and the equations (\ref{eq2 of Zm}) become:
\begin{equation}\label{eq2 of Z2}
(\phi_1\circ D_1-\psi_1\circ a)\circ\chi-(\phi_1\circ a-\psi_1\circ\alpha)\circ\theta\in\wedge^2V^\vee.
\end{equation}
Now one can easily check that, for a general point
$x=(D_1,\phi_1,\psi,a,\alpha)\in(\wedge^2V)^0\times(\wedge^2V^\vee)^{\times4}$,
the equations (\ref{eq2 of Z2}) as a linear system on the pair
$(\chi,\theta))\in(\wedge^2V^\vee)^{\times2}$
has maximal rank equal 10. Thus the space $F_x$ of solutions of this system as a subspace of
$(\wedge^2V^\vee)^{\times2}$ has dimension 2. This means that there exists a component $Z$ of $Z_2$ with projection
$p_Z:Z\to(\wedge^2V)^0\times(\wedge^2V^\vee)^{\times4}:(D_1,\phi_1,\psi,a,\alpha,\chi,\theta)\mapsto
(D_1,\phi_1,\psi,a,\alpha)$
with a smooth fibre $F_x=p_Z^{-1}(x)$ of dimension 2. Hence, in particular, $Z$ is generically reduced and
$\dim Z\le\dim((\wedge^2V)^0\times(\wedge^2V^\vee)^{\times4})+2=32$.
On the other hand, since (\ref{eq2 of Z2}) is a system of 10 equations of $Z_2$ in
$(\mathbf{S}_2^\vee)^0\times\mathbf{\Phi}_2$,
it follows that $Z$ as irreducible component of $Z_2$ has dimension
$\ge\dim((\mathbf{S}_2^\vee)^0\times\mathbf{\Phi}_2)-10=42-10=32$. Hence
$\dim Z=32$ and $p_Z$ is dominant. As a corollary, the projection
$p_2|Z:Z\to(\mathbf{S}_2^\vee)^0:(D_1,\phi_1,\psi,a,\alpha,\chi,\theta)\mapsto(D_1,a,\alpha)$
is also dominant. Moreover, since $F_x$ is smooth and $p_Z(Z)$ is smooth as a dense open subset of
 $(\wedge^2V)^0\times(\wedge^2V^\vee)^{\times4}$, it follows that  $Z$ is generically reduced.
 Now we use the following remark.

\begin{remark}\label{remark on lci}
Let $\mathcal{\tilde{X}}$ be a locally closed subscheme of an affine space $\mathbb{A}^M$ defined locally by $N$
equations. Let $\mathcal{X}$ be an irreducible component of $\mathcal{\tilde{X}}$ and let $\mathcal{X}^0$ be a
complement in $\mathcal{X}$ of its intersection with the union of other possible components of $\mathcal{\tilde{X}}$.
Let $\mathcal{X}$ be generically reduced and let $\dim\mathcal{X}=M-N$. Then $\mathcal{X}^0$ is an integral locally
complete intersection subscheme of  $\mathbb{A}^M$.
\end{remark}

Applying this remark to the case
$\mathcal{X}=Z_2,\ \mathbb{A}^{42}=(\wedge^2V)^0\times(\wedge^2V^\vee)^{\times6}$,
we obtain from the above that the statements (i)-(ii) of Proposition
\ref{part of inductn step} are true for $Z$. Now an explicit computation shows that the statement (iii) of this
Proposition is also true for $Z$. We thus have proved Proposition \ref{part of inductn step} for $m=2$.

We proceed now to the proof of Proposition \ref{part of inductn step} for $m\ge3$.
For this, note that, by the assumption, $Z_{m-1}$ is an integral subscheme of
$(\mathbf{S}_{m-1}^\vee)^0\times\mathbf{\Phi}_{m-1}$ such that
$$
\dim Z_{m-1}=4(m^2-1)
$$
and the natural projection $p_{m-1}:Z_{m-1}\to(\mathbf{S}_{m-1}^\vee)^0: \ (D_{m-1},\phi_{m-1})\mapsto D_{m-1}$
is surjective:
\begin{equation}\label{p m-1 sur}
p_{m-1}(Z_{m-1})=(\mathbf{S}_{m-1}^\vee)^0.
\end{equation}
Hence, since $\dim(\mathbf{S}_{m-1}^\vee)^0=3m(m-1)$ and so $\dim Z_{m-1}-\dim(\mathbf{S}_{m-1}^\vee)^0=(m-1)(m+4)$,
it follows that the set
\begin{equation}\label{S00}
(\mathbf{S}_{m-1}^\vee)^{int}:=\{D_{m-1}\in(\mathbf{S}_{m-1}^\vee)^0\ |\ {\rm the\ fibre}\ p_m^{-1}(D_{m-1})\
{\rm is\ integral\ of\ dimension}\ (m-1)(m+4)\}
\end{equation}
is a dense open subset of $(\mathbf{S}_{m-1}^\vee)^0$; respectively,
\begin{equation}\label{Z00}
Z_{m-1}^{int}:=p_{m-1}^{-1}((\mathbf{S}_{m-1}^\vee)^{int})
\end{equation}
is a dense open subset of $Z_{m-1}$.

Next, using (\ref{dir p}) and the embedding
$Z_m\hookrightarrow\mathbf{S}_m^\vee\times\mathbf{\Phi}_m$
consider the projections
\begin{equation}\label{prn pim}
pr_m:\mathbf{S}_m^\vee\times\mathbf{\Phi}_m\to\mathbf{B}_{\theta}\times\mathbf{B}_{\alpha}\times
\mathbf{\Psi}_{m-1}^\vee\times\mathbf{S}_{m-1}^\vee:\ \
(D,\phi)=(\theta,\alpha,a,D_{m-1},\phi_{m-1},\psi,\chi)\mapsto
\end{equation}
$$
\mapsto(\theta,\alpha,a,D_{m-1}),\ \ \ \
\pi_m:=pr_m|_{Z_m}:Z_m\to\mathbf{B}_{\theta}\times\mathbf{B}_{\alpha}\times
\mathbf{\Psi}_{m-1}^\vee\times\mathbf{S}_{m-1}^\vee.
$$

We are going now to study the fibre
$$
\pi_m^{-1}(y^0)
$$
of the projection $\pi_m$ over the point
\begin{equation}\label{y0}
y^0:=(\theta^0,\alpha^0,0,D_{m-1})\in
\mathbf{B}_{\theta}\times\mathbf{B}_{\alpha}\times\mathbf{\Psi}_{m-1}^\vee\times(\mathbf{S}_{m-1}^\vee)^0,
\end{equation}
where
\begin{equation}\label{alpha0,theta0}
\alpha^0=(p_{ij})\in\wedge^2V^{\vee}\simeq\mathbf{B}_{\alpha},\ \ \
\theta^0=(q_{ij})\in\wedge^2V^{\vee}\simeq\mathbf{B}_{\theta},\ \ \ p_{ij},q_{ij}\in\mathbf{k}.
\end{equation}
\footnote{Here and below we use a fixed basis $e_1,...,e_4$ of $V$ in order to
understand points of $\wedge^2V$ and $\wedge^2V^{\vee}$ as skew $4\times4$- matrices.}

Note that, by the definition of $\pi_m$, the fibre $\pi_m^{-1}(y^0)$ naturally lies in
$\mathbf{\Phi}_{m-1}\times\mathbf{\Psi}_{m-1}\times\mathbf{\Psi}_{m-1}$:
\begin{equation}\label{fibre lies in}
\pi_m^{-1}(y^0)\subset\mathbf{\Phi}_{m-1}\times\mathbf{\Psi}_{m-1}\times\mathbf{\Psi}_{m-1}.
\end{equation}
Thus, substituting (\ref{y0}) into (\ref{eq1 of Zm})
and (\ref{eq2 of Zm}), we obtain the equations of $\pi_m^{-1}(y^0)$ as a subscheme
of
$\mathbf{\Phi}_{m-1}\times\mathbf{\Psi}_{m-1}\times\mathbf{\Psi}_{m-1}$ as equations in the variables
$\phi_{m-1},\chi$ and $\psi$:
\begin{equation}\label{eq1 of fibre}
\phi_{m-1}^\vee\circ D_{m-1}\circ\phi_{m-1}+\psi^\vee\circ\alpha^0\circ\psi\in\mathbf{S}_{m-1},
\end{equation}
\begin{equation}\label{eq2 of fibre}
\phi_{m-1}^\vee\circ D_{m-1}\circ\chi+\psi^\vee\circ\alpha^0\circ\theta^0\in\mathbf{\Psi}_{m-1}.
\end{equation}

For an arbitrary point $y^0$ in (\ref{y0}), where $D_{m-1}\in(\mathbf{S}_{m-1}^\vee)^0$, consider the set
\begin{equation}\label{def of F}
F(\theta^0,\alpha^0,D_{m-1}):=\pi_m^{-1}(y^0)\cap\{\chi=\psi=0\}.
\end{equation}
It follows from (\ref{eq1 of fibre}) that
\begin{equation}\label{F=}
F(\theta^0,\alpha^0,D_{m-1})\simeq\{\phi_{m-1}\in\mathbf{\Phi}_{m-1}\ |\ \phi_{m-1}^\vee\circ D_{m-1}\circ
\phi_{m-1}\in\mathbf{S}_{m-1}\}.
\end{equation}
Hence,
$\underset{D_{m-1}\in(\mathbf{S}_{m-1}^\vee)^0}\cup F(\theta^0,\alpha^0,D_{m-1})=
\{(\theta^0,\alpha^0)\}\times Z_{m-1}.$
Moreover, the definition (\ref{S00}) implies that for $D_{m-1}\in(\mathbf{S}_{m-1}^\vee)^{int}$ the set
$F(\theta^0,\alpha^0,D_{m-1})$ is irreducible of dimension $(m-1)(m+4)$ and, by (\ref{dir p}), (\ref{Z00}) and
(\ref{def of F}),
\begin{equation}\label{union of fibres}
\underset{D_{m-1}\in(\mathbf{S}_{m-1}^\vee)^{int}}\cup F(\theta^0,\alpha^0,D_{m-1})=
\{(\theta^0,\alpha^0,0)\}\times Z^{int}_{m-1}\times\{(0,0)\}.
\end{equation}

\end{sub}

\begin{sub}\label{m odd 1st computns}{\bf Proof of Proposition \ref{part of inductn step}: case $m$ odd, first
computations.}\label{m odd 1}
\rm
In this subsection we prove Proposition \ref{part of inductn step} for the case of odd $m$,
$$
m=2p+1,\ \ \ p\ge1.
$$

Fix decompositions
\begin{equation}\label{2 decompns}
H_{m-1}\simeq \underbrace{H_2\oplus...\oplus H_2}_{p},\ \ \ H_2\simeq H_1\oplus H_1.
\end{equation}
Under these decompositions consider the points
$D^\Delta_{m-1}\in(\mathbf{S}^\vee_{m-1})^0$ and $\phi^\Delta_{m-1}\in\mathbf{\Phi}_{m-1}$
given by the matrices
\footnote{Here and everywhere below the empty entries of matrices mean zeroes. Besides, we use the standard notation
$A=A_1\oplus...\oplus A_n$ for a direct sum $A$ of matrices $A_1,...,A_n$ which is a block matrix with diagonal blocks
$A_1,...,A_n$ and the zero rest blocks. }
\begin{equation}\label{matrix D0,phi0}
D^\Delta_{m-1}:= \underbrace{D_2\oplus...\oplus D_2}_{p},\ \ \
\phi^\Delta_{m-1}=\phi^\Delta_{m-1}(N,a,d,f,g):=\underbrace{\phi_2\oplus...\oplus\phi_2}_{p},
\end{equation}
where
\begin{equation}\label{matrix D2}
D_2= D'\oplus D''\in\mathbf{S}^\vee_2,\ \
D'= \left(
\begin{array}{cccc}
 & -1 &  & \\
1 & &  & \\
 &  &  & 1\\
 &  & -1 &
\end{array}
\right)\in\wedge^2V,\ \ \
D''= \left(
\begin{array}{cccc}
 &  & 1 & \\
 &  &  & 1\\
-1 &  & & \\
 & -1 &  &
\end{array}
\right)\in\wedge^2V,
\end{equation}
\begin{equation}\label{matrix phi2}
\phi_2= \left(
\begin{array}{cc}
\phi_{11}&\phi_{12}\\
\phi_{21}&\phi_{22}
\end{array}
\right)\in\mathbf{\Phi}_2,\ \ \ \
\phi_{11}= \left(
\begin{array}{cccc}
 & -1 &  & \\
1 &  &  & \\
 &  &  & N\\
 &  & -N &
\end{array}
\right),\ \ \
\phi_{22}= \left(
\begin{array}{cccc}
 &  & 1 & \\
 &  &  & N\\
-1 &  & & \\
 & -N &  &
\end{array}
\right),\ \ \ N\in\mathbf{k},
\end{equation}
$$
\phi_{12}= \left(
\begin{array}{cccc}
 &  &  & f \\
 &  & g & \\
 & -g &  & \\
-f &  &  &
\end{array}
\right),\ \ \
\phi_{21}= \left(
\begin{array}{cccc}
 &  & a & f\\
 &  & -g & d\\
-a & g & & \\
-f & -d &  &
\end{array}
\right)\in\wedge^2V^{\vee},
\ \ \ a,d,f,g\in\mathbf{k}.
$$
One easily checks that
\begin{equation}\label{pt in Sm-1}
(\phi^\Delta_{m-1})^\vee\circ D^\Delta_{m-1}\circ\phi^\Delta_{m-1}\in\mathbf{S}_{m-1},
\end{equation}
hence the point
$(D^\Delta_{m-1},\phi^\Delta_{m-1})\in\mathbf{S}_{m-1}^\vee\times\mathbf{\Phi}_{m-1}$
lies in $\widehat{Z}_{m-1}$. Moreover, since $D^\Delta_{m-1}\in(\mathbf{S}^\vee_{m-1})^0$, it follows that
\begin{equation}\label{pt in Zm-1}
(D^\Delta_{m-1},\phi^\Delta_{m-1})\in Z_{m-1}.
\end{equation}
In addition, it follows from (\ref{pt in Sm-1}) that the equations (\ref{eq1 of fibre}) are automatically satisfied for
any $\psi\in\Psi_{m-1}$.
Now, substituting the data $(\theta^0,\alpha^0,D^\Delta_{m-1},\phi^\Delta_{m-1})$ into (\ref{eq2 of fibre}), we obtain
the equations on $(\chi,\psi)$:
\begin{equation}\label{eq2a of fibre}
(\phi^\Delta_{m-1})^\vee\circ D^\Delta_{m-1}\circ\chi+\psi^\vee\circ\alpha^0\circ\theta^0\in\mathbf{\Psi}_{m-1}.
\end{equation}
Set
\begin{equation}\label{W()}
W(\theta^0,\alpha^0,D^\Delta_{m-1},\phi^\Delta_{m-1}):=\{(\chi,\psi)\in\mathbf{\Psi}_{m-1}\times\mathbf{\Psi}_{m-1}\
|\ (\chi,\psi)
\ {\rm satisfies}\ (\ref{eq2a of fibre})\}.
\end{equation}
Note that, since the equations (\ref{eq2a of fibre}) on $(\chi,\psi)$ are linear, it follows that
$W(D^\Delta_{m-1},\phi^\Delta_{m-1},\alpha^0,\theta^0)$ is a linear subspace of the vector space
$\mathbf{\Psi}_{m-1}\times\mathbf{\Psi}_{m-1}\simeq\mathbf{\Psi}_{m-1}^\vee\oplus\mathbf{\Psi}_{m-1}$.

Find the dimension of the vector space $W(\theta^0,\alpha^0,D^\Delta_{m-1},\phi^\Delta_{m-1})$. For this, using the
decompositions (\ref{2 decompns}) we represent
$\chi$ and $\psi$ as $p$-ples
\begin{equation}\label{chi-psi}
\chi=(\chi_1,...,\chi_p),\ \ \ \psi=(\psi_1,...,\psi_p),\ \ \ \psi_k,\chi_k\in\mathbf{\Psi}_2,\ k=1,...,p,
\end{equation}
where
\begin{equation}\label{Xk,Yk}
\chi_k=(X_k,Y_k),\ \ \ \psi_k=(A_k,B_k),\ \ \ X_k,Y_k,A_k,B_k\in\wedge^2V^\vee,
\end{equation}
and
\begin{equation}\label{Ak,Bk}
X_k=(x_{ij}^{(k)}),\ Y_k=(y_{ij}^{(k)}),\ A_k=(a_{ij}^{(k)}),\ B_k=(b_{ij}^{(k)})
\end{equation}
are skew-symmetric 4$\times$4-matrices. Inserting $D^\Delta_{m-1}$ and $\phi^\Delta_{m-1}$ from (\ref{matrix D0,phi0}) into the system of
equations (\ref{eq2a of fibre}) we rewrite this system as
\begin{equation}\label{set of systems1}
\phi_2^\vee\circ D_2\circ\chi_k+\psi_k^\vee\circ\alpha^0\circ\theta^0\in\mathbf{\Psi}_2,\ \ \ k=1,...,p.
\end{equation}
Substituting here $D_2,\phi_2$ and $\theta^0$ from (\ref{matrix D2}), (\ref{matrix phi2}) and (\ref{alpha0,theta0}) and denoting
$x_1^{(k)}=x_{12}^{(k)},x_2^{(k)}=x_{34}^{(k)},x_3^{(k)}=x_{13}^{(k)},x_4^{(k)}=x_{14}^{(k)},x_5^{(k)}=x_{23}^{(k)},
x_6^{(k)}=x_{24}^{(k)},x_7^{(k)}=y_{12}^{(k)},x_8^{(k)}=y_{34}^{(k)},x_9^{(k)}=y_{13}^{(k)},x_{10}^{(k)}=y_{14}^{(k)},
x_{11}^{(k)}=y_{23}^{(k)},x_{12}^{(k)}=y_{24}^{(k)},x_{13}^{(k)}=a_{12}^{(k)},x_{14}^{(k)}=a_{34}^{(k)},x_{15}^{(k)}=a_{13}^{(k)},
x_{16}^{(k)}=x_{14}^{(k)},x_{17}^{(k)}=x_{23}^{(k)},x_{18}^{(k)}=x_{24}^{(k)},x_{19}^{(k)}=b_{12}^{(k)},x_{20}^{(k)}=b_{34}^{(k)},
x_{21}^{(k)}=b_{13}^{(k)},x_{22}^{(k)}=b_{14}^{(k)},x_{23}^{(k)}=b_{23}^{(k)},x_{24}^{(k)}=b_{24}^{(k)},$ we rewrite the system
(\ref{set of systems1}) as
\begin{equation}\label{set of systems2}
\sum_{j=1}^{24}m_{ij}x_j^{(k)}=0,\ \ \ i=1,...,20,\ \ \ k=1,...,p,
\end{equation}
where $M:=(m_{ij})$ is the $20\times24$-matrix with entries depending on $N,a,d,f,g,p_{ij},q_{ij}$.

Now a direct computation of the matrix $\mathbf{M}=(m_{ij})$ for
\begin{equation}\label{N etc}
N=101,\ a=4,\ d=6,\ f=2,\ g=5,
\end{equation}
\begin{equation}\label{pij,qij}
p_{12}=-9,\ p_{13}=-2,\ p_{14}=-4,\ p_{23}=6,\ p_{24}=-3,\ p_{34}=-7,
\end{equation}
$$
q_{12}=-4,q_{13}=-4,q_{14}=-2,q_{23}=3,q_{24}=-7,q_{34}=8,
$$
shows that $\mathbf{M}$ is the upper left block submatrix
\begin{equation}\label{matrix bfM}
\mathbf{M}=\left(
\begin{array}{cccc}
\mathbf{M}_{11} & \mathbf{M}_{12} & \mathbf{M}_{\psi} & \mathbf{0}\\
\mathbf{M}_{21} & \mathbf{M}_{22} & \mathbf{0} & \mathbf{M}_{\psi}
\end{array}
\right)
\end{equation}
of the block matrix $\mathbf{\widetilde{M}}$ given below in (\ref{matrix tilde bfM})-(\ref{5blocks of tilde bfM}).
From (\ref{matrix bfM}) and (\ref{1blocks of tilde bfM})-(\ref{5blocks of tilde bfM}) it follows by an explicit
computation that
\begin{equation}\label{rk M=20}
\rk\mathbf{M}=20.
\end{equation}
Hence, since the matrix of the system (\ref{set of systems2}) is a direct sum of
$p$ copies of matrix
$\mathbf{M}$, it follows that its rank equals
\begin{equation}\label{p rk}
p\cdot\rk\mathbf{M}=20p=10(m-1).
\end{equation}

Next, denote by
\begin{equation}\label{bold phi etc}
\boldsymbol{\phi}_{m-1},\ \ \ \text{resp.,}\ \ \ \boldsymbol{\alpha},\ \ \boldsymbol{\theta}
\end{equation}
the matrices obtained by inserting the entries (\ref{N etc}) into the matrix $\phi_{m-1}^\Delta$ in
(\ref{matrix D0,phi0}), respectively, the entries (\ref{pij,qij}) into the matrices $\alpha^0$ and $\theta^0$ in
(\ref{alpha0,theta0}). In this notation, denoting by $R(\theta^0,\alpha^0,D_{m-1}^\Delta,\phi_{m-1}^\Delta)$ the
rank of the linear system (\ref{eq2a of fibre}) as a function of $\theta^0,\alpha^0,D_{m-1}^\Delta,\phi_{m-1}^\Delta$
we rewrite (\ref{p rk}) as
\begin{equation}\label{R=}
R(\boldsymbol{\theta},\boldsymbol{\alpha},D_{m-1}^\Delta,\boldsymbol{\phi}_{m-1})=10(m-1).
\end{equation}
Note that $(D_{m-1}^\Delta,\boldsymbol{\phi}_{m-1})\in Z_{m-1}$ by (\ref{pt in Zm-1}), and by (\ref{Z00})
$Z_{m-1}^{int}$ is irreducible and dense open in $Z_{m-1}$. In addition, since the maximal value of
$R(\theta^0,\alpha^0,D_{m-1},\phi_{m-1})$ equals $10(m-1)$,  the condition
$R(\theta^0,\alpha^0,D_{m-1},\phi_{m-1})=10(m-1)$ imposed on the point $(D_{m-1},\phi_{m-1})\in Z_{m-1}$ is open.
Hence it follows from (\ref{R=}) that

a) the set $(Z_{m-1}^{int})^0:=\{(D_{m-1},\phi_{m-1})\in Z_{m-1}^{int}\ |\ R(\boldsymbol{\theta},\boldsymbol{\alpha},D_{m-1},
\phi_{m-1})=10(m-1)\}$
is dense open in $Z_{m-1}^{int}$, hence also in $Z_{m-1}$.
By (\ref{p m-1 sur}) this implies that

b) there exists a dense
open subset $(\mathbf{S}_{m-1}^\vee)^*$ of $(\mathbf{S}_{m-1}^\vee)^{int}$ such that, for
$D_{m-1}\in(\mathbf{S}_{m-1}^\vee)^*$, the set
$$
F(\boldsymbol{\theta},\boldsymbol{\alpha},D_{m-1})^0:=
F(\boldsymbol{\theta},\boldsymbol{\alpha},D_{m-1})\cap(Z_{m-1}^{int})^0
$$
where $F(\theta^0,\alpha^0,D_{m-1})$ is defined in (\ref{def of F}),
is an integral scheme of dimension $(m-1)(m+4)$ and it is a dense open subset of
$F(\boldsymbol{\theta},\boldsymbol{\alpha},D_{m-1})$.

Now for $D_{m-1}\in(\mathbf{S}_{m-1}^\vee)^*$ set
$$
\mathbf{F}:=\pi_m^{-1}(\boldsymbol{\theta},\boldsymbol{\alpha},0,D_{m-1}),\ \ \
F=F(D_{m-1}):=F(\boldsymbol{\theta},\boldsymbol{\alpha},D_{m-1})=\mathbf{F}\cap\{\chi=\psi=0\}.
$$
From a) and b) it follows similar to (\ref{union of fibres}) that
$\underset{D_{m-1}\in(\mathbf{S}_{m-1}^\vee)^*}\cup F(D_{m-1})$ is dense open in
$\{(\boldsymbol{\theta},\boldsymbol{\alpha})\}\times Z^{int}_{m-1}\times\{(0,0)\},$ hence
\begin{equation}\label{closure union of fibres}
\overline{\underset{D_{m-1}\in(\mathbf{S}_{m-1}^\vee)^*}\cup F(D_{m-1})}=
\{(\boldsymbol{\theta},\boldsymbol{\alpha},0)\}\times \widehat{Z}_{m-1}\times\{(0,0)\},
\end{equation}
where the closure is taken in $\mathbf{S}_m^\vee\times\mathbf{\Phi}_m$ and we use the isomorphism (\ref{dir p}).

Take an arbitrary point $D_{m-1}\in(\mathbf{S}_{m-1}^\vee)^*$.
By
b) $F=F(D_{m-1})$ is integral of dimension $(m-1)(m+4)$ and contains a dense open subset $F^0$ such that, for any
point $w=(\boldsymbol{\theta},\boldsymbol{\alpha},D_{m-1},\phi'_{m-1})\in F^0$, one has
$R(w):=R(\boldsymbol{\theta},\boldsymbol{\alpha},D_{m-1},\phi'_{m-1})=10(m-1)$. Fix such a point $w$ which is smooth on $F$.
We are going now to compute the dimension of the tangent space $T_w\mathbf{F}$.

Note that by (\ref{fibre lies in}) we consider $\mathbf{F}$ as lying in
$\mathbf{\Phi}_{m-1}\times\mathbf{\Psi}_{m-1}\times\mathbf{\Psi}_{m-1}$. Hence the equations of the tangent space
$$
T_w\mathbf{F}
$$
are given by differentiating at $w$ the equations (\ref{eq1 of fibre}) and (\ref{eq2 of fibre}):
\begin{equation}\label{diff eq1}
d\phi_{m-1}^\vee|_{\phi'_{m-1}}\circ D_{m-1}\circ\phi_{m-1}+
{\phi'}_{m-1}^\vee\circ D_{m-1}\circ d\phi_{m-1}|_{\phi'_{m-1}}\in\mathbf{S}_{m-1},
\end{equation}
\begin{equation}\label{diff eq2}
{\phi'}_{m-1}^\vee\circ D_{m-1}\circ d\chi|_0+d\psi|_0^\vee\circ\alpha^0\circ\theta^0\in\mathbf{\Psi}_{m-1}.
\end{equation}
Here the equations (\ref{diff eq1}) coincide with the equations obtained by differentiating at $w$ the equations
$\phi_{m-1}^\vee\circ D_{m-1}\circ\phi_{m-1}\in\mathbf{S}_{m-1}$ defining $F$ as a subscheme of $\mathbf{\Phi}_{m-1}$.
Since $w$ is a smooth point of $F^0$, it follows that the equations (\ref{diff eq1}) define the tangent space
$T_wF^0=T_wF$ as a subspace of $T_{\phi'_{m-1}}\mathbf{\Phi}_{m-1}$ and
\begin{equation}\label{dim TwF}
\dim T_wF=\dim F=(m-1)(m-4).
\end{equation}
On the other hand, the equations (\ref{diff eq2}) just coincide with (\ref{eq2 of fibre}) via identifying
$(\chi|_0,d\psi|_0)$ with $(\chi,\psi)$, i.e. they are the equations of the subspace
$W(w)=W(\boldsymbol{\theta},\boldsymbol{\alpha},D_{m-1},\phi'_{m-1})$ in $\Psi_{m-1}\oplus\Psi_{m-1}$. Hence
$\dim W(w)=\dim(\Psi_{m-1}\oplus\Psi_{m-1})-R(w)=12(m-1)-10(m-1)=2(m-1)$. In view of (\ref{dim TwF}) we have
\begin{equation}\label{dim leq}
\dim_w\mathbf{F}\leq\dim T_w\mathbf{F}=\dim T_wF+\dim W(w)=(m-1)(m+4)+2(m-1)=m^2+5m-6.
\end{equation}
Note that, since $D_{m-1}\in(\mathbf{S}^\vee_{m-1})^0$ and $\boldsymbol{\alpha}\in\mathbf{S}_1^0$
(see (\ref{alpha0,theta0})), it follows that
$D=D_{m-1}\oplus\boldsymbol{\alpha}\in(\mathbf{S}^\vee_m)^0$,
so that
\begin{equation}\label{w in Zm}
w\in Z_m.
\end{equation}
In addition,
$\dim(\mathbf{B}_{\theta}\times\mathbf{B}_{\alpha}\times\mathbf{\Psi}_{m-1}^\vee\times\mathbf{S}_{m-1}^\vee)=
\dim(\mathbf{B}_{\theta}\times\mathbf{S}_m^\vee)=6+3m(m+1)=3m^2+3m+6$.
Counting the dimension of the fibres of
$\pi_m:Z_m\to\mathbf{B}_{\theta}\times\mathbf{B}_{\alpha}
\times\mathbf{\Psi}_{m-1}^\vee\times\mathbf{S}_{m-1}^\vee\simeq
\mathbf{B}_{\theta}\times\mathbf{S}_m^\vee$ and using (\ref{dim leq}) we obtain
$$
\dim_wZ_m\leq\dim_w\mathbf{F}+\dim(\mathbf{B}_{\theta}\times\mathbf{S}_m^\vee)\le(m^2+5m-6)+(3m^2+3m+6)=4m(m+2).
$$
Comparing this with (\ref{dim Zm}) we see that the above inequalities on dimensions are strict equalities.
In particular, $\dim_wZ_m=4m(m+2)$ and $\dim_w\mathbf{F}=\dim T_w\mathbf{F}=m^2+5m-6$ and
$\dim\pi_m(Z_m)=(3m^2+3m+6)=\dim(\mathbf{B}_{\theta}\times\mathbf{S}_m^\vee)$.
This together with the assertion (iii) of Lemma \ref{flat implies irred} implies that there exists a unique
irreducible component, say, $Z$ of $Z_m$ passing through $w$ such that:

(i) $\dim Z=4m(m+2)$ and $Z_m$, respectively, $Z$ is smooth at $w$; hence, in notations of Proposition
\ref{part of inductn step}(i), $Z^0$ is an integral locally complete intersection subscheme of
$(\mathbf{S}_m^\vee)^0\times\mathbf{\Phi}_m$ (we use here Remark \ref{remark on lci});

(ii) $\pi_m(Z)$ is dense in
$\mathbf{B}_{\theta}\times\mathbf{S}_m^\vee$;
respectively,
$p_m(Z)=pr_{\mathbf{S}}(\pi_m(Z))$ is dense in $\mathbf{S}_m^\vee$,
where
$pr_{\mathbf{S}}:\mathbf{B}_{\theta}\times\mathbf{S}_m^\vee\to\mathbf{S}_m^\vee$
is the projection. This gives proof of the statements (i) and (ii) of Proposition \ref{part of inductn step}.

Moreover, by a) and b) above, $F=F(D_{m-1})\subset Z$ for $D_{m-1}\in(\mathbf{S}_{m-1}^\vee)^*,$ so that
(\ref{closure union of fibres}) implies the existence of an embedding
\begin{equation}\label{embed in Z^0m}
\{(\boldsymbol{\theta},\boldsymbol{\alpha},0)\}\times \widehat{Z}_{m-1}\times\{(0,0)\}\subset\overline{Z},
\end{equation}
where $\overline{Z}$ is the closure of $Z$ in $\mathbf{S}_m^\vee\times\mathbf{\Phi}_m$.
In particular, similar to (\ref{w in Zm}) we have in view of (\ref{pt in Zm-1}):
\begin{equation}\label{w0 in Z^0*m}
w^0:=(\boldsymbol{\theta},\boldsymbol{\alpha},0,D^\Delta_{m-1},\phi^\Delta_{m-1},0,0)\in Z.
\end{equation}

\end{sub}

\begin{sub}\label{m odd last computns}{\bf Proof of Proposition \ref{part of inductn step}: case $m$ odd, last
computations.}\label{m odd 2}
\rm
In this subsection we prove the last statement (iii) of Proposition \ref{part of inductn step} in case of odd $m$. For
this, consider the following modification of the data (\ref{matrix D0,phi0})-(\ref{matrix phi2}):
\begin{equation}\label{matrix D0,phi0 new}
D^\Delta_{m-1}(c,\mathbf{f},\mathbf{g}):=D_2(c,f_1,g_1)\oplus...\oplus D_2(c,f_p,g_p),
\end{equation}
$$
\phi^\Delta_{m-1}(\varepsilon,\mathbf{f},\mathbf{g}):=
\phi_2(\varepsilon,f_1,g_1)\oplus...\oplus\phi_2(\varepsilon,f_p,g_p),
$$
where
\begin{equation}\label{matrix D2 new}
D_2(c,f_i,g_i)= \left(
\begin{array}{cc}
D'(c,f_i,g_i)&\\
&D''
\end{array}
\right)\in\mathbf{S}^\vee_2,\ \ \ \
\end{equation}
$$
D'(c,f_i,g_i)= \left(
\begin{array}{cccc}
 & -1 &  & cg_i \\
1 & & cf_i & \\
 & -cf_i &  & 1\\
-cg_i &  & -1 &
\end{array}
\right),\ i=1,...,p,\ \ \
D''= \left(
\begin{array}{cccc}
 &  & 1 & \\
 &  &  & 1\\
-1 &  & & \\
 & -1 &  &
\end{array}
\right)\in\wedge^2V,
$$
\begin{equation}\label{matrix phi2 new}
\phi_2(\varepsilon,f_i,g_i)=\left(
\begin{array}{cc}
\phi_{11}&\phi_{12}(\varepsilon,f_i,g_i)\\
\phi_{21}(\varepsilon,f_i,g_i)&\phi_{22}
\end{array}
\right)\in\mathbf{\Phi}_2,\ \ \ \
\phi_{11}= \left(
\begin{array}{cccc}
 & -1 &  & \\
1 &  &  & \\
 &  &  & N\\
 &  & -N &
\end{array}
\right),\ \ \
\end{equation}
$$
\phi_{22}= \left(
\begin{array}{cccc}
 &  & 1 & \\
 &  &  & N\\
-1 &  & & \\
 & -N &  &
\end{array}
\right),\ \ \
\phi_{12}(\varepsilon,f_i,g_i)=\left(
\begin{array}{cccc}
 &  &  &\varepsilon f_i \\
 &  & \varepsilon g_i & \\
 & -\varepsilon g_i &  & \\
-\varepsilon f_i &  &  &
\end{array}
\right),\ \ \
$$
$$
\phi_{21}(\varepsilon,f_i,g_i)=\left(
\begin{array}{cccc}
 &  & \varepsilon a & \varepsilon f_i\\
 &  & -\varepsilon g_i & \varepsilon d\\
-\varepsilon a & \varepsilon g_i & & \\
-\varepsilon f_i & -\varepsilon d &  &
\end{array}
\right)\in\wedge^2V^{\vee},
\ \ \ c,\epsilon,N,a,d,f_i,g_i\in\mathbf{k},\ i=1,...,p,
$$
and where
$\mathbf{f}=(f_1,...,f_p),\mathbf{g}=(g_1,...,g_p)\in\mathbf{k}^p$.
One easily checks that
$(\phi^\Delta_{m-1}(\varepsilon))^\vee\circ D^\Delta_{m-1}(c,\mathbf{f},\mathbf{g})
\circ\phi^\Delta_{m-1}(\varepsilon)\in\mathbf{S}_{m-1}$,
hence the point
$$
(D^\Delta_{m-1}(c,\mathbf{f},\mathbf{g}),\phi^\Delta_{m-1}(\varepsilon,\mathbf{f},\mathbf{g}))
\in\mathbf{S}_{m-1}^\vee\times\mathbf{\Phi}_{m-1}
$$
lies in $\widetilde{Z}_{m-1}$. Moreover, since
$(D^\Delta_{m-1}(0,\mathbf{f},\mathbf{g})=D^\Delta_{m-1}\in(\mathbf{S}^\vee_{m-1})^0$ and
$(\mathbf{S}^\vee_{m-1})^0$ is open in $\mathbf{S}^\vee_{m-1}$,
it follows that, for any $\mathbf{f},\mathbf{g}\in\mathbf{k}^p$
there exists some dense open subset
$\mathcal{U}(\mathbf{f},\mathbf{g})$ of $\mathbf{k}$
such that
$D^\Delta_{m-1}(c,\mathbf{f},\mathbf{g})\in(\mathbf{S}^\vee_{m-1})^0,\ \ c\in\mathcal{U}(\mathbf{f},\mathbf{g})$.
Hence,
$(D^\Delta_{m-1}(c,\mathbf{f},\mathbf{g}),\phi^\Delta_{m-1}(\varepsilon,\mathbf{f},\mathbf{g}))\in Z_{m-1}$
for
$c\in\mathcal{U}(\mathbf{f},\mathbf{g})$,
so that, since
$Z_{m-1}$ is closed in $\mathbf{S}_{m-1}^\vee\times\mathbf{\Phi}_{m-1}$,
\begin{equation}\label{pt in Zm-1 new}
(D^\Delta_{m-1}(c,\mathbf{f},\mathbf{g}),\phi^\Delta_{m-1}(\varepsilon,\mathbf{f},\mathbf{g}))\in Z_{m-1},\ \ \
c,\varepsilon\in\mathbf{k},\ \ \ \mathbf{f},\mathbf{g}\in\mathbf{k}^p.
\end{equation}
In particular, take $c=1$ and $\varepsilon=0$ in (\ref{matrix D0,phi0 new})-(\ref{matrix phi2 new}). It follows
immediately that the point
$$
w(\mathbf{f},\mathbf{g},\theta^0,\alpha^0):=(D^\Delta_{m-1}(1,\mathbf{f},\mathbf{g})\oplus\alpha^0,
\phi^\Delta_{m-1}(0,\mathbf{f},\mathbf{g})\oplus\theta^0),\ \ \ (\theta^0,\alpha^0)\in\wedge^2V^\vee\times\wedge^2V,
$$
is the image of the point
$$
((D'(1,f_1,g_1),...,D'(1,f_p,g_p),\underbrace{D'',...,D''}_{p},\alpha^0),
(\underbrace{\phi_{11},...,\phi_{11}}_{p},\underbrace{\phi_{22},...,\phi_{22}}_{p},\theta^0))\in
U_\mathbf{S}\times U_\mathbf{\Phi}\
$$
under the embedding
$\tau_h:U_\mathbf{S}\times U_\mathbf{\Phi}\hookrightarrow\mathbf{S}_m^\vee\times \mathbf{\Phi}_m$
defined (up to a permutation of direct summands) as in (\ref{isom h})-(\ref{tau h}) via the isomorphism
\begin{equation}\label{again h}
h:\underbrace{H_1\oplus...\oplus H_1}_{m}\overset{\simeq}\to H_m,\ m=2p+1,
\end{equation}
determined by the decompositions (\ref{2 decompns}).

On the other hand, by (\ref{dir p}) and (\ref{pt in Zm-1 new}) we have
$w(\mathbf{f},\mathbf{g},\boldsymbol{\theta},\boldsymbol{\alpha})\in\{(\boldsymbol{\theta},\boldsymbol{\alpha},0)\}
\times\widehat{Z}_{m-1}\times\{(0,0)\}$,
so that, in view of (\ref{embed in Z^0m}),
$w(\mathbf{f},\mathbf{g},\boldsymbol{\theta},\boldsymbol{\alpha})\in\overline{Z}$.
Thus,
\begin{equation}\label{w(f,g) in Z0m}
w(\mathbf{f},\mathbf{g},\boldsymbol{\theta},\boldsymbol{\alpha})
\in Z\cap\tau_h(U_\mathbf{S}\times U_\mathbf{\Phi}),\ \ \ \mathbf{f},\mathbf{g}\in\mathbf{k}^p.
\end{equation}
Note that $D^\Delta_{m-1}(1,\mathbf{0},\mathbf{0})=D^\Delta_{m-1}$, hence it follows from the definition of
$w(\mathbf{f},\mathbf{g},\theta^0,\alpha^0)$
that the point
$w(\mathbf{0},\mathbf{0},\boldsymbol{\theta},\boldsymbol{\alpha})$ lies in $Z_m$ (cf. (\ref{w0 in Z^0*m})). Since the
condition
$w(\mathbf{f},\mathbf{g},\theta^0,\alpha^0)\in Z_m$
on the point
$(\mathbf{f},\mathbf{g},\theta^0,\alpha^0)\in\mathbf{k}^{2p}\times\wedge^2V^\vee\times\wedge^2V$
is open, we obtain from (\ref{w(f,g) in Z0m}) that there exists a
dense open subset $\mathcal{U}\in\mathbf{k}^{2p}\times\wedge^2V^\vee\times\wedge^2V$ such that
\begin{equation}\label{w(f,g) in Z0*m}
w(\mathbf{f},\mathbf{g},\theta^0,\alpha^0)\in Z_m\cap\tau_h(U_\mathbf{S}\times U_\mathbf{\Phi}),\ \ \
(\mathbf{f},\mathbf{g},\theta^0,\alpha^0)\in\mathcal{U}.
\end{equation}

Next, one easily sees that, for general $f_i,g_i\ne0$
the points
$D'(0,f_i,g_i),D'(1,f_i,g_i)$ lie in $(\wedge^2V)^0$ and, moreover, the projective plane
$\Span(<D'(0,f_i,g_i)^{-1}>,<D'(1,f_i,g_i)^{-1}>,<\phi_{11}>)$ in $P(\wedge^2V^\vee)$
intersects the Grassmannian $G=G(1,3)$ in a smooth conic. This immediately implies that, in the notation of
(\ref{L(a,b)}), for a general choice of $f_1,g_1,f_2,g_2\in\mathbf{k}$, the sets
$L(D'(1,f_1,g_1)^{-1},\phi_{11})$ and $L(D'(1,f_2,g_2)^{-1},\phi_{11})$ are well defined and disjont. In other words,
using the notation of (\ref{US UPhi}) and considering the projection onto the direct summand
$$
pr_{ij}:U_\mathbf{S}\times U_\mathbf{\Phi}\to
((\mathbf{S}_1^\vee)_{(i)}\oplus(\mathbf{S}_1^\vee)_{(j)})\times((\mathbf{\Phi}_1)_{(i)}\oplus(\mathbf{\Phi}_1)_{(j)})
\simeq(\mathbf{S}_1^\vee\oplus\mathbf{S}_1^\vee)\times(\mathbf{\Phi}_1\oplus\mathbf{\Phi}_1)
$$
for any $1\le i< j\le m$ and taking the dense open subset $W_{ij}$ of
$U_\mathbf{S}\times U_\mathbf{\Phi}$ defined as
$$
W_{ij}:=pr_{ij}^{-1}(\{((D_1,D_2),(\phi_1,\phi_2))\in(\mathbf{S}_1^\vee\oplus\mathbf{S}_1^\vee)\times(\mathbf{\Phi}_1
\oplus\mathbf{\Phi}_1)\ |\ {\rm the\ subsets}\ L(D_1^{-1},\phi_1)\ {\rm and}\ L(D_2^{-1},\phi_2)\
$$
$$
{\rm of}\ \mathbb{P}^3{\rm are\ well\ defined\ and\ disjoint}\})
$$
{\rm are\ well\ defined,\ pairwise\ disjoint}
we obtain in view of (\ref{w(f,g) in Z0*m}) that
\begin{equation}\label{not empty}
Z\cap\tau_h(W_{12})\ne\emptyset.
\end{equation}
Now since the set ${\rm Isom}_m$ of all isomorphisms $h$ in (\ref{again h}) is a principal homogeneous space of the
group $GL(H_m)$ which is connected, it follows from (\ref{not empty}) that $Z_m\cap\tau_h(W_{ij})\ne\emptyset$
for a general $h\in{\rm Isom}_m$ and any pair
$(i,j),\ 1\le i< j\le m$. Hence, since
$W_{\mathbf{S\Phi}}=\underset{1\le i< j\le m}{\cap}W_{ij}$ by the definition (\ref{W SPhi}) of
$W_{\mathbf{S\Phi}}$, we deduce that $Z\cap\tau_h(W_{\mathbf{S\Phi}})\neq\emptyset$.
This finishes the proof of Proposition \ref{part of inductn step} for $m$ odd.

\end{sub}

\begin{sub}\label{inductn step even case}{\bf \bf Proof of Proposition \ref{part of inductn step}: case $m$
even.}\label{m even}
\rm

The proof of Proposition \ref{part of inductn step} for the case of even $m$,
$$
m=2p+4,\ \ \ p\ge0.
\footnote{Note that we start with $m=4$ since the case $m=2$ has been already treated in subsection
\ref{induction step}.}
$$
is completely parallel to that given above for the case of odd $m$. Namely, similar to (\ref{2 decompns}) fix the
decompositions
\begin{equation}\label{3 decompns}
H_{m-1}\simeq H_3\oplus\underbrace{H_2\oplus...\oplus H_2}_{p},\ \ \ H_2\simeq H_1\oplus H_1,\ \ \
H_3\simeq H_1\oplus H_1\oplus H_1.
\end{equation}
Under these decompositions, similar to (\ref{matrix D0,phi0}) consider the points
$D^\Delta_{m-1}\in(\mathbf{S}^\vee_{m-1})^0$ and $\phi^\Delta_{m-1}\in\mathbf{\Phi}_{m-1}$
given by the matrices with diagonal blocks
\begin{equation}\label{2matrix D0,phi0}
D^\Delta_{m-1}:=D_3\oplus \underbrace{D_2\oplus...\oplus D_2}_{p},\ \ \
\phi^\Delta_{m-1}=\phi^\Delta_{m-1}(N,a,d,f,g,\lambda):=\phi_3\oplus\underbrace{\phi_2\oplus...\oplus\phi_2}_{p},
\end{equation}
\begin{equation}\label{2matrix D3}
D_3=D_2\oplus D',\ \ \ \
\phi_3= \left(
\begin{array}{ccc}
\phi_{11}&\phi_{12}&\phi_{13}\\
\phi_{21}&\phi_{22}&\lambda\phi_{21}\\
\phi_{31}&\lambda\phi_{12}&\phi_{11}
\end{array}
\right)\in\mathbf{\Phi}_3,\ \ \ \ \lambda\in\mathbf{k},
\end{equation}
where $D_2,\ D'$ and $\phi_2,\ \phi_{i,j},\ i,j=1,2,$ are given by (\ref{matrix D2})-(\ref{matrix phi2}) and
\begin{equation}\label{2matrix phi3}
\phi_{13}=(r_{ij})\in\wedge^2V^{\vee},\ \ \ \phi_{31}=(s_{ij})\in\wedge^2V^{\vee},\ \ \
\end{equation}
where $r_{ij},s_{ij}\in\mathbf{k}$ satisfy the additional relations
\begin{equation}\label{rel r s}
r_{i3}+r_{i4}=s_{i3}+s_{i4},\ \ \ i=1,2.
\end{equation}
We now proceed along the same lines as before. In particular,
it follows from (\ref{matrix phi2}) and (\ref{2matrix D0,phi0})-(\ref{rel r s}) that the relations (\ref{pt in Sm-1})
and (\ref{pt in Zm-1}) are satisfied for the point $(D^\Delta_{m-1},\phi^\Delta_{m-1})$. Hence, as before, the equations
(\ref{eq1 of fibre}) are automatically satisfied for any $\psi\in\Psi_{m-1}$. Now, substituting the data
$(\theta^0,\alpha^0,D^\Delta_{m-1},\phi^\Delta_{m-1})$ from (\ref{alpha0,theta0}) and
(\ref{2matrix D0,phi0})-(\ref{2matrix phi3}) into (\ref{eq2 of fibre}), we obtain the equations on $(\chi,\psi)$:
\begin{equation}\label{again eq2a of fibre}
(\phi^\Delta_{m-1})^\vee\circ D^\Delta_{m-1}\circ\chi+\psi^\vee\circ\alpha^0\circ\theta^0\in\mathbf{\Psi}_{m-1}.
\end{equation}

Next, using the decompositions (\ref{3 decompns}) we represent
$\chi$ and $\psi$ as $(p+1)$-ples (cf. (\ref{chi-psi}))
\begin{equation}\label{again chi-psi}
\chi=(\chi_0,...,\chi_p),\ \ \ \psi=(\psi_0,...,\psi_p),\ \ \ \psi_0,\chi_0\in\mathbf{\Psi}_3,
\ \ \ \psi_k,\chi_k\in\mathbf{\Psi}_2,\ k=1,...,p,
\end{equation}
where
$\chi_k=(X_k,Y_k),\ \psi_k=(A_k,B_k),\ k=1,...,p,$ are the same matrices of variables as in (\ref{Xk,Yk}),
and
$\chi_0=(X_0,Y_0,Z_0),\ \ \ \psi_0=(A_0,B_0,C_0),\ \ \ X_0,Y_0,Z_0,A_0,B_0,C_0\in\wedge^2V^\vee,$
i.e.
\begin{equation}\label{Ak,Bk,Ck}
X_0=(x_{ij}^{(0)}),\ Y_0=(y_{ij}^{(0)}),\ Z_0=(z_{ij}^{(0)}),\
A_0=(a_{ij}^{(0)}),\ B_0=(b_{ij}^{(0)}),\ C_0=(c_{ij}^{(0)}).
\end{equation}
are skew-symmetric 4$\times$4-matrices of variables. Using the same notation for variables
$x_1^{(k)},...,x_{24}^{(k)},k=1,...,p,$ as in (\ref{set of systems2}) and introducing new variables
$x_{1}^{(0)},...,x_{36}^{(0)}$ as follows:
$x_1^{(0)}=x_{12}^{(0)},x_2^{(0)}=x_{34}^{(0)},x_3^{(0)}=x_{13}^{(0)},x_4^{(0)}=x_{14}^{(0)},x_5^{(0)}=x_{23}^{(0)},
x_6^{(0)}=x_{24}^{(0)},x_7^{(0)}=y_{12}^{(0)},x_8^{(0)}=y_{34}^{(0)},x_9^{(0)}=y_{13}^{(0)},x_{10}^{(0)}=y_{14}^{(0)},
x_{11}^{(0)}=y_{23}^{(0)},x_{12}^{(0)}=y_{24}^{(0)},
x_{13}^{(0)}=z_{12}^{(0)},x_{14}^{(0)}=z_{34}^{(0)},x_{15}^{(0)}=z_{13}^{(0)},x_{16}^{(0)}=z_{14}^{(0)}
,x_{17}^{(0)}=z_{23}^{(0)},x_{18}^{(0)}=z_{24}^{(0)},
x_{19}^{(0)}=a_{12}^{(0)},x_{20}^{(0)}=a_{34}^{(0)},x_{21}^{(0)}=a_{13}^{(0)},x_{22}^{(0)}=a_{14}^{(0)},
x_{23}^{(0)}=a_{23}^{(0)},x_{24}^{(0)}=a_{24}^{(0)},
x_{25}^{(0)}=b_{12}^{(0)},x_{26}^{(0)}=b_{34}^{(0)},x_{27}^{(0)}=b_{13}^{(0)},x_{28}^{(0)}=b_{14}^{(0)},
x_{29}^{(0)}=b_{23}^{(0)},x_{30}^{(0)}=b_{24}^{(0)},
x_{31}^{(0)}=c_{12}^{(0)},x_{32}^{(0)}=c_{34}^{(0)},x_{33}^{(0)}=c_{13}^{(0)},x_{34}^{(0)}=c_{14}^{(0)},
x_{35}^{(0)}=c_{23}^{(0)},x_{36}^{(0)}=c_{24}^{(0)},$
we rewrite the system
(\ref{again eq2a of fibre}) similar to (\ref{set of systems2}) as
\begin{equation}\label{set of systems3}
\sum_{j=1}^{36}\tilde{m}_{ij}x_j^{(0)}=0,\ \ \ \sum_{j=1}^{24}m_{ij}x_j^{(k)}=0,\ \ \ i=1,...,20,\ \ \ k=1,...,p.
\end{equation}
A direct computation of the matrices $\mathbf{M}=(m_{ij})$ and $\mathbf{\widetilde{M}}=(\tilde{m}_{ij})$ for the
above chosen values (\ref{N etc}),(\ref{pij,qij}) of $N,a,d,f,g, p_{ij},q_{ij}$ in (\ref{matrix phi2}) and
(\ref{2matrix D0,phi0}) and, respectively, for the following values  values of $\lambda,\ r_{ij},\ s_{ij}$ in
(\ref{2matrix D3}) and (\ref{2matrix phi3})
satisfying (\ref{rel r s}):
\begin{equation}\label{lambda,rij,sij}
\lambda=-2,\ r_{12}=3,\ r_{13}=7,\ r_{14}=-2,\ r_{23}=4,\ r_{24}=-6,\ r_{34}=-8,
\end{equation}
$$
s_{12}=-8,s_{13}=-3,s_{14}=8,s_{23}=-2,s_{24}=0,s_{34}=-5,
$$
show that $\mathbf{M}$ is the block matrix (\ref{matrix bfM}) and $\mathbf{\widetilde{M}}$ is the block matrix
\begin{equation}\label{matrix tilde bfM}
\mathbf{\widetilde{M}}=\left(
\begin{array}{cccccc}
\mathbf{M}_{11} & \mathbf{M}_{12} & \mathbf{M}_{13} & \mathbf{M}_{\psi} & \mathbf{0} & \mathbf{0} \\
\mathbf{M}_{21} & \mathbf{M}_{22} & \mathbf{M}_{23} & \mathbf{0} & \mathbf{M}_{\psi} & \mathbf{0} \\
\mathbf{M}_{31} & \mathbf{M}_{32} & \mathbf{M}_{33} & \mathbf{0} & \mathbf{0} & \mathbf{M}_{\psi} \\
\end{array}
\right)
\end{equation}
with blocks
\begin{equation}\label{1blocks of tilde bfM}
\mathbf{M}_{11}=\left(
\begin{array}{cccccc}
0 & 0 & 0 & 0 & 0 & 0 \\
0 & 0 & 0 & 0 & 0 & 0 \\
0 & 0 & 100 & 0 & 0 & 0 \\
0 & 0 & 0 & 100 & 0 & 0 \\
0 & 0 & 0 & 0 & 100 & 0 \\
0 & 0 & 0 & 0 & 0 & 100 \\
0 & 0 & 0 & 0 & 0 & 0 \\
0 & 0 & 0 & 0 & 0 & 0 \\
0 & 0 & 0 & 0 & 0 & 0 \\
0 & 0 & 0 & 0 & 0 & 0 \\
 \end{array}
\right),\ \ \
\mathbf{M}_{12}=\left(
\begin{array}{cccccc}
2 & 0 & 0 & 0 & 0 & 0 \\
0 & 2 & 0 & 0 & 0 & 0 \\
0 & 0 & 0 & -5 & -2 & 0 \\
0 & 0 & 2 & 2 & 0 & -2 \\
0 & 0 & 5 & 0 & -2 & -5 \\
0 & 0 & 0 & 5 & 2 & 0 \\
5 & 0 & 0 & 0 & 0 & 0 \\
0 & -5 & 0 & 0 & 0 & 0 \\
2 & 0 & 0 & 0 & 0 & 0 \\
0 & -2 & 0 & 0 & 0 & 0 \\
 \end{array}
\right),\ \ \
\end{equation}
\begin{equation}\label{2blocks of tilde bfM}
\mathbf{M}_{21}=\left(
\begin{array}{cccccc}
0 & 0 & 0 & -5 & 2 & 0 \\
0 & 0 & 0 & -5 & 2 & 0 \\
0 & 0 & 0 & 0 & 0 & 0 \\
-2 & -2 & 0 & 0 & 0 & 0 \\
-5 & -5 & 0 & 0 & 0 & 0 \\
0 & 0 & 0 & 0 & 0 & 0 \\
0 & 0 & 0 & 0 & 0 & -5 \\
0 & 0 & -5 & 0 & 0 & 0 \\
0 & 0 & 2 & 0 & 0 & 0 \\
0 & 0 & 0 & 0 & 0 & 2 \\
\end{array}
\right),\ \ \
\mathbf{M}_{22}=\left(
\begin{array}{cccccc}
100 & 0 & 0 & 0 & 0 & 0 \\
0 & 100 & 0 & 0 & 0 & 0 \\
0 & 0 & 0 & 0 & 0 & 0 \\
0 & 0 & 0 & 100 & 0 & 0 \\
0 & 0 & 0 & 0 & -100 & 0 \\
0 & 0 & 0 & 0 & 0 & 0 \\
0 & 0 & 0 & 0 & 0 & 0 \\
0 & 0 & 0 & 0 & 0 & 0 \\
0 & 0 & 0 & 0 & 0 & 0 \\
0 & 0 & 0 & 0 & 0 & 0 \\
\end{array}
\right),\ \ \
\end{equation}
\begin{equation}\label{3blocks of tilde bfM}
\mathbf{M}_{13}=\left(
\begin{array}{cccccc}
0 & 0 & -6 & -4 & -2 & -7 \\
0 & 0 & 6 & -4 & -2 & 7 \\
-7 & -7 & -5 & 0 & 0 & 0 \\
2 & 2 & 0 & -5 & 0 & 0 \\
-4 & -4 & 0 & 0 & -5 & 0 \\
6 & 6 & 0 & 0 & 0 & -5 \\
0 & 0 & 0 & 0 & -12 & -8 \\
0 & 0 & -8 & 0 & 14 & 0 \\
0 & 0 & -4 & -14 & 0 & 0 \\
0 & 0 & 0 & 12 & 0 & -4 \\
\end{array}
\right),\ \ \
\mathbf{M}_{23}=\left(
\begin{array}{cccccc}
0 & 0 & 0 & 10 & -4 & 0 \\
0 & 0 & 0 & 10 & -4 & 0 \\
0 & 0 & 0 & 0 & 0 & 0 \\
4 & 4 & 0 & 0 & 0 & 0 \\
10 & 10 & 0 & 0 & 0 & 0 \\
0 & 0 & 0 & 0 & 0 & 0 \\
0 & 0 & 0 & 0 & 0 & 10 \\
0 & 0 & 10 & 0 & 0 & 0 \\
0 & 0 & -4 & 0 & 0 & 0 \\
0 & 0 & 0 & 0 & 0 & -4 \\
\end{array}
\right),\ \ \
\end{equation}
\begin{equation}\label{4blocks of tilde bfM}
\mathbf{M}_{31}=\left(
\begin{array}{cccccc}
0 & 0 & 0 & 2 & 8 & 3 \\
0 & 0 & 0 & 2 & 8 & -3 \\
3 & 3 & -13 & 0 & 0 & 0 \\
-8 & -8 & 0 & -13 & 0 & 0 \\
2 & 2 & 0 & 0 & -13 & 0 \\
0 & 0 & 0 & 0 & 0 & -13 \\
0 & 0 & 0 & 0 & 0 & 4 \\
0 & 0 & 4 & 0 & -6 & 0 \\
0 & 0 & 16 & 0 & -6 & 0 \\
0 & 0 & 0 & 0 & 0 & 16 \\
\end{array}
\right),\ \ \
\mathbf{M}_{32}=\left(
\begin{array}{cccccc}
-4 & 0 & 0 & 0 & 0 & 0 \\
0 & -4 & 0 & 0 & 0 & 0 \\
0 & 0 & 0 & 10 & 4 & 0 \\
0 & 0 & -4 & -4 & 0 & 4 \\
0 & 0 & -10 & 0 & 4 & 10 \\
0 & 0 & 0 & -10 & -4 & 0 \\
-10 & 0 & 0 & 0 & 0 & 0 \\
0 & 10 & 0 & 0 & 0 & 0 \\
-4 & 0 & 0 & 0 & 0 & 0 \\
0 & -4 & 0 & 0 & 0 & 0 \\
\end{array}
\right).
\end{equation}
\begin{equation}\label{5blocks of tilde bfM}
\mathbf{M}_{33}=\left(
\begin{array}{cccccc}
0 & 0 & 0 & 0 & 0 & 0 \\
0 & 0 & 0 & 0 & 0 & 0 \\
0 & 0 & 100 & 0 & 0 & 0 \\
0 & 0 & 0 & 100 & 0 & 0 \\
0 & 0 & 0 & 0 & 100 & 0 \\
0 & 0 & 0 & 0 & 0 & 100 \\
0 & 0 & 0 & 0 & 0 & 0 \\
0 & 0 & 0 & 0 & 0 & 0 \\
0 & 0 & 0 & 0 & 0 & 0 \\
0 & 0 & 0 & 0 & 0 & 0 \\
\end{array}
\right),\ \ \
\mathbf{M}_{\psi}=\left(
\begin{array}{cccccc}
-20 & 0 & 20 & 66 & -5 & -47 \\
0 & 3 & 40 & -38 & -37 & -57 \\
57 & -47 & -82 & -38 & -22 & 0 \\
37 & 5 & 7 & -79 & 0 & -22 \\
-38 & 66 & -28 & 0 & -62 & -38 \\
40 & -20 & 0 & -28 & 7 & -59 \\
-56 & 0 & 0 & 0 & 40 & 132 \\
0 & -76 & -76 & 0 & -114 & 0 \\
44 & 0 & -10 & -94 & 0 & 0 \\
0 & -14 & 0 & 80 & 0 & -74 \\
\end{array}
\right).
\end{equation}
Now as in (\ref{rk M=20}) we have $\rk\mathbf{M}=20$. Respectively, from
(\ref{matrix tilde bfM})-(\ref{5blocks of tilde bfM}) we obtain by an explicit computation that
$\rk\mathbf{\widetilde{M}}=30$. Hence, since the matrix of the system (\ref{set of systems3}) is a direct sum of
matrix $\mathbf{\widetilde{M}}$ and $p$ copies of matrix
$\mathbf{M}$, it follows that its rank equals
\begin{equation}\label{again 10(m-1)}
\rk\mathbf{\widetilde{M}}+p\cdot\rk\mathbf{M}=30+20p=10(m-1).
\end{equation}
Denote now by $R(\theta^0,\alpha^0,D_{m-1}^\Delta,\phi_{m-1}^\Delta)$ the rank of the linear system
(\ref{again eq2a of fibre}), equivalent to (\ref{set of systems3}, as a function of
$\theta^0,\alpha^0,D_{m-1}^\Delta,\phi_{m-1}^\Delta$. It follows from (\ref{again 10(m-1)}) that, similar
to (\ref{bold phi etc}), there exist values $\boldsymbol{\phi}_{m-1},\boldsymbol{\alpha},\boldsymbol{\theta}$ of
$\phi_{m-1}^\Delta,\alpha^0,\theta^0$, respectively, such that, as in (\ref{R=}),
\begin{equation}\label{again R=}
R(\boldsymbol{\theta},\boldsymbol{\alpha},D_{m-1}^\Delta,\boldsymbol{\phi}_{m-1})=10(m-1).
\end{equation}
Repeating now the arguments from subsection \ref{m odd 1} and using (\ref{again R=}), we obtain the inclusions
(\ref{embed in Z^0m}) and (\ref{w0 in Z^0*m}) for the above chosen data
$\boldsymbol{\theta},\boldsymbol{\alpha},D_{m-1}^\Delta,\phi_{m-1}^\Delta$.

Finally, using (\ref{2matrix D0,phi0})-(\ref{2matrix phi3}), we modify appropriately the matrices
(\ref{matrix D0,phi0 new})-(\ref{matrix phi2 new}), so that, arguing as in subsection \ref{m odd 2} and using the
inclusions (\ref{embed in Z^0m}) and (\ref{w0 in Z^0*m}), we deduce that
$Z\cap\tau_h(W_{\mathbf{S\Phi}})\neq\emptyset$.
This finishes the proof of Proposition \ref{part of inductn step} for $m$ even.

\begin{remark}\label{computations}
In perfoming the above computations of the rank of the linear system (\ref{eq2 of fibre}) one might try to simplify the
shape of the matrices $\phi_2$ in (\ref{matrix phi2}). E.g., in order to do computations simultaneously for odd and
even values of $m$, one might set $\phi_{12}=\phi_{21}=0$. However, under these constraints the experiments with
computations for arbitrary values of parameters $N,p_{ij},q_{ij}$ give at best the value $9(m-1)$ for the rank of the
system (\ref{eq2 of fibre}), which is insufficient for further arguments. Respectively, in case of $m$ even one might
also try to simplify the shape of the matrix $\phi_3$ in (\ref{2matrix D3}). E.g., one might set
$\phi_{13}=\phi_{31}=0$,
and this would satisfy the equations (\ref{eq1 of fibre}). However, experiments with computations in this case for
arbitrary values of the parameters $N,p_{ij},q_{ij},a,d,f,g,\lambda$ give at best the value 29 for the rank of the
matrix $\mathbf{\widetilde{M}}$ which is also insufficient.
\end{remark}

\end{sub}

\vspace{1cm}

\section{Geometric meaning of $Z_m$. Its relation to t'Hooft instantons}

\vspace{0.5cm}

\begin{sub}{\bf One property of the component $Z$ of the scheme $Z_m$.}\label{Zm and tH}
\rm In this subsection we prove one openness property of the component $Z$ of $Z_m,\ m\ge3,$ introduced in Proposition
\ref{part of inductn step} - see Lemma \ref{Z*(j)} below.

Take an arbitrary point
$$
D\in(\mathbf{S}_m^\vee)^0.
$$
Then in the notation of (\ref{E2m+2(D-1)}) we obtain a symplectic rank-$2m$ vector bundle
$$
E_{2m}(D^{-1})
$$
(see (\ref{E2m+2}) and (\ref{theta}) where we take $2m$ instead of $2m+2$ and put  $B=D^{-1}$)
and a natural epimorphism
$$
c_D:H_m^\vee\otimes\wedge^2 V^\vee\twoheadrightarrow
W_{5m}:=H_m^\vee\otimes\wedge^2 V^\vee/{\rm im}({}^\sharp(D^{-1}))
\simeq H^0(E_{2m}(D^{-1})(1)),\ \ \ \dim W_{5m}=5m.
$$
Now take an arbitrary point
$$
z=(D,\phi)\in Z_m.
$$
Here the morphism $\phi$ understood as a homomorphism
${}^\sharp\phi:H_m\to H_m^\vee\otimes\wedge^2V^\vee$ defines the diagram
\begin{equation}\label{diag n+1}
\xymatrix{
& & H_m\ar[d]_{{}^\sharp\phi}\ar
[dr]^{s(z)}  & & \\
0\ar[r] & H_m\ar[r]^-{{}^\sharp(D^{-1})} &
H_m^\vee\otimes\wedge^2 V^\vee\ar[r]^-{c_D} & W_{5m}\ar[r] & 0.}
\end{equation}
The lower horizontal triple in (\ref{diag n+1}) yields the diagram
\begin{equation}\label{diag n+2}
\xymatrix{0\ar[r] &
H_m\otimes\mathcal{O}_{\mathbb{P}^3}\ar[r]^-{{}^\sharp(D^{-1})}\ar@{=}[d] &
H_m^\vee\otimes\wedge^2 V^\vee\otimes\mathcal{O}_{\mathbb{P}^3}
\ar[r]^-{c_D}\ar@{>>}[d]^{ev} &
W_{5m}\otimes\mathcal{O}_{\mathbb{P}^3}\ar[r]\ar@{>>}[d]^{ev}
& 0 \\
0\ar[r] & H_m\otimes\mathcal{O}_{\mathbb{P}^3}\ar[r]^-{\widetilde{D^{-1}}}
& H_m^\vee\otimes\Omega_{\mathbb{P}^3}(2)
\ar[r]^-{c_D} & E_{2m}(D^{-1})(1)\ar[r] & 0.}
\end{equation}
Moreover, the diagrams (\ref{diag n+1}) and (\ref{diag n+2}) define the composition
\begin{equation}\label{sz}
s_z:\ H_m\otimes\mathcal{O}_{\mathbb{P}^3}(-1)\overset{s(z)}\to W_{5m}
\otimes\mathcal{O}_{\mathbb{P}^3}(-1)\overset{ev}\twoheadrightarrow E_{2m}(D^{-1}).
\end{equation}

Note that the relation $\phi^\vee\circ D\circ\phi\in \mathbf{S}_m$ following from the
definition of $Z$ can be easily rewritten as
\begin{equation}\label{compn sz=0}
{}^ts_z\circ s_z=0,
\end{equation}
where
${}^ts_z:=s_z^\vee\circ\theta$ and $\theta:E_{2m}(D^{-1})\overset{\sim}\to
E_{2m}(D^{-1})^\vee$
is the symplectic structure on $E_{2m}(D^{-1})$ defined as in (\ref{theta}). We have an antiselfdual
complex
\begin{equation}\label{complx}
0\to H_m\otimes\mathcal{O}_{\mathbb{P}^3}(-1)\overset{s_z}\to E_{2m}(D^{-1})
\overset{{}^ts_z}\to H_m^\vee\otimes\mathcal{O}_{\mathbb{P}^3}(1)\to0.
\end{equation}

Now, according to statement (iii) of Proposition \ref{part of inductn step}, take a point
\begin{equation}\label{1good z}
z=(D,\phi)\in Z\cap\tau_h(W_{\mathbf{S\Phi}}),
\end{equation}
where $h$ is a fixed decomposition (\ref{isom h}), and consider the induced decompositions
\begin{equation}\label{D and Phi}
D=D_1\oplus...\oplus D_m,\ \ \ \phi:=\phi_1\oplus...\oplus\phi_m,\ \ \
(D_i,\phi_i)\in(\wedge^2V^\vee)^0\times(\wedge^2V)^0,
\end{equation}
such that
\begin{equation}\label{Di and Phii}
\mathbf{L}:=\underset{i=1}{\overset{m}\cup} L(D_i,\phi_i)=\underset{i=1}{\overset{m}\sqcup} L(D_i,\phi_i).
\end{equation}
is a disjoint union of $2m$ lines in $\mathbb{P}^3$.
Moreover, for this point $z$ we have
\begin{equation}\label{E2m(D-1)}
E_{2m}(D^{-1})=\underset{i=1}{\overset{m}\oplus}E_2(D_i^{-1}),
\end{equation}
where $E_2(D_i^{-1}),\ i=1,...,m,$ are rank-2 null-correlation bundles.

Under the decomposition (\ref{isom h}) the diagrams (\ref{diag n+1}) and  (\ref{diag n+2}) decompose into the direct
sums of $m$ diagrams
\begin{equation}\label{diag n+1 i}
\xymatrix{
& & \mathbf{k}\ar@{>->}[d]_{{}^\sharp\phi_i}\ar
[dr]^{s_i(z)}  & & \\
0\ar[r] & \mathbf{k}\ar[r]^{{}^\sharp(D_i^{-1})\ \ \ \ \ \ \ } &
\wedge^2 V^\vee\ar[r]^-{c_{D_i}} & W_{5(i)}\ar[r] & 0,}
\end{equation}
\begin{equation}\label{diag n+2 i}
\xymatrix{0\ar[r] &
\mathcal{O}_{\mathbb{P}^3}\ar[r]^-{{}^\sharp(D_i^{-1})}\ar@{=}[d] &
\wedge^2 V^\vee\otimes\mathcal{O}_{\mathbb{P}^3}
\ar[r]^-{c_{D_i}}\ar@{>>}[d]^{ev} &
W_{5(i)}\otimes\mathcal{O}_{\mathbb{P}^3}\ar[r]\ar@{>>}[d]^{ev}
& 0 \\
0\ar[r] & \mathcal{O}_{\mathbb{P}^3}\ar[r]^-{\widetilde{D_i^{-1}}}
& \Omega_{\mathbb{P}^3}(2)
\ar[r]^-{c_{D_i}} & E_2(D_i^{-1})(1)\ar[r] & 0},\ \ i=1,...,m,
\end{equation}
in which we substitute $\mathbf{k}$ for $H_1$ and set
$W_{5(i)}:=\wedge^2 V^\vee/{\rm im}({}^\sharp(D_i^{-1}:\mathbf{k}\to\wedge^2V^\vee)),\ \dim W_{5(i)}=5,\ i=1,...,m.$

Note that the decomposition (\ref{isom h}) induces a decomposition of the complex (\ref{complx}) into a direct sum of
$m$ comlexes
\begin{equation}\label{complx i}
0\to\mathcal{O}_{\mathbb{P}^3}(-1)\overset{s_i}\to E_2(D_i^{-1})
\overset{{}^ts_i}\to\mathcal{O}_{\mathbb{P}^3}(1)\to0,\ \ \ i=1,...,m.
\end{equation}
Here the sections $0\ne s_i\in H^0(E_2(D_i^{-1})(1))\simeq W_{5(i)}$ understood as homomorphisms
$\mathbf{k}\to W_{5(i)}$ coincide by construction with homomorphisms
$s_i(z)$ in the diagram (\ref{diag n+1 i}). Hence the homomorphism $s(z)$ in the diagram (\ref{diag n+1}) is also
injective as  the direct sum of $s_i(z)$'s. This means that
${\rm im}({}^\sharp\phi)\cap{\rm im}({}^\sharp(D^{-1}))=\{0\}$ i.e.
\begin{equation}\label{in S0Phi*}
z\in((\mathbf{S}_m^\vee)^0\times\mathbf{\Phi}_m)^*:=\{(D,\phi)\in(\mathbf{S}_m^\vee)^0\times\mathbf{\Phi}_m\
|\  {\rm the\ homomorphism}\ {}^\sharp\phi:H_m\to H_m^\vee\otimes\wedge^2V^\vee
\end{equation}
$$
{\rm is\ injective\ and}\ {\rm im}({}^\sharp\phi)\cap{\rm im}({}^\sharp(D^{-1}))=\{0\}\ \}.
$$
Next, from the definition of $\mathbf{L}$ and the construction of the morphisms $s_z,s_i,i=1,...,m,$
(see (\ref{diag n+1})--(\ref{diag n+2 i}), (\ref{sz}) and (\ref{complx i})) it follows that these complexes are exact
except in their righthand terms and
\begin{equation}\label{cokers}
\coker({}^ts_z)=\mathcal{O}_{\mathbf{L}}(1),\ \ \
\coker({}^ts_i)=\mathcal{O}_{L(D_i,\phi_i)}(1),\ \ \ (s_i)_0=L(D_i,\phi_i),\ \ \ i=1,...,m,
\end{equation}

\begin{remark}
An arbitrary point $D\in(\mathbf{S}_m^\vee)^0$ defines a point

For an arbitrary embedding
$$
j:H_{m-1}\hookrightarrow H_m
$$
and an arbitrary point $z\in(\mathbf{S}_m^\vee)^0\times\mathbf{\Phi}_m$
there is defined an induced morphism of sheaves
\begin{equation}\label{szj}
s_z(j):H_{m-1}\otimes\mathcal{O}_{\mathbb{P}^3}(-1)\overset{j}\to
H_m\otimes\mathcal{O}_{\mathbb{P}^3}(-1)\overset{s_z}\to E_{2m}(D^{-1}).
\end{equation}
\end{remark}
\vspace{3mm}
Let $e_1,...,e_m$ be the basis of $H_m$ related to the decomposition (\ref{isom h}) and set
$$
H_{m-1}:=\Span(e_1,...,e_{m-1}).
$$
Consider the monomorphism
\begin{equation}\label{j0}
j_0:H_{m-1}\hookrightarrow H_m:\ e_i\mapsto e_i+e_{i+1},\ \ \ i=1,...,m-1.
\end{equation}
Since $\mathbf{L}$ is a disjoint union of pairs of lines
$L(D_i,\phi_i),\ i=1,...,m,$
it follows from (\ref{cokers}) and (\ref{j0}) that $s_z(j_0)$ is a subbbundle morphism, i.e.
\begin{equation}\label{coker=0}
\coker({}^ts_z(j_0))=0.
\end{equation}
Now for a given monomorphism $j:H_{m-1}\hookrightarrow H_m$ consider the following conditions on a point
$z=(D,\phi)\in Z$:

(\textbf{I}) the composition $s_{z}(j)=s_{z}\circ j:H_{m-1}\otimes\mathcal{O}_{\mathbb{P}^3}(-1)\to E_{2m}(D^{-1})$
is a subbbundle morphism;

(\textbf{II}) $s_{z}:H_m\otimes\mathcal{O}_{\mathbb{P}^3}(-1)\to E_{2m}(D^{-1})$
is an injective morphism of sheaves (but not a subbundle morphism).

Note that the conditions (\textbf{I}) and (\textbf{II}) are open conditions on the point $z\in Z_m$.
The condition (\textbf{I}) is satisfied for
the point $z$ from (\ref{1good z}) and the embedding $j_0$ by (\ref{coker=0}). The condition (\textbf{II}) is satisfied
for this point $z$ in view of (\ref{cokers}).
Thus, since the set $((\mathbf{S}_m^\vee)^0\times\mathbf{\Phi}_m)^*$ defined in (\ref{in S0Phi*}) is dense open
in $(\mathbf{S}_m^\vee)^0\times\mathbf{\Phi}_m$,
we obtain the following result.
\begin{lemma}\label{Z*(j)}
(i) There exists a monomorphism $j:H_{m-1}\hookrightarrow H_m$ such that the sets
$$
Z(j):=\{z=(D,\phi)\in Z\cap((\mathbf{S}_m^\vee)^0\times\mathbf{\Phi}_m)^*\
|\ z\ {\rm satisfies\ the\ conditions\ (i)\ and\ (\mathbf{II})\ above}\},
$$
$$
{Z}(j,\mathbf{I}):=\{z=(D,\phi)\in Z\cap((\mathbf{S}_m^\vee)^0\times\mathbf{\Phi}_m)^*\
|\ z\ {\rm satisfies\ the\ condition\ (\mathbf{I})\ above}\}
$$
are dense open subsets of $Z$, and we have open embeddings
$Z(j)\hookrightarrow {Z}(j,\mathbf{I})\hookrightarrow Z$. The same is true for a generic monomorphism
$j:H_{m-1}\hookrightarrow H_m$.

(ii) Fix a monomorphism $j:H_{m-1}\hookrightarrow H_m$. Then the sets
$$
Z_m(j):=\{z=(D,\phi)\in Z_m\cap((\mathbf{S}_m^\vee)^0\times\mathbf{\Phi}_m)^*\
|\ z\ {\rm satisfies\ the\ conditions\ (\mathbf{I})\ and\ (\mathbf{II})\ above}\}
$$
$$
Z_m(j,\mathbf{I}):=\{z=(D,\phi)\in Z_m\cap((\mathbf{S}_m^\vee)^0\times\mathbf{\Phi}_m)^*\
|\ z\ {\rm satisfies\ the\ conditions\ (\mathbf{I})\ and\ (\mathbf{II})\ above}\}
$$
are open subsets of $Z_m$. Respectively, let $\widetilde{Z}$ be an arbitrary irreducible component of $Z_m$.
Then the sets
\begin{equation}\label{tilde Z(j,I)}
\widetilde{Z}(j):=\widetilde{Z}\cap Z_m(j),\ \ \ \widetilde{Z}(j,\mathbf{I}):=
\widetilde{Z}\cap Z_m(j,\mathbf{I})
\end{equation}
are open subsets of $\widetilde{Z}$.
\end{lemma}

\end{sub}

\vspace{0.5cm}

\begin{sub}{\bf Relation between $Z$ and t'Hooft instantons. Morphism
$\lambda_{(j)}:Z_m\to\mathbf{S}_{2m-1}$.}\label{widetilde(Z) and tHooft bdls}
\rm

\vspace{0.3cm}

In this subsection we relate the open subset $\widetilde{Z}(j)$ of $Z_m$ introduced in
Lemma \ref{Z*(j)}(ii) to t'Hooft instantons - see Lemma \ref{E2(z) is tHooft}.

In the notation of Lemma \ref{Z*(j)}, assume that $\widetilde{Z}(j)\ne\emptyset$ and take an arbitrary point
$z=(D,\phi)\in \widetilde{Z}(j)$,
so that the symplectic vector bundle $E_{2m}(D^{-1})$  satisfies the diagrams (\ref{diag n+1})-(\ref{diag n+2}).
Respectively, the morphism of sheaves
$s_z$ defined in (\ref{sz}) is injective
- see the definition of condition (ii) above.
In addition, $s_z$ satisfies the relation (\ref{compn sz=0})
which clearly implies the relation
\begin{equation}\label{complex in middle}
{}^ts_z(j)\circ s_z(j)=0
\end{equation}
for the subbundle morphism $s_z(j)$, i.e. we obtain a monad
\begin{equation}\label{E2-monad}
0\to H_{m-1}\otimes\mathcal{O}_{\mathbb{P}^3}(-1)\overset{s_z(j)}\to E_{2m}(D^{-1})
\overset{{}^ts_z(j)}\to H_{m-1}^\vee\otimes\mathcal{O}_{\mathbb{P}^3}(1)\to0,
\end{equation}
From the diagram (\ref{diag n+2}) we deduce the equalities $h^i(E_{2m}(D^{-1})(-2))=0,\ \ \ i\ge0,$
hence the cohomology sheaf of the monad (\ref{E2-monad}) is an instanton bundle
\begin{equation}\label{E2(z)}
E_2(z,j):=\Ker({}^ts_z(j))/\im(s_z(j)),\ \ \ [E_2(z,j)]\in I_{2m-1}.
\end{equation}

Now consider the subvariety $I^{tH}_{2m-1}\subset I_{2m-1}$ of t'Hooft instanton bundles
(see subsection \ref{THinstantons}),
$$
I^{tH}_{2m-1}=\{[E]\in I_{2m-1}\ |\ h^0(E(1))\ne0\}.
$$

\begin{lemma}\label{E2(z) is tHooft}
(i) In notations of Lemma \ref{Z*(j)}(i) let $Z(j)\ne\emptyset$ and let $z=(D,\phi)$ be an arbitrary point of
$Z(j)$. Then the bundle $E_2(z,j)$ is a t'Hooft instanton bundle, i.e. $[E_2(z,j)]\in I^{tH}_{2m-1}$;

(ii) In notations of Lemma \ref{Z*(j)}(iii) let $\widetilde{Z}(j)\ne\emptyset$. Take an arbitrary point
$z\in Z'(j)$. Then the monad (\ref{E2-monad}) is well defined and its cohomology bundle $E_2(z,j)$ is a t'Hooft bundle;

(iii) Fix an isomorphism
\begin{equation}\label{xi again}
\xi:H_m\oplus H_{m-1}\overset{\simeq}\to H_{2m-1},\ \ \ \xi\in{\rm Isom}_{2m-1}.
\end{equation}
Then there is a well defined morphism
\begin{equation}\label{lambda m}
\lambda_{(j)}:Z_m\to\mathbf{S}_{2m-1}:\ z=(D,\phi)\mapsto
A=\widetilde{\xi}(D^{-1},\phi\circ j,-(\phi\circ j)^\vee\circ D\circ(\phi\circ j)).
\end{equation}
such that
\begin{equation}\label{lambda m(Z*(j))}
\lambda_{(j)}(Z_m(j))\subset MI^{tH}_{2m-1}(\xi).
\end{equation}
\end{lemma}

\begin{proof}
(i) Consider the complexes (\ref{complx}) and (\ref{E2-monad}) and set
$$
\mathcal{H}_{m-1}:=H_{m-1}\otimes\mathcal{O}_{\mathbb{P}^3}(-1),\ \
\mathcal{H}_m:=H_m\otimes\mathcal{O}_{\mathbb{P}^3}(-1),\ \
\mathcal{K}_{m+1}:=\coker s_z(j),\ \ \mathcal{K}_m:=\coker s_z.
$$
The complexes  (\ref{complx}) and (\ref{E2-monad}) are antiselfdual, hence they extend to a
commutative diagram
\begin{equation}\label{diagram E2 is tHooft}
\xymatrix{& & & & & E_2(z,j)\ar@{>->}[d]_{\tau} &
\mathcal{O}_{\mathbb{P}^3}(-1)\ar@{>->}[dl]_{\alpha}\ar@{.>}[l]_{u_z} \\
& & \mathcal{H}_{m-1} \ar@{>->}[r]^{s_z(j)}
\ar@{>->}[dl]_{j} & E_{2m}(D^{-1})\ar@{->>}[rr]\ar@{=}[dl]\ar@{->>}[ddd]_{{}^ts_z(j)} &&
\mathcal{K}_{m+1}\ar@{->>}[dl]_{\beta}\ar@{->>}[ddd]_{\delta} & \\
& \mathcal{H}_m \ar@{>->}[r]^{s_z}\ar@{->>}[dl]  &
E_{2m}(D^{-1})\ar@{->>}[rr]\ar[ddd]_{{}^ts_z} & & \mathcal{K}_m\ar[ddd]_{\gamma} \\
\mathcal{O}_{\mathbb{P}^3}(-1) & & & & & \\
& & & \mathcal{H}_{m-1}^\vee\ar@{=}[rr] && \mathcal{H}_{m-1}^\vee & \\
& & \mathcal{H}_m^\vee\ar@{=}[rr]\ar@{->>}[ur]^{j^\vee} && \mathcal{H}_m^\vee\ar@{->>}[ur]^{j^\vee} & & \\
& \mathcal{O}_{\mathbb{P}^3}(1) \ar@{=}[rr]\ar@{>->}[ur] & &
\mathcal{O}_{\mathbb{P}^3}(1), \ar@{>->}[ur] & &}
\end{equation}
in which $\alpha, \beta, \gamma,\delta$ and $\tau$ are the induced morphisms. In this diagram
we have $\beta\circ\alpha=0$ and $j^\vee\circ\gamma\circ\beta=\delta$. Hence
$\delta\circ\alpha=0$. This implies that $\alpha$ factors through the morphism $\tau$, i.e.
there exists an injection $u_z:\mathcal{O}_{\mathbb{P}^3}(-1)\to E_2(z,j)$ such that
$\alpha=\tau\circ u_z$. This injection $u_z$ is a nonzero section $u_z\in H^0(E_2(z,j)(1))$. Hence
$E_2(z,j)$ is a t'Hooft bundle.

(ii) Repeat the above argument.

(iii) This immediately follows from Lemma \ref{E(xi,A)=E(A)} since (\ref{E2-monad})-(\ref{E2(z)}) coincides with
(\ref{monad for E2(xi,A)})-(\ref{E2(xi,A)}) after sustituting $m-1$ for $m$ and putting $B=D^{-1}$.
\end{proof}
\begin{remark}
From the diagram (\ref{E2(z) is tHooft}) it follows that the point $z\in Z(j)$ (respectively, the point
$z\in \widetilde{Z}(j)$) defines not only a t'Hooft bundle $[E_2(z,j)]$, but also a proportionality class $<u_z>$ of a
section $0\ne u_z\in H^0(E_2(z,j))$. Moreover, the pointwise constructions (over $z\in\widetilde{Z}(j)$) of Lemma
\ref{E2(z) is tHooft}
clearly globalize to $\mathbb{P}^3\times \widetilde{Z}(j)$. In particular, the morphism
$\lambda_{(j)}:\widetilde{Z}(j)\to\mathbf{S}_{2m-1}$ defines a subbundle morphism of sheaves
\begin{equation}\label{tilde AZ}
\widetilde{\mathbf{A}}_Z:\mathcal{O}_{\widetilde{Z}(j)}\to\mathbf{S}_{2m-1}\otimes\mathcal{O}_{\widetilde{Z}(j)},
\end{equation}
i.e., equivalently, a family of instanton nets of quadrics
\begin{equation}\label{AZ}
\mathbf{A}_Z: H_{2m-1}\otimes V\otimes\mathcal{O}_{\widetilde{Z}(j)}\to
H_{2m-1}^\vee\otimes V^\vee\otimes\mathcal{O}_{\widetilde{Z}(j)}.
\end{equation}
Let $\pi_Z:\mathbb{P}^3\times \widetilde{Z}(j)\to \widetilde{Z}(j)$ be the projection.
By construction we have a rank $4m$ bundle
$\mathbf{W}_Z:={\rm im}\mathbf{A}_Z$ on $\widetilde{Z}(j)$ and the correspondig monad
$0\to H_{2m-1}\otimes\mathcal{O}_{\mathbb{P}^3}(-1)\otimes\mathcal{O}_{\widetilde{Z}(j)}\to
\mathcal{O}_{\mathbb{P}^3}\boxtimes \mathbf{W}_Z
\to H_{2m-1}^\vee\otimes\mathcal{O}_{\mathbb{P}^3}(1)\otimes\mathcal{O}_{\widetilde{Z}(j)}\to0$
with the cohomology rank 2 bundle $\mathbf{E}_Z$  such that
$\mathbf{E}_Z|_{\mathbb{P}^3\times\{z\}}=E_2(z,j),\ z\in \widetilde{Z}(j)$. This monad,
together with relative Serre duality
for the projection $\pi_Z$, defines in a standard way an isomorphism of locally free
$\mathcal{O}_{\widetilde{Z}(j)}$-sheaves
\begin{equation}\label{fZ}
f_Z:\ H_{2m-1}\otimes\mathcal{O}_{\widetilde{Z}(j)}\overset{\simeq\ }\to\mathbf{G}_Z:=
(\Ext_{\pi_Z}^1(\mathbf{E}_Z(-3),\omega_{\pi_Z}))^\vee
\end{equation}
relativizing the pointwise isomorphisms $f:H_{2m-1}\overset{\simeq\ }\to H^2(E_2(z,j)(-3))$ (cf. Section \ref{general})
and Serre duality $H^2(E_2(z,j)(-3))\overset{\simeq\ }\to(\Ext^1(E_2(z,j)(-3),\omega_{\mathbb{P}^3}))^\vee$.
(Here we set
$\mathbf{E}_Z(k):=\mathbf{E}_Z\otimes\mathcal{O}_{\mathbb{P}^3}(-1)\boxtimes
\mathcal{O}_{\widetilde{Z}(j)},\ k\in\mathbb{Z}$.)
In addition, the sections
$u_z\in H^0(E_2(z,j)),\ z\in\widetilde{Z}(j),$
glue up to a section
\begin{equation}\label{global u}
u:\ \mathcal{O}_{\mathbb{P}^3\times\widetilde{Z}(j)}\to \mathbf{E}_Z(1).
\end{equation}
\end{remark}

\end{sub}

\vspace{0.5cm}

\begin{sub}{\bf Description of the fibers of the morphism $\lambda_{(j)}:Z_m(j)\to\mathbf{S}_{2m-1}$.}
\label{Pointwise descrn of Zm at widetilde Z}
\rm

\vspace{0.3cm}
In this subsection we will give a description of the fibres of the morphism
$\lambda_{(j)}:Z_m(j)\to\mathbf{S}_{2m-1}$
and of its restriction onto $Z$,
$\lambda_j:=\lambda_{(j)}|_{Z}:Z\to\mathbf{S}_{2m-1}$.
The precise statement is given in Lemma \ref{fibre of lambda} below.

To formulate the result on the fibres, note that the point $z=(D,\phi)\in Z_m(j)$ defines
the monad (\ref{E2-monad}) with the cohomology bundle $E_2(z,j)$ with $[E_2(z,j)]\in I_{2m-1}^{tH}$
(see Lemma \ref{E2(z) is tHooft}). The display of this monad twisted by $\mathcal{O}_{\mathbb{P}^3}(1)$ is
\begin{equation}\label{display tH}
\xymatrix{& & E_2(z,j)(1)\ar@{>->}[d] \\
H_{m-1}\otimes\mathcal{O}_{\mathbb{P}^3}\ \ar@{>->}[r]^{s_z(j)} &
E_{2m}(D^{-1})(1)\ar@{->>}[r]^{\epsilon}\ar@{->>}[dr]^{{}^ts_z(j)} &
\coker(s_z(j))\ar@{->>}[d] \\
& & H_{m-1}^\vee\otimes\mathcal{O}_{\mathbb{P}^3}(2).}
\end{equation}

Note that from (\ref{h0=1}) and the definition of $MI^{tH}_{2m-1}$ it follows that
$h^0(E_2(z,j)(1))\le2$. Hence, passing to sections in the diagram (\ref{display tH}) we obtain a
well defined epimorphism
\begin{equation}\label{b(z,j)}
\xymatrix{b(z,j):=h^0({}^ts_z(j)):H^0(E_{2m}(D^{-1})(1))\ar@{->>}[r]^-{h^0(\epsilon)} &
H^0(\coker(s_z(j)))\ar@{->>}[r]^-{can} & }
\end{equation}
$$
\twoheadrightarrow
H^0(\coker(s_z(j)))/H^0(E_2(z,j)(1))
\simeq\left\{\begin{array}{cc}
\mathbf{k}^{4m}, & {\rm if}\ h^0(E_2(z,j)(1))=1\\
\mathbf{k}^{4m-1}, & {\rm if}\ h^0(E_2(z,j)(1))=2
\end{array}\right\}\hookrightarrow H_{m-1}^\vee\otimes S^2V^\vee.
$$
(Note that $h^0(E_2(1))\le2$ for any $[E_2]\in I_{2m-1}^{tH}$.)
In addition, as in Remark \ref{cap=0}, where we take $m-1$ instead of $m$, it follows that
\begin{equation}\label{again cap=0}
{\rm im}({}^\sharp D^{-1})\cap{\rm im}({}^\sharp\phi\circ j)=\{0\},\ \ \ \dim\Span({\rm im}({}^\sharp D^{-1}),\
{\rm im}({}^\sharp\phi\circ j))=2m-1.
\end{equation}

Consider the epimorphism
$c_D:H_m^\vee\otimes\wedge^2 V^\vee\twoheadrightarrow H^0(E_{2m}(D^{-1})(1))$
in this triple (see the diagram (\ref{diag n+2})) and set
\begin{equation}\label{V(z,j)}
V(z,j):=c_D^{-1}(\ker b(z,j)).
\end{equation}
From (\ref{b(z,j)}) it follows immediately that
\begin{equation}\label{V(z,j)=}
V(z,j)\simeq
\left\{\begin{array}{cc}
\mathbf{k}^{2m}, & {\rm if}\ h^0(E_2(z,j)(1))=1,\\
\mathbf{k}^{2m+1}, & {\rm if}\ h^0(E_2(z,j)(1))=2.
\end{array}\right.
\end{equation}

Now observe that the complex (\ref{E2-monad}) is well defined for any $z\in Z_m$ and any
$j:H_{m-1}\hookrightarrow H_m$ since the condition
(\ref{complex in middle}) is a closed condition satisfied for any $z\in Z_m$ (this complex now might be apriori
not left- and right-exact). Hence the homomorphisms
$b(z,j)=h^0({}^ts_z(j)):H^0(E_{2m}(D^{-1})(1))\to H_{m-1}^\vee\otimes S^2V^\vee$
and
$c_D:H_m^\vee\otimes\wedge^2 V^\vee\twoheadrightarrow H^0(E_{2m}(D^{-1})(1))$
are  well defined, and we define the set $V(z,j)$ by the same formula (\ref{V(z,j)}).
Since $Z$ is irreducible, from (\ref{V(z,j)}) it follows by semicontinuity that
\begin{equation}\label{dimV(z,j)}
\dim V(z,j)\ge 2m,\ \ \ z\in Z.
\end{equation}

\begin{lemma}\label{fibre of lambda}
Let $j$ be as in Lemma \ref{Z*(j)}.

(i) For any point $z\in Z_m$ the fibre of the morphism $\lambda_{(j)}:Z_m\to\mathbf{S}_{2m-1}$
through the point $z$ is a reduced scheme naturally identified with $V(z,j)$:
\begin{equation}\label{1fibre=V(z,j)}
\lambda_{(j)}^{-1}(\lambda_{(j)}(z))\overset{\simeq}\to V(z,j),
\end{equation}
where $V(z,j)$ is defined in (\ref{V(z,j)}). Hence, in particular, for any $z\in Z$,
$\dim\lambda_{(j)}^{-1}(\lambda_{(j)}(z))\ge2m$.

(ii) Let $Z_1$ be the union of all possible irreducible components of $Z_m$ distinct from $Z$ and let
$Z_0(j):=Z(j)\smallsetminus Z_1$.
Consider the morphism $\lambda_j:=\lambda_{(j)}|_{Z}:Z\to\mathbf{S}_{2m-1}$.
Then for any $z\in Z_0(j)$ one has a natural isomorphism
\begin{equation}\label{2fibre=V(z,j)}
\lambda_j^{-1}(\lambda_j(z))\overset{\simeq}\to V(z,j),
\end{equation}
where the dimension of $V(z,j)$ is given by (\ref{V(z,j)=}),
and, for an arbitrary $z\in Z$,
\begin{equation}\label{3fibre=V(z,j)}
\lambda_j^{-1}(\lambda_j(z))\subset\lambda_{(j)}^{-1}(\lambda_j(z))=V(z,j),\ \ \
\dim\lambda_j^{-1}(\lambda_j(z))\ge2m.
\end{equation}
If $z\in Z(j,\mathbf{I})$, then the dimension of $V(z,j)$ in (\ref{3fibre=V(z,j)}) is given by
(\ref{V(z,j)=}).

(iii) Let $\widetilde{Z}$ be an arbitrary irreducible component of $Z_m$, let $Z_1$ be the union
of all possible irreducible components of $Z_m$ distinct from $\widetilde{Z}$ and let
$\widetilde{Z}_0(j):=\widetilde{Z}(j)\smallsetminus Z_1$.
Consider the morphism $\tilde{\lambda}_j:=\lambda_{(j)}|_{\widetilde{Z}}:\widetilde{Z}\to\mathbf{S}_{2m-1}$.
Then for any $z\in \widetilde{Z}_0(j)$ one has the natural isomorphism
\begin{equation}\label{4fibre=V(z,j)}
\tilde{\lambda}_j^{-1}(\tilde{\lambda}_j(z))\overset{\simeq}\to V(z,j),
\end{equation}
where the dimension of $V(z,j)$ is given by (\ref{V(z,j)=}), and, for an arbitrary $z\in \widetilde{Z}$,
\begin{equation}\label{5fibre=V(z,j)}
\tilde{\lambda}_j^{-1}(\tilde{\lambda}_j(z))\subset\lambda_{(j)}^{-1}(\tilde{\lambda}_j(z))=V(z,j),\ \ \
\dim\tilde{\lambda}_j^{-1}(\tilde{\lambda}_j(z))\ge2m.
\end{equation}

\end{lemma}
\begin{proof} (i) Consider the spaces
$\mathbf{\Lambda}_m=\wedge^2H_m^\vee\otimes S^2 V^\vee$
and
$\mathbf{\Lambda}_{m-1}=\wedge^2H_{m-1}^\vee\otimes S^2 V^\vee$
together with projections
$q_m:\wedge^2(H_m^\vee\otimes V^\vee)\rightarrow\mathbf{\Lambda}_m$
and
$q_{m-1}:\wedge^2(H_{m-1}^\vee\otimes V^\vee)\rightarrow\mathbf{\Lambda}_{m-1}$,
respectively (cf. (\ref{Sm,Lambdam}) and (\ref{qm})). Fix a monomorphism
$j_{\mathbf{k}}:\mathbf{k}\hookrightarrow H_m$ such that
$j(H_{m-1})\cap\mathbf{k}=\{0\}$,
i. e. we have a direct sum decomposition of $H_m$ together with embeddings of summands
\begin{equation}\label{Hm decomp}
H_m=H_{m-1}\oplus\mathbf{k},\ \ \
H_{m-1}\overset{j}\hookrightarrow H_m\overset{j_{\mathbf{k}}}\hookleftarrow\mathbf{k}.
\end{equation}
This decomposition induces a direct sum decomposition of $\mathbf{\Lambda}$ together with projections
\begin{equation}\label{decomp of Lambda m}
\mathbf{\Lambda}_m=\mathbf{\Lambda}_{m-1}\oplus\Hom(\mathbf{k},H^\vee_{m-1}\otimes S^2V^\vee),\ \ \
\mathbf{\Lambda}_{m-1}\overset{pr'}\leftarrow\mathbf{\Lambda}_m\overset{pr''}\to
\Hom(\mathbf{k},H^\vee_{m-1}\otimes S^2V^\vee).
\end{equation}
Now the equations of $Z_m$ in $(\mathbf{S}_m^\vee)^0\times\mathbf{\Phi}_m$ are
\begin{equation}\label{eqns on mathcal A}
\mathcal{A}:=q_m(\phi^\vee\circ D\circ\phi)=0.
\end{equation}
Next, consider the diagram (\ref{diag DC}) twisted by $\mathcal{O}_{\mathbb{P}^3}(1)$, in which we substitute $m-1$ for
$m$, set $B=D^{-1}$ and put $s_z(j)$ instead of $\rho_{\xi,A}$ and $\phi\circ j$ instead of
$\widetilde{C}$, respectively.
Proceeding to sections in this diagram and, respectively, to sections in the diagram  (\ref{display tH}) we see that
the condition
\begin{equation}\label{prns of A}
0=pr'(\mathcal{A}):=q_{m-1}((\phi\circ j)^\vee\circ D\circ(\phi\circ j))=b(z,j)\circ e(z)
\end{equation}
is automatically satisfied,
where $e(z)$ is a homomorphism $e(z)=h^0(s_z(j)):H_{m-1}\to H^0(E_{2m}(B)(1))$. (Clearly, the vanishing of
$pr'(\mathcal{A})$ can be equivalently rewritten as the condition that ${}^\sharp\phi\circ j$ embeds $H_{m-1}$ in
$V(z,j)$.)
Hence the equations
(\ref{eqns on mathcal A}) are equivalent to the equations
\begin{equation}\label{2prns of A}
pr''(\mathcal{A})=b(z,j)\circ c(z)\circ{}^\sharp\phi\circ j_{\mathbf{k}}=0,
\end{equation}
which in view of the definition (\ref{V(z,j)}) mean that
$$
{}^\sharp\phi|_{\mathbf{k}}\subset V(z,j)
$$
Thus, since the point $\lambda_{(j)}(z)$ is given, so that the points $D$ and $\phi\circ j$ are determined by
$\lambda_{(j)}(z)$ (see (\ref{lambda m})), it follows that the point $(D,\phi)\in\lambda_{(j)}(z)^{-1}(\lambda_{(j)}(z))$
is determined by the data ${}^\sharp\phi|_{\mathbf{k}}$.
Hence, the above inclusion implies that $\lambda_{(j)}(z)^{-1}(\lambda_{(j)}(z))\simeq V(z,j)$.

(ii)-(iii) follow from (i).

Note that the above argument can be illustrateded by the diagram
\begin{equation}\label{diag V(z,j)}
\xymatrix{ & & H_m\ar@{=}[rr]\ar@{>->}[dd] & & H_m\ar@{>->}[dd]^{{}^\sharp B} &  \\
H_{m-1}\ar@{^(->}[r]^-{j}\ar@{^(->}[drr]_-{{}^\sharp\phi\circ j}\ar@{=}[dd] & H_m\ar[dr]^-{{}^\sharp\phi} & &
\ \mathbf{k}\ar@{_(->}[ll]_-{j_{\mathbf{k}}\ \ \ \ \ }\ar[dl]^{{}^\sharp\phi|_{\mathbf{k}}} & & \\
& 0\ar[r] & V(z,j)\ar[rr]\ar@{>>}[dd] & & H_m^\vee\otimes\wedge^2V^\vee\ar[r]\ar@{>>}[dd]^{c_D} &
H_{m-1}^\vee\otimes S^2V^\vee\ar@{=}[dd]\\
H_{m-1}\ar@{^(->}[drr]\ar@{^(->}[drrrr]^{\ \ \ \ \ e(z)} & & & & & \\
& 0\ar[r] & \ker b(z,j)\ar[rr] & & H^0(E_{2m}(D^{-1})(1))\ar[r]^{b(z,j)} & H_{m-1}^\vee\otimes S^2V^\vee.}
\end{equation}
\end{proof}

\begin{remark}\label{fibre in Phi m}
Note here that, as it follows from the proof of this Lemma, for $z=(D,\phi)\in Z$
the fiber
$V(z,j)=\lambda_{(j)}^{-1}(\lambda_{(j)}(z))\subset H_m^\vee\otimes\wedge^2V^\vee$
of the morphism $\lambda_{(j)}$
naturally lies in
$\{D\}\times\mathbf{\Phi}_m$
via the embedding
$j^*_\mathbf{k}:H_m^\vee\otimes\wedge^2V^\vee\hookrightarrow\Hom(H_m,H_m^\vee)\otimes\wedge^2V^\vee=\mathbf{\Phi}_m
=\{D\}\times\mathbf{\Phi}_m$
induced by the embedding
$j_\mathbf{k}:\mathbf{k}\hookrightarrow H_m$.
\end{remark}

\begin{lemma}\label{codim RZ*}
Consider the set $R_{Z}=\{z=(D,\phi)\in Z\ |\ {\rm rank}\ s(z)\le m-2\}$
where the homomorphism
$s(z)=c_D\circ{}^\sharp\phi:H_m\to W_{5m}$
is defined for $z=(D,\phi)$ in (\ref{diag n+1}). Then
$$
{\rm codim}_{Z}R_{Z}\ge2.
$$
\end{lemma}
\begin{proof}
Fix a monomorphism $j:H_{m-1}\hookrightarrow H_m$ satisfying the conditions of Lemma \ref{Z*(j)}, so that
$Z(j)$ is nonempty, hence dense in $Z$.
and take any point $z=(D,\phi)\in Z$.
From the definition of the set $V(z,j)$ (see (\ref{V(z,j)})) it follows that, for $z\in R_{Z}$, one has a natural
inclusion $c_D^{-1}({\rm im}\ s(z))\subset \lambda_j^{-1}(\lambda_j(z))\subset V(z,j)$
(cf. the diagram (\ref{diag V(z,j)}), so that the diagram
(\ref{diag n+1}) and the definition of $R_{Z}$ imply
$\dim c_D^{-1}({\rm im}\ s(z))\le{\rm rank}\ s(z)+m\le2m-2$. Hence by Lemma \ref{fibre of lambda}(ii)
${\rm codim}_{\lambda_j^{-1}(\lambda_j(z))}c_D^{-1}({\rm im}\ s(z))\ge2$.
Thus we have an inclusion
$R_{Z}\simeq\underset{z\in Z}\cup c_D^{-1}({\rm im}\ s(z))\subset
\underset{z\in Z}\cup\lambda_j^{-1}(\lambda_j(z))=Z$,
which together with the last inequality yields the Lemma.
\end{proof}

\end{sub}

\vspace{0.5cm}

\section{Complete family of t'Hooft sheaves with $c_2=2m-1$. End of the proof of Theorem
\ref{Irreducibility of Zm}}\label{gen pos}

\vspace{0.5cm}

In this section we construct a complete $(10m-1)$-dimensional family $T$ of t'Hooft $(2m-1)$-instsantons and their
degenerations (we call these degenerations {\it t'Hooft sheaves}). The family $T$ will be used to prove that the
variety $Z$ studied in the previous two sections coincides with $Z_m$. This finishes the proof of Theorem
\ref{Irreducibility of Zm}.

\begin{sub}{\bf Construction of a complete family $\mathbf{E}\to\mathbf{T}$ of $(2m-1)$-t'Hooft
sheaves.}\label{Definiton of T}
\rm
\vspace{0.3cm}

Consider the subvariety $I^{tH}_{2m-1}\subset I_{2m-1}$ of t'Hooft $(2m-1)$-instantons.
We first recall the following two properties of an arbitrary t'Hooft instanton $[E]\in I^{tH}_{2m-1},\ m\ge1,$ -
see \cite{BT} and \cite{NT}:

(i) $h^0(E(1))\le2$;

(ii) for any section $0\ne s\in H^0(E(1))$ the zero scheme $Z_s=(s)_0$ is locally contained in a smooth surface;

(iii)$(Z_s)_{red}$ is a disjoint union of lines $l_1,...,l_r,\ 1\le r\le2m,$ and
$\mathcal{O}_{Z_s}=\underset{i=1}{\overset{r}\oplus}\mathcal{O}_{Z_i}$, where for each $i,\ 1\le i\le r,$ the scheme
$Z_i$ has a filtration by subschemes
$l_i=Z_{1i}\subset Z_{1i}\subset...\subset Z_{m_i,i}=Z_i$
for some $m_i\ge1,$ with $\Supp(Z_{ji})=l_i$ such that,
if $m_i\ge2$, then
\begin{equation}\label{mult str}
\mathcal{O}_{Z_{j-1,i}}=\mathcal{O}_{Z_{ji}}/\mathcal{O}_{l_i},\ 2\ge j\ge m_i;
\end{equation}

For a given integer $d\ge1$ consider the Hilbert scheme $\mathcal{H}_d:={\rm Hilb}^dG$ of 0-dimensional subschemes of
length $d$ of the Grassmannian $G=G(1,3)$ of lines in $\mathbb{P}^3$, and let
$\Gamma_{\mathcal{H}_d}\subset G\times\mathcal{H}_d$ be the universal family with projections
$G\overset{p_d}\leftarrow\Gamma_{\mathcal{H}_d}\overset{q_d}\to\mathcal{H}_d$. For a given point $x\in\mathcal{H}_d$
we denote by $Y_x$ the corresponding 0-dimensional subscheme $p_d(q_d^{-1}(x))$ of $G$.
We call a point $x\in\mathcal{H}_d$ {\it curvilinear} if there exists an integer $b\ge1$, a partition
$d=d_1+...+d_b,\ d_i\ge1$, and points $x_i\in\mathcal{H}_{d_i},\ 1\le i\le b,$ such that\\
(a) for each $i,\ 1\le i\le b,$ the subscheme $Y_{x_i}\subset G$ is isomorphic to
${\rm Spec}(\mathbf{k}[t]/(t^{d_i+1}))$, and\\
(b) $Y_x$ is a disjoint union $Y_x=Y_{x_1}\sqcup...\sqcup Y_{x_b}$.

Set $\mathcal{H}_d^{curv}:=\{x\in\mathcal{H}_d\ |\ x\ \text{is curvilinear}\}$. It is well known (and easily seen)
that $\mathcal{H}_d^{curv}$ is an open smooth $4d$-dimensional subscheme of $\mathcal{H}_d$.
Next, let $\Gamma\subset\mathbb{P}^3\times G$ be the graph of incidence, together with projections
$\mathbb{P}^3\overset{p}\leftarrow\Gamma\overset{q}\to G$.
From the above properties (i)-(iii) we deduce now the following lemma.
\begin{lemma}\label{curv}
For each $[E]\in I^{tH}_{2m-1} $ and $0\ne s\in H^0(E(1))$, there exists a curvilinear point
$x=x([E],s)\in\mathcal{H}_{2m}^{curv}$ such that $Z_s\overset{sets}=p(q^{-1}(Y_x))$
and the scheme structure of $Z_s$
coincides with that given by formula
\begin{equation}\label{scheme str of Zs}
\mathcal{O}_{Z_s}=p_*q^*\mathcal{O}_{Y_x}.
\end{equation}
\end{lemma}
\begin{proof}
Since by (ii) the support of $Z_s$ is a disjoint union of lines; hence from the definition of
curviliear schemes we deduce that it is enough to consider the case when $Z_s$ is a single line, say, $l$ with a
nonreduced structure, i.e. there is a filtration of $Z_s$ by subschemes
\begin{equation}\label{subschemes}
l=Z_1\subset Z_2\subset...\subset Z_{2m}=Z_s,\ m\ge2,
\end{equation}
such that the following triples are exact
(see (\ref{mult str})):
\begin{equation}\label{triples for Zs}
0\to\mathcal{O}_l\to
\mathcal{O}_{Z_2}\to\mathcal{O}_l\to0\ ,\ .\ .\ . \ ,\ \
0\to\mathcal{O}_l\to\mathcal{O}_{Z_{2m}}\to\mathcal{O}_{Z_{2m-1}}\to0.
\end{equation}
From the first triple in (\ref{triples for Zs}), (ii) and the Ferrand construction \cite[\S 1]{BF} it follows that
$\mathcal{O}_l$ is a factor-sheaf of the conormal sheaf $N_{l/\mathbb{P}^3}\simeq2\mathcal{O}_{\mathbb{P}^3}$ and that
the surjection $N_{l/\mathbb{P}^3}\twoheadrightarrow\mathcal{O}_l$ gives a double structure on $l$ coinciding with
the scheme structure of $Z_2$. This surjection implies that $Z_2$ lies as a scheme on a smooth quadric, say, $Q$
passing through $l$. Choose homogeneous coordinates $(x_0:x_1:x_2:x_3)$ on $\mathbb{P}^3$ such that\\
(1) $l=\{x_2=x_3=0\},\ Q=\{x_0x_2-x_1x_3=0\}$, and\\
(2) let $\mathbb{P}^3=U_0\cup U_1$ be the open cover of $\mathbb{P}^3$ by the sets $U_i=\{x_i\ne0\},\ i=0,1;$ then the
ideal of $Z_2\cap U_i$ in $\mathbf{k}[U_i]$ is generated by $x_2/x_0$ and $(x_3/x_0)^2$ for $i=0$ and, respectively, by
 $x_3/x_1$ and $(x_2/x_1)^2$ for $i=1$.

Let $S_1,..,S_c$ be quasiprojective smooth surfaces in $\mathbb{P}^3$ such that the sets
$Z_{(k)}:=Z_s\cap S_k,\ k=1,...,c,$ constitute an open cover of $Z_s$. (Such surfaces exist because of (ii).) Set
$Z_{(ik)}:=Z_{(k)}\cap U_i,\ i=0,1,\ k=1,...,c.$
From (1)-(iii) and (1)-(2) follows the property\\
(3) for $k=1,...,c$ the ideal $I_{Z_{(ik)}}$ of $Z_{(ik)}$ in $\mathcal{O}[U_i\cap S_k]$ is generated by $(x_3/x_0)^{2m+1}$ for $i=0$ and,
respectively, by $(x_2/x_1)^{2m+1}$ for $i=1$.\\
Since by (1) the elements $x_3/x_0\in\mathcal{O}[Z_{(0k)}]$ and $x_2/x_1\in\mathcal{O}[Z_{(1k)}]$
coincide in $\mathcal{O}[Z_{(0k)}\cap Z_{(1k)}],\ k=1,...,c,$ it follows that there are well defined homomorphisms
$\mathbf{k}[t]/(t^{2m+1})\to\mathcal{O}[Z_{(ik)}]:\ 1$mod$(t^{2m+1})\mapsto 1$mod$I_{Z_{(ik)}}$ and
$t$mod$(t^{2m+1})\mapsto(x_3/x_0)$mod$I_{Z_{(0k)}}$ for $i=0$, respectively,
$t$mod$(t^{2m+1})\mapsto(x_2/x_1)$mod$I_{Z_{(1k)}}$ for $i=1$, which are compatible on $Z_{(0k)}\cap Z_{(1k)}$.
This defines a morphism
$\pi_Z:Z_s\to{\rm Spec}(\mathbf{k}[t]/(t^{2m+1}))$. Set $\tau_i:={\rm Spec}(\mathbf{k}[t]/(t^{i+1})),\ i=0,...,2m$.
From the definition of the morphism $\pi_Z$ and exact triples
(\ref{triples for Zs}) it follows that, for $i=2,..,2m,$ the (nilpotent) ideal sheaf
$\mathcal{I}_i:=\mathcal{I}_{\tau_{i-1},\tau_i}\subset\mathcal{O}_{\tau_i}$
satisfies the isomorphism
$mult:\mathcal{I}_i\otimes_{\mathcal{O}_{\tau_i}}\mathcal{O}_{Z_i}\overset{\simeq}\to
\mathcal{I}_{Z_i}:a\otimes\bar{1}\mapsto\pi_Z^*(a)$.
Hence, by \cite[Lemma 2.13]{HL} the morphism $\pi_Z$ is a flat family of lines over
$\tau_{2m}$,
so that it defines an embedding
$\tau_{2m}={\rm Spec}(\mathbf{k}[t]/(t^{2m+1}))\hookrightarrow G$,
i.e. a curvilinear point $x\in\mathcal{H}_{2m}$ such that
$p:\ q^{-1}(Y_x)\overset{\simeq}\to Z_s$ is an isomorphism. Lemma is proved.
\end{proof}

\begin{remark}\label{tH-curv}
One easily sees that $\mathcal{H}_{2m}^{tH-curv}:=\{x\in\mathcal{H}_{2m}^{curv}\ |\ x=x([E],s)$ for some
$[E]\in I^{tH}_{2m-1} $ and $0\ne s\in H^0(E(1))\}$ is a dense open subset of $\mathcal{H}_{2m}^{curv}$.
We thus consider its closure
$\overline{\mathcal{H}_{2m}^{tH-curv}}=\overline{\mathcal{H}_{2m}^{curv}}$
in ${\Hilb}^{2m}G$. Fix a desingularization $\mathcal{H}$ of $\overline{\mathcal{H}_{2m}^{tH-curv}}$.
$\mathcal{H}$ is a smooth integral scheme, and there is thea graph of incidence
$\Gamma_{\mathcal{H}}\subset G\times\mathcal{H}$
with projections
$G\overset{p_\mathcal{H}}\leftarrow\Gamma_{\mathcal{H}}\overset{q_\mathcal{H}}\to\mathcal{H}$.
\end{remark}

Consider the subcheme
$\widetilde{\mathbf{L}}_\mathcal{H}=\Gamma_\mathcal{H}\times_{G\times\mathcal{H}}\Gamma\times\mathcal{H}$
of
$\Gamma\times\mathcal{H}$ and set
$$
\mathbf{L}_\mathcal{H}:=pr_1(\widetilde{\mathbf{L}}_\mathcal{H}),
$$
where $pr_1:\Gamma\times\mathcal{H}\to\mathbb{P}^3\times\mathcal{H}$
is the projection. We endow $\mathbf{L}_\mathcal{H}$
with the structure of a subscheme of $\mathbb{P}^3\times\mathcal{H}$ via setting
$$
\mathcal{O}_{\mathbf{L}_\mathcal{H}}:=pr_{1*}(\mathcal{O}_{\widetilde{\mathbf{L}}_\mathcal{H}}).
$$
Since the sheaf $pr_{1*}(\mathcal{O}_{\widetilde{\mathbf{L}}_\mathcal{H}})$ is clearly flat over $\mathcal{H}$,
in order to prove that the above definition is consistent, one has to check it fibrewise with respect to the projection
$p_L:\mathbf{L}_\mathcal{H}\to\mathcal{H}$.
Thus, taking any point $y\in\mathcal{H}$ and the corresponding 0-dimensional
scheme $Z=Z_y$ of $G$, respectively, the subscheme $\widetilde{L}_y=q^{-1}(Z_y)$ of $P^3\times G$, we have to check
that the sheaf $p_*\mathcal{O}_{\Gamma_y}$ is the structure sheaf of a certain subscheme $L_y$ of $\mathbb{P}^3$
supported at $p(\widetilde{L}_y)$. Take any closed point $z\in Z_y$ and set $\tilde{l}=q^{-1}(z)$, respectively,
$l=p(\tilde{l})$.
Also, take an arbitrary point $\tilde{x}\in\tilde{l}$, respectively, $x=p(\tilde{x})\in l$.
Applying the functor $p_*$ to the composition of surjections
$\mathcal{O}_\Gamma\twoheadrightarrow\mathcal{O}_{L_y}\twoheadrightarrow
\mathcal{O}_{\tilde{l}}\twoheadrightarrow\mathbf{k}_{\tilde{x}}$
we obtain a surjection
$\mathcal{O}_{\mathbb{P}^3}=p_*\mathcal{O}_\Gamma\twoheadrightarrow p_*\mathbf{k}_{\tilde{x}}=\mathbf{k}_x$
as the composition
$\mathcal{O}_{\mathbb{P}^3}\overset{\epsilon}\to
p_*\mathcal{O}_{\tilde{L}_y}\twoheadrightarrow\mathbf{k}_x$.
Hence, by Nakayama's lemma $\epsilon$ is an epimorphism, as stated. Note that, by construction, the scheme
$L_y$ has a filtration by subschemes as in (\ref{subschemes})-(\ref{triples for Zs}):
\begin{equation}\label{subschemes2}
\emptyset=L_0=L_1\subset L_2\subset...\subset L_{2m}=L_y,\ \ \
\mathcal{O}_{L_{i-1}}=\mathcal{O}_{L_i}/ \mathcal{O}_{l_i},\ \ 1\le i\le2m,
\end{equation}
where $l_1,...,l_{2m}$ are lines in $\mathbb{P}^3$, not necessarily distinct, corresponding to closed points of the
scheme $Z_y$.

\begin{remark}\label{special tH bdls}
Consider the set $\mathcal{H}_s:=\{x\in\mathcal{H}_{2m}^{tH-curv}\ |\ x=x([E],s)$ for some
$[E]\in I^{tH}_{2m-1} $ with $h^0(E(1))\ge2\}$. $\mathcal{H}_s$ is a closed subset of $\mathcal{H}_{2m}^{tH-curv}$ and
it is well known (see, e.g., \cite{BT}) that the condition
$x([E],s)\in\mathcal{H}_s$ is equivalent to the condition that the scheme $Z_s=(s)_0$ lies on a smooth quadric in
$\mathbb{P}^3$. This is, in turn, equivalent to saying that the 0-dimensional subscheme $Y_x$ of $G$ lies one a
projective plane $\mathbb{P}^2$ in $\mathbb{P}^5=\Span(G)$ intersecting $G$ in a smooth conic (i.e. a general plane in
$\mathbb{P}^5$. Whence it follows that $\dim\mathcal{H}_s={\rm length}(Y_x)+\dim G(2,\mathbb{P}^5)=2m+9$. Respectively,
\begin{equation}\label{codim Hs}
\codim_{\mathcal{H}}\mathcal{H}_s=8m-(2m+9)=6m-9>2,\ \ \ \ m\ge2.
\end{equation}

\end{remark}

Now let $pr_2:\mathbb{P}^3\times\mathcal{H}$ be the projection and consider the flat over $\mathcal{H}$ sheaf
$\mathcal{I}_\mathbf{L}(1):=\mathcal{I}_{\mathbf{L},\mathbb{P}^3\times\mathcal{H}}\otimes
\mathcal{O}_{\mathbb{P}^3}(1)\boxtimes\mathcal{O}_\mathcal{H}$
and the relative $\Ext$-sheaf
$$
\mathbf{F}=\Ext^1_{pr_2}(\mathcal{I}_\mathbf{L}(1),\mathcal{O}_{\mathbb{P}^3}(-1)\boxtimes\mathcal{O}_\mathcal{H}).
$$
A standard computation using (\ref{subschemes2}) shows that the sheaf $\mathbf{F}$ satisfies the base change isomorphism
\begin{equation}\label{base-change y}
b_y:\mathbf{F}\otimes\mathbf{k}_y\overset{\simeq}\to
\Ext^1(\mathcal{I}_{L_y,P^3}(1),\mathcal{O}_{\mathbb{P}^3}(-1))\simeq\mathbf{k}^{2m},\ \ \ y\in\mathcal{H}.
\end{equation}
Hence $\mathbf{F}$ is a locally free $\mathcal{O}_\mathcal{H}$-sheaf of rank $2m$. We thus have a smooth integral
$(10m-1)$-dimensional scheme
$\mathbf{T}=\mathbf{Proj}(\mathbf{F}^\vee)$ with structure morphism
$p_\mathbf{T}:\mathbf{T}\to\mathcal{H}$ and the Grothendieck sheaf $\mathcal{O}_{\mathbf{T}/\mathcal{H}}(1)$.
In particular, $\mathbf{T}$ is a smooth variety of dimension
\begin{equation}\label{dim T}
\dim\mathbf{T}=\dim\mathcal{H}+\rk\mathbf{F}^\vee-1=8m+2m-1=10m-1.
\end{equation}
Moreover, let
$\mathbf{p_T}=id_{\mathbb{P}^3}\times p_\mathbf{T}:\mathbb{P}^3\times\mathbf{T}\to\mathbb{P}^3\times\mathcal{H}$
be the projection and set $\mathbf{L}_\mathbf{T}:=\mathbf{p_T}^{-1}(\mathbf{L})$.
On $\mathbb{P}^3\times\mathbf{T}$ there is a universal family of (classes of) extensions of sheaves
- see, e.g., \cite[Cor. 4.5]{L}:
\begin{equation}\label{univ ext}
0\to\mathcal{O}_{\mathbb{P}^3}(-1)\boxtimes\mathcal{O}_{\mathbf{T}/\mathcal{H}}(1)
\to\mathbf{E}\to\mathcal{I}_{\mathbf{L}_\mathbf{T}}(1)\to0,
\end{equation}
where $\mathcal{I}_{\mathbf{L}_\mathbf{T}}:=\mathcal{I}_{\mathbf{L}_\mathbf{T},\mathbb{P}^3\times\mathbf{T}}$.
By construction, for any closed point $t\in\mathbf{T}$ the sheaf $E_t=E|_{\mathbb{P}^3\times\{t\}}$
is a nontrivial extension of the form
\begin{equation}\label{ext Et}
0\to\mathcal{O}_{\mathbb{P}^3}(-1)\to E_t\to\mathcal{I}_{L_y}(1)\to0,\ y=p_T(t),
\end{equation}
hence

(i) $E_t$ is a stable rank-2 sheaf (i.e. $[E_t]\in M_{\mathbb{P}^3}(2;0,2,0)$), which  satisfies the condition
$h^0(E_y(1))>0$; furthermore, from (\ref{ext Et}) and (\ref{subschemes2}) it follows easily that

(ii) $h^0(E_t(-2))=0$;

(iii) there exists a dense open subset $\mathbf{T}'$ of $p_\mathbf{T}^{-1}(\mathcal{H}_{2m}^{tH-curv})$,
hence also of $\mathbf{T}$ such that, for $t\in\mathbf{T}'$, $E_t$ is locally free, i.e. $E_t$ is a t'Hooft bundle;

(iv) there exists a dense open subset  $\mathbf{T}''$ of  $\mathbf{T}'$ such that, for $t\in\mathbf{T}''$,
$h^0(E_t(1))=1$; furthermore, for any two distinct points $t,t'\in\mathbf{T}''$ one has $E_t\not\simeq E_{t'}$.

The properties (i)-(iv) mean that there is a well defined modular morphism
$\mathbf{f}:\mathbf{T}\to M_{\mathbb{P}^3}(2;0,2,0):t\mapsto[E_t]$ such that
\begin{equation}\label{im f}
\mathbf{f}(\mathbf{T})=\overline{I^{tH}_{2m-1}}
\end{equation}
is the closure of $I^{tH}_{2m-1}$ in $M_{\mathbb{P}^3}(2;0,2,0)$. Moreover, $f|_{\mathbf{T}^0}$ is injective. We thus
call the family $\mathbf{E}\to\mathbf{T}$ {\it the complete $(10m-1)$-dimensional family of t'Hooft sheaves}.

Note also that the property (iii) above implies that
\begin{equation}\label{dim supp}
\Supp\mathcal{E}xt^1_{\mathcal{O}_{\mathbb{P}^3\times\mathbf{T}}}(\mathbf{E},\mathcal{O}_{\mathbb{P}^3\times\mathbf{T}})
\subset\mathbb{P}^3\times\partial\mathbf{T},\ \ \ \partial\mathbf{T}:=\mathbf{T}\smallsetminus\mathbf{T}'.
\end{equation}
\begin{remark}
Assume that we are given a vector bundle $\mathbf{E}_B$ on $\mathbb{P}^3\times B$ such that,
(i) for each $b\in B$, $E_b=\mathbf{E}_B|_{\mathbb{P}^3\times\{b\}}$ is a t'Hooft bundle,
(ii) there is given a morphism
$u_B:\mathcal{O}_{\mathbb{P}^3}(-1)\otimes\mathcal{N}_B\to\mathbf{E}_B$
nonvanishing for any $b\in B$, where $\mathcal{N}_B$ is some invertible sheaf on $B$.  Then
$\coker u_B=\mathcal{O}_{\mathbb{P}^3}(1)\boxtimes\mathcal{O}_B\otimes\mathcal{I}_{\mathbf{L}_B,\mathbb{P}^3\times B}$
where $\mathbf{L}_B=\underset{b\in B}\cup Z_b$ is a union of subschemes $Z_b$ of $\mathbb{P}^3$ described in
Lemma \ref{curv}. We thus have an extension
$0\to\mathcal{O}_{\mathbb{P}^3}(-1)\boxtimes\mathcal{O}_B\overset{u_B}\to\mathbf{E}_B\to
\mathcal{O}_{\mathbb{P}^3}(1)\boxtimes\mathcal{O}_B\otimes\mathcal{I}_{\mathbf{L}_B,\mathbb{P}^3\times B}\to0$.
It follows in a standard way from \cite{L} that there exists a morphism $r:B\to\mathbf{T}'$ such that the last
extension is obtained via applying the functor $(id_{\mathbb{P}^3}\times r)^*$ to the triple (\ref{univ ext}).
In particular, applying this remark to the bundle $\mathbf{E}_Z$ on $\mathbb{P}^3\times Z(j)$ and the morphism
$u$ in (\ref{global u}), i.e. taking $B=Z(j)$ and $u_B=u$, we obtain the morphism
$r=r_\mathbf{T}:Z(j)\to\mathbf{T}'$ such that
\begin{equation}\label{rZ}
(id_{\mathbb{P}^3}\times r_\mathbf{T})^*\mathbf{E}=\mathbf{E}_Z,\ \ \ \ \
r_\mathbf{T}^*\mathcal{O}_{\mathbf{T}/\mathcal{H}}(1)=\mathcal{O}_{Z(j)}.
\end{equation}

\end{remark}

\end{sub}

\vspace{5mm}

\begin{sub}{\bf A family of nets of quadrics $\mathbf{A}$ associated to the family $\mathbf{E}\to\mathbf{T}$.}
\label{Nets for T}
\rm

\vspace{5mm}

In this subsection we construct associated to $\mathbf{E}\to\mathbf{T}$ a family of nets of quadrics which will be used
below. For this we first note that, by (\ref{subschemes2}) and (\ref{ext Et}), we obtain the following equalities for
a sheaf $E_t$ in (\ref{ext Et}):\\
$
\dim\Ext^1(E_t(-4),\omega_{\mathbb{P}^3})=\dim\Ext^2(E_t,\omega_{\mathbb{P}^3})=4m-4,\ \
$
$
\dim\Ext^1(E_t(-3),\omega_{\mathbb{P}^3})=\dim\Ext^2(E_t(-1),\omega_{\mathbb{P}^3})=2m-1,\ \
$
$
\Ext^i(E_t,\omega_{\mathbb{P}^3})=\Ext^i(E_t(-1),\omega_{\mathbb{P}^3})=
$
$
\Ext^{3-i}(E_t(-3),\omega_{\mathbb{P}^3})=\Ext^{3-i}(E_t(-4),\omega_{\mathbb{P}^3})=0,\ i\ne2,
$
and
$
\Ext^i(E_t(-2),\omega_{\mathbb{P}^3})=0,\ i\ge0,\\
$
where $t\in\mathbf{T}$ is an arbitrary point and $\omega_{\mathbb{P}^3}=\mathcal{O}_{\mathbb{P}^3}(-4)$.
Therefore, applying the functor
$\Ext^i_\pi(-,\omega_\pi)$ to the sheaves
$\mathbf{E}(-j):=\mathbf{E}\otimes\mathcal{O}_{\mathbb{P}^3}(-j)\boxtimes\mathcal{O}_\mathbf{T},\ 0\le j\le4$, where
$\pi:\mathbb{P}^3\times\mathbf{T}\to\mathbf{T}$ is the projection,
the sheaf $\mathbf{E}$ is defined in (\ref{univ ext}) and
$\omega_\pi=\omega_{\mathbb{P}^3}\boxtimes\mathcal{O}_\mathbf{T}$,
and using base change for relative $\Ext$-sheaves we obtain that the sheaves
\begin{equation}\label{Def Fi}
\mathbb{F}_i:=\Ext^2_\pi(\mathbf{E}(-i),\omega_\pi),\ \ \
\mathbb{G}_i:=\Ext^1_\pi(\mathbf{E}(i-4),\omega_\pi),\ \ \ i=0,1,\ \ \
\end{equation}
are locally free $\mathcal{O}_\mathbf{T}$-sheaves of ranks, respectively,
\begin{equation}\label{rk Fi}
\rk\mathbb{F}_0=\rk\mathbb{G}_0=4m-4,\ \ \ \rk\mathbb{F}_1=\rk\mathbb{G}_1=2m-1,
\end{equation}and
\begin{equation}\label{Ext-sheaves vanish}
\Ext^i_\pi(\mathbf{E},\omega_\pi)=\Ext^i_\pi(\mathbf{E}(-1),\omega_\pi)=
\Ext^{3-i}_\pi(\mathbf{E}(-3),\omega_\pi)=\Ext^{3-i}_\pi(\mathbf{E}(-4),\omega_\pi)=0,\ i\ne2,
\end{equation}
$$
\Ext^i_\pi(\mathbf{E}(-2),\omega_\pi)=0,\ i\ge0,
$$
Similarly, we obtain that $\mathbb{H}:=R^1\pi_*(\mathbf{E}(-1))$ is a locally free $\mathcal{O}_\mathbf{T}$-sheaf
of rank
\begin{equation}\label{rk H}
\rk\mathbb{H}=2m-1.
\end{equation}
Using (\ref{ext Et}) we also see that the sheaf $\mathbb{H}$ duality commutes with the base change. Hence, there is a
relative Serre-Grothendieck duality isomorphism (see, e.g., \cite{K})
\begin{equation}\label{rel duality}
SD: \mathbb{F}_1\overset{\simeq}\to\mathbb{H}^\vee.
\end{equation}
Next, the local-to-relative spectral sequence $E^{p,q}_2=
R^p\pi_*\mathcal{E}xt^q_{\mathcal{O}_{\mathbb{P}^3\times\mathbf{T}}}(\mathbf{E}(-3),\omega_\pi)
\Rightarrow\Ext^{p+q}_\pi(\mathbf{E}(-3),\omega_\pi)$
gives an exact sequence
$0\to R^1\pi_*(\mathbf{E}^\vee(-1))\to\mathbb{G}_1\to
\pi_*\mathcal{E}xt^1_{\mathcal{O}_{\mathbb{P}^3\times\mathbf{T}}}(\mathbf{E}(-3),\omega_\pi)$,
where by (\ref{dim supp})
$\Supp\pi_*\mathcal{E}xt^1_{\mathcal{O}_{\mathbb{P}^3\times\mathbf{T}}}(\mathbf{E}(-3),\omega_\pi)
\subset\partial\mathbf{T}$. Since $\codim_\mathbf{T}\partial\mathbf{T}\ge1$, dualizing this sequence we obtain an
injective morphism of $\mathcal{O}_\mathbf{T}$-sheaves
\begin{equation}\label{morphism alpha}
0\to\mathbb{G}_1^\vee\overset{\alpha}\to(R^1\pi_*(\mathbf{E}^\vee(-1)))^\vee
\end{equation}
Next, dualizing the triple (\ref{univ ext}) and using the fact that
$\codim_{\mathbb{P}^3\times\mathbf{T}}\mathbf{L}_\mathbf{T}=2$
we obtain an exact sequence
\begin{equation}\label{dual to univ ext}
0\to\mathcal{O}_{\mathbb{P}^3}(-1)\boxtimes\mathcal{O}_{\mathbf{T}}
\to\mathbf{E}^\vee\to
\mathcal{O}_{\mathbb{P}^3}(1)\boxtimes\mathcal{O}_{\mathbf{T}/\mathcal{H}}(-1)\to
\mathcal{E}xt^2_{\mathcal{O}_{\mathbb{P}^3\times\mathbf{T}}}
(\mathcal{O}_{\mathbf{L}_\mathbf{T}}(1),\mathcal{O}_{\mathbb{P}^3\times\mathbf{T}})\to
\end{equation}
$$
\to\mathcal{E}xt^1_{\mathcal{O}_{\mathbb{P}^3\times\mathbf{T}}}(\mathbf{E},\mathcal{O}_{\mathbb{P}^3\times\mathbf{T}})\to0,
$$
so that $\det\mathbf{E}^\vee=\pi^*\mathcal{O}_{\mathbf{T}/\mathcal{H}}(-1)$. Hence, as $\mathbf{T}$ is a smooth integral
scheme, it follows by \cite[Prop. 1.10]{H1} that
$$
\mathbf{E}^{\vee\vee}\simeq\mathbf{E}^\vee\otimes(\det\mathbf{E}^\vee)^{-1}=
\mathbf{E}^\vee\otimes\pi^*\mathcal{O}_{\mathbf{T}/\mathcal{H}}(1).
$$
Dualizing (\ref{univ ext}) twice we see that the canonical morphism
$can:\mathbf{E}\to\mathbf{E}^{\vee\vee}\simeq\mathbf{E}^\vee\otimes\pi^*\mathcal{O}_{\mathbf{T}/\mathcal{H}}(1)$
is injective, and we obtain an exact sequence
$0\to\mathbf{E}(-1)\overset{can}\to\mathbf{E}(-1)^\vee\otimes\pi^*\mathcal{O}_{\mathbf{T}/\mathcal{H}}(1)\to\coker(can)\to0$,
where $\Supp\coker(can)\subset\mathbb{P}^3\times\partial\mathbf{T}$. Applying to this triple the functor
$R^i\pi_*$ and using the fact that $\mathbb{H}$ is locally free on $\mathbf{T}$,
we thus obtain an exact sequence
$0\to\mathbb{H}\overset{g}\to
R^1\pi_*(\mathbf{E}^\vee(-1))\otimes\mathcal{O}_{\mathbf{T}/\mathcal{H}}(1)\to\coker(g)\to0$,
where
$\Supp\coker(g)\subset\partial\mathbf{T}$. Dualizing this sequence we obtain an injective morphism of
$\mathcal{O}_\mathbf{T}$-sheaves
$\beta:(R^1\pi_*(\mathbf{E}^\vee(-1)))^\vee\to\mathbb{H}^\vee\otimes\mathcal{O}_{\mathbf{T}/\mathcal{H}}(1)$.
Composing it with the morphism $\alpha$ from (\ref{morphism alpha}) and the inverse of the relative duality
isomorphism $SD$ from (\ref{rel duality}) we obtain an injective morphism of locally free
$\mathcal{O}_\mathbf{T}$-sheaves
\begin{equation}\label{morphism gamma}
\gamma=SD^{-1}\circ\beta\circ\alpha:\mathbb{G}_1^\vee\to\mathbb{F}_1\otimes\mathcal{O}_{\mathbf{T}/\mathcal{H}}(1).
\end{equation}
In view of the property (iii) above (\ref{im f}) one easily sees that $\gamma$ is an isomorphism when restricted
onto $\mathbf{T}'$:
$$
\gamma|_{\mathbf{T}'}:\mathbb{G}_1^\vee|_{\mathbf{T}'}\overset{\simeq}\to
\mathbb{F}_1\otimes\mathcal{O}_{\mathbf{T}/\mathcal{H}}(1)|_{\mathbf{T}'}.
$$
(In fact, the restriction of $\gamma$ onto an arbitrary point $t\in\mathbf{T}'$ is just the Serre duality isomorphism
$H^2(E_t(-3))\overset{\simeq}\to H^1(E_t(-1))^\vee$ for a t'Hooft instanton $E_t$.)

Next, the resolution of the diagonal
$\Delta$ on $\mathbb{P}^3\times\mathbb{P}^3$ extends to a diagram of sheaves
\begin{equation}\label{res diag}
\xymatrix{
& & 0\ar[d]& 0\ar[d]& \\
& 0 \ar[d] & \mathcal{O}_{\mathbb{P}^3}(-4)\boxtimes \mathcal{O}_{\mathbb{P}^3}\ar@{=}[r]
\ar[d] & \mathcal{O}_{\mathbb{P}^3}(-4)\boxtimes \mathcal{O}_{\mathbb{P}^3} \ar[d] & \\
0\ar[r] & \mathcal{O}_{\mathbb{P}^3}(-3)\boxtimes
\mathcal{O}_{\mathbb{P}^3}(1)\ar[r]\ar[d] &
\mathcal{O}_{\mathbb{P}^3}(-3)\boxtimes \wedge^3V^\vee
\otimes \mathcal{O}_{\mathbb{P}^3}\ar[r]\ar[d] &
\mathcal{O}_{\mathbb{P}^3}(-3)\boxtimes T_{\mathbb{P}^3}(-1)\ar[r]\ar[d] & 0 \\
0\ar[r] & \mathcal{O}_{\mathbb{P}^3}(-2)\boxtimes T_{\mathbb{P}^3}(-2)
\ar[r]\ar[d] &
\mathcal{O}_{\mathbb{P}^3}(-2)\boxtimes\wedge^2V^\vee
\otimes \mathcal{O}_{\mathbb{P}^3}\ar[r]\ar[d] & \mathcal{O}_{\mathbb{P}^3}(-2)\boxtimes
\Omega_{\mathbb{P}^3}(2)\ar[r]\ar[d] & 0 \\
0\ar[r] & \mathcal{O}_{\mathbb{P}^3}(-1)\boxtimes
\Omega_{\mathbb{P}^3}(1)\ar[r]\ar[d] &
\mathcal{O}_{\mathbb{P}^3}(-1)\boxtimes V^\vee\otimes \mathcal{O}_{\mathbb{P}^3}\ar[r]\ar[d] &
\mathcal{O}_{\mathbb{P}^3}(-1)\boxtimes \mathcal{O}_{\mathbb{P}^3}(1)\ar[r]\ar[d] & 0 \\
0\ar[r] & \mathcal{I}_{\Delta,\mathbb{P}^3\times\mathbb{P}^3}\ar[r]\ar[d] &
\mathcal{O}_{\mathbb{P}^3}\boxtimes \mathcal{O}_{\mathbb{P}^3}\ar[r]\ar[d] &
\mathcal{O}_{\Delta}\ar[r]\ar[d] & 0 \\
&0&0&0&}
\end{equation}
Let $\boldsymbol{\rho}:\mathbb{P}^3\times\mathbf{T}\times\mathbb{P}^3\to\mathbb{P}^3\times\mathbb{P}^3$ and
$\boldsymbol{\pi}=\pi\times id_{\mathbb{P}^3}:
\mathbb{P}^3\times\mathbf{T}\times\mathbb{P}^3\to\mathbf{T}\times\mathbb{P}^3$
be the projections and denote
$\omega_{\boldsymbol{\pi}}=\omega_\pi\boxtimes\mathcal{O}_{\mathbb{P}^3}$.
Applying the functor $\Ext^i_{\boldsymbol{\pi}}(-,\omega_{\boldsymbol{\pi}})$ to the diagram
$\boldsymbol{\rho}^*$(\ref{res diag})$\otimes\mathbf{E}\boxtimes\mathcal{O}_{\mathbb{P}^3}$
and using (\ref{Def Fi}), (\ref{Ext-sheaves vanish}) and base change we obtain the commutative diagram of sheaves
on $\mathbb{P}^3\times\mathbf{T}\simeq\mathbf{T}\times\mathbb{P}^3$:
\begin{equation}\label{diag of monads}
\xymatrix{
&& & 0 & \ \ 0 & \\
&& 0 & \mathcal{O}_{\mathbb{P}^3}\boxtimes \mathbb{G}_0\ar[l]
\ar[u] & \mathbb{M} \ar[l]\ar[u] & \ \mathbf{E}^\vee\ar[l] & 0\ar[l] \\
&0 & \mathcal{O}_{\mathbb{P}^3}(1)\boxtimes
\mathbb{G}_1\ar[l]\ar[u] &
\mathcal{O}_{\mathbb{P}^3}\boxtimes V^\vee\otimes\mathbb{G}_1\ar[l]_-{\mathbf{e}}\ar[u] &
\Omega_{\mathbb{P}^3}(1)\boxtimes \mathbb{G}_1\ar[l]\ar[u] & 0\ar[l] \\
&0 & T_{\mathbb{P}^3}(-1)\boxtimes
\mathbb{F}_1\ar[l]\ar[u] &
\mathcal{O}_{\mathbb{P}^3}\boxtimes V\otimes \mathbb{F}_1\ar[l]\ar[u]_{id_\mathcal{O}\boxtimes\mathbb{A}'} &
\mathcal{O}_{\mathbb{P}^3}(-1)\boxtimes\mathbb{F}_1
\ar[l]_-{\mathbf{i}}\ar[u]_-{\mathbf{a}} & 0\ar[l]\\
0 &\mathbf{E}^\vee\ar[l] & \mathbb{K}\ar[l]\ar[u] &
\mathcal{O}_{\mathbb{P}^3}\boxtimes\mathbb{F}_0\ar[l]\ar[u] & 0\ar[l]\ar[u] \\
&&0\ar[u] & 0\ar[u]}
\end{equation}
where we denote
$\mathbb{K}=\Ext^2_{\boldsymbol{\pi}}(\boldsymbol{\rho}^*\mathcal{I}_{\Delta,\mathbb{P}^3\times\mathbb{P}^3}
\otimes\mathbf{E}\boxtimes\mathcal{O}_{\mathbb{P}^3},\omega_{\boldsymbol{\pi}})$,
$\mathbb{M}=\coker\mathbf{a}$ and where $\mathbb{A}'$ is a morphism $V\otimes\mathbb{F}_1\to
V^\vee\otimes\mathbb{G}_1$ given by this diagram.

Now set $\mathbb{W}:={\rm im}\mathbb{A}'$ and let
$\epsilon_{\mathbb{A}'}:V\otimes\mathbb{F}_1\twoheadrightarrow\mathbb{W}$,\
$i_{\mathbb{A}'}:\mathbb{W}\rightarrowtail V^\vee\otimes\mathbb{G}_1$,\
$\mathbf{g}:\mathcal{O}_{\mathbb{P}^3}\boxtimes \mathbb{W}\overset{id_\mathcal{O}\boxtimes i_{\mathbb{A}'}}
\longrightarrow
\mathcal{O}_{\mathbb{P}^3}\boxtimes V^\vee\otimes\mathbb{G}_1
\overset{\mathbf{e}}\longrightarrow\mathcal{O}_{\mathbb{P}^3}(1)\boxtimes\mathbb{G}_1$
and
$\mathbf{f}:\mathcal{O}_{\mathbb{P}^3}(-1)\boxtimes\mathbb{F}_1\overset{\mathbf{i}}
\longrightarrow\mathcal{O}_{\mathbb{P}^3}\boxtimes V\otimes\mathbb{F}_1
\overset{id_\mathcal{O}\boxtimes\epsilon_{\mathbb{A}'}}
\longrightarrow\mathcal{O}_{\mathbb{P}^3}\boxtimes \mathbb{W}$
be the induced morphisms. From (\ref{rk Fi}) and the middle vertical sequence in (\ref{diag of monads}) it follows that
$\mathbb{W}$ is a locally free $\mathcal{O}_\mathbf{T}$-sheaf of rank $4m$:
\begin{equation}\label{rk W}
\rk\mathbb{W}=4m.
\end{equation}
Moreover, the diagram (\ref{diag of monads})
gives the monad with the cohomology sheaf $\mathbf{E}^\vee$:
\begin{equation}\label{monad for bfEdual}
0\to\mathcal{O}_{\mathbb{P}^3}(-1)\boxtimes\mathbb{F}_1\overset{\mathbf{f}}\to
\mathcal{O}_{\mathbb{P}^3}\boxtimes \mathbb{W}\overset{\mathbf{g}}
\longrightarrow\mathcal{O}_{\mathbb{P}^3}(1)\boxtimes\mathbb{G}_1\to0,\ \ \
\mathbf{E}^\vee=\ker\mathbf{g}/{\rm im}\mathbf{f}.
\end{equation}
\begin{remark}
One can, of course, obtain the monad (\ref{monad for bfEdual}) from the Beilinson spectral sequence with $E_1$-term
$E^{p,q}_1=\Ext^{3-q}_\pi(\mathbf{E}\otimes\Omega^{-p}_{\mathbb{P}^3}(-p)\boxtimes\mathcal{O}_{\mathbf{T}},\omega_\pi)$
(cf. \cite[Ch. II, 3.1.4]{OSS}). However, we use here the diagram (\ref{diag of monads}) because it will be also used
below in producing the monad (\ref{complex for E dual}) and Lemma \ref{T0=T'}.
\end{remark}

Next, from the definition of the morphisms $\mathbf{f},\mathbf{g}$ and $\gamma$ follows the diagram
\begin{equation}\label{diag f,g,gamma}
\xymatrix{
\mathcal{O}_{\mathbb{P}^3}(-1)\boxtimes\mathbb{G}^\vee_1\otimes\mathcal{O}_{\mathbf{T}/\mathcal{H}}(-1)
\ar[r]^{\mathbf{e}^\vee}\ar[d]_\gamma &
\mathcal{O}_{\mathbb{P}^3}\boxtimes V\otimes\mathbb{G}^\vee_1\otimes\mathcal{O}_{\mathbf{T}/\mathcal{H}}(-1)
\ar[d]_\gamma \\
\mathcal{O}_{\mathbb{P}^3}(-1)\boxtimes\mathbb{F}_1\ar[r]^{\mathbf{i}} &
\mathcal{O}_{\mathbb{P}^3}\boxtimes V\otimes\mathbb{F}_1}.
\end{equation}
is commutative. Thus, the composition
$\mathbb{A}:V\otimes\mathbb{G}^\vee_1\otimes\mathcal{O}_{\mathbf{T}/\mathcal{H}}(-1)
\overset{\gamma}\to V\otimes\mathbb{F}_1\overset{\mathbb{A}'}\to V^\vee\otimes\mathbb{G}_1$
fits in the (left- and right-exact) complex of sheaves
\begin{equation}\label{complex A}
0\to\mathcal{O}_{\mathbb{P}^3}(-1)\boxtimes\mathbb{G}^\vee_1\otimes\mathcal{O}_{\mathbf{T}/\mathcal{H}}(-1)
\overset{\mathbf{e}^\vee}\to
\mathcal{O}_{\mathbb{P}^3}\boxtimes V\otimes\mathbb{G}^\vee_1\otimes\mathcal{O}_{\mathbf{T}/\mathcal{H}}(-1)
\overset{id_\mathcal{O}\boxtimes\mathbb{A}}\to
\end{equation}
$$
\to\mathcal{O}_{\mathbb{P}^3}\boxtimes V^\vee\otimes\mathbb{G}_1
\overset{\mathbf{e}}\to\mathcal{O}_{\mathbb{P}^3}(1)\boxtimes\mathbb{G}_1\to0
$$
and ${\rm im\mathbb{A}}\subset\mathbb{W}$. In addition, by construction for any $t\in \mathbf{T}'$ the homomorphism
$\mathbb{A}\otimes\mathbf{k}_t$ in view of Serre duality $H:=H^2(E_t(-3))\overset{\simeq}\to H^1(E_t(-1))$
coincides with the skew-symmetric middle vertical homomorphism $A:V\otimes H\to V^\vee\otimes H^\vee$ in
(\ref{qA}) for $E=E_t$ and $n=2m-1$. Hence, $\mathbb{A}$ is skew-symmetric,
$\mathbb{A}\in H^0(\wedge^2(V^\vee\otimes \mathbb{G}_1)\otimes\mathcal{O}_{\mathbf{T}/\mathcal{H}}(1))$.
We thus obtain
the induced skew-symmetric morphism
$\mathbf{q}:\mathbb{W}^\vee\otimes\mathcal{O}_{\mathbf{T}/\mathcal{H}}(-1))\to\mathbb{W}$
which yields a decomposition of $\mathbb{A}$
as
$\mathbb{A}=i_{\mathbb{A}'}\circ\mathbf{q}\circ i_{\mathbb{A}'}^\vee$. This decomposition, being restricted onto an
arbitrary point $t\in\mathbf{T}'$,
gives the rightmost square in (\ref{qA}). In particular, it follows that
\begin{equation}\label{family bbA}
\mathbb{A}\in H^0(\wedge^2V^\vee\otimes S^2\mathbb{G}_1)\otimes\mathcal{O}_{\mathbf{T}/\mathcal{H}}(1)),
\end{equation}
and that $\mathbf{q}|_{\mathbf{T}'}$ is an isomorphism. We thus consider the dense open subset
$\mathbf{T}_0$ of $\mathbf{T}$ containing $\mathbf{T}'$ which is defined as
\begin{equation}\label{T0}
\mathbf{T}_0:=\{t\in\mathbf{T}\ |\ {\mathbf{q}}|_{\mathbb{P}^3\times\{t\}}:
\mathbb{W}^\vee\otimes\mathcal{O}_{\mathbf{T}/\mathcal{H}}(-1)\otimes\mathbf{k}_t\to\mathbb{W}\otimes\mathbf{k}_t\
{\rm is\ an\ isomorphism}\},
\end{equation}
$$
\mathbf{T}_0\supset \mathbf{T}'.
$$
Denote
\begin{equation}\label{E0vee}
\mathbf{W}:=\mathbb{W}^\vee,\ \ \mathbf{W}_0:=\mathbf{W}|_{\mathbf{T}_0},\ \ \mathbf{q}_0:=\mathbf{q}|_{\mathbf{T}_0},
\ \ \boldsymbol{\mathcal{L}}:=\mathcal{O}_{\mathbf{T}/\mathcal{H}}(-1),\ \
\boldsymbol{\mathcal{L}}_0:=\boldsymbol{\mathcal{L}}|_{\mathbf{T}_0},\ \
\mathbf{E}_0=\mathbf{E}|_{\mathbf{T}_0},\ \
\end{equation}
$$
\mathbf{g}_0:=\mathbf{g}^\vee|_{\mathbf{T}_0},\ \
\mathbf{G}:=\mathbb{G}_1^\vee,\ \ \mathbf{G}_0=\mathbf{G}|_{\mathbf{T}_0}.
$$
In this notation the complex (\ref{complex A}) induces the following right- and left-exact complex
\begin{equation}\label{last complex}
0\to\mathcal{O}_{\mathbb{P}^3}(-1)\boxtimes\mathbf{G}\otimes\boldsymbol{\mathcal{L}}
\overset{\mathbf{g}}\to
\mathcal{O}_{\mathbb{P}^3}\boxtimes\mathbf{W}\otimes\boldsymbol{\mathcal{L}}
\overset{\mathbf{q}}\to\mathcal{O}_{\mathbb{P}^3}\boxtimes\mathbf{W}^\vee
\overset{\mathbf{g}^\vee}\to\mathcal{O}_{\mathbb{P}^3}(1)\boxtimes\mathbf{G}^\vee\to0,
\end{equation}
Standard diagram chasing with (\ref{diag of monads})-(\ref{complex A}) shows that the restriction of the monad
(\ref{monad for bfEdual}) onto $\mathbb{P}^3\times\mathbf{T}_0$
coincides with the restriction onto $\mathbb{P}^3\times\mathbf{T}_0$ of the complex (\ref{last complex})
and is isomorphic to a (antiselfdual) monad
\begin{equation}\label{complex for E dual}
0\to\mathcal{O}_{\mathbb{P}^3}(-1)\boxtimes\mathbf{G}_0\otimes\boldsymbol{\mathcal{L}}_0
\overset{\mathbf{g}_0}\to
\mathcal{O}_{\mathbb{P}^3}\boxtimes\mathbf{W}_0\otimes\boldsymbol{\mathcal{L}}_0
\overset{\mathbf{q}_0}{\underset{\simeq}\to}\mathcal{O}_{\mathbb{P}^3}\boxtimes\mathbf{W}_0^\vee
\overset{\mathbf{g}_0^\vee}\to\mathcal{O}_{\mathbb{P}^3}(1)\boxtimes\mathbf{G}_0^\vee\to0,
\end{equation}
$$
\mathbf{E}_0^\vee=\ker\mathbf{g}_0^\vee/{\rm im}(\mathbf{q}_0\circ\mathbf{g}_0).
$$

From this monad and (\ref{T0}) immediately follows

\begin{lemma}\label{T0=T'}
$\mathbf{E}_0$ is a locally free
$\mathcal{O}_{\mathbb{P}^3\times\mathbf{T}_0}$-sheaf, i.e. $\mathbf{T}_0=\mathbf{T}'$.
\end{lemma}

Consider the variety
$\mathbf{Y}:=\mathbf{P}{\rm roj}(\mathcal{H}om(\mathbf{G},H_{2m-1}\otimes\mathcal{O}_\mathbf{T}))$
with the projection $p_\mathbf{Y}:\mathbf{Y}\to\mathbf{T}$
and set
$\mathbf{G}_\mathbf{Y}:=p_\mathbf{Y}^*\mathbf{G},\ \ \ \boldsymbol{\mathcal{L}}_\mathbf{Y}:=
p_\mathbf{Y}^*\boldsymbol{\mathcal{L}}\otimes\mathcal{O}_{\mathbf{Y}/\mathbf{T}}(-1)$.
The universal morphism
\begin{equation}\label{tau}
\tau:\ H_{2m-1}\otimes\mathcal{O}_\mathbf{Y}\otimes\mathcal{O}_{\mathbf{Y}/\mathbf{T}}(-1)
\to\mathbf{G}_\mathbf{Y}
\end{equation}
on $\mathbf{Y}$ together with the family
$p_\mathbf{Y}^*\mathbb{A}:\mathbf{G}_\mathbf{Y}\otimes V\otimes\boldsymbol{\mathcal{L}}_\mathbf{Y}\to
\mathbf{G}_\mathbf{Y}^\vee\otimes V^\vee$
yields a family of nets of quadrics
$\mathbf{A}:H_{2m-1}\otimes V\otimes\boldsymbol{\mathcal{L}}_\mathbf{Y}\to
H_{2m-1}^\vee\otimes V^\vee\otimes\mathcal{O}_\mathbf{Y}$, i.e., equivalently, the morphism
\begin{equation}\label{family bfA}
\mathbf{A}:\boldsymbol{\mathcal{L}}_\mathbf{Y}\to S^2H_{2m-1}^\vee\otimes\wedge^2V^\vee\otimes\mathcal{O}_\mathbf{Y}=
\mathbf{S}_{2m-1}\otimes\mathcal{O}_\mathbf{Y}.
\end{equation}
We call {\it $\mathbf{A}$ the family of nets of quadrics associated to the family} $\mathbf{E}\to\mathbf{T}$.

Now consider the principal $PGL(H_{2m-1})$-bundle
$p_{\mathbf{Y}_0}:\mathbf{Y}_0:=
P(\mathcal{I}som(H_{2m-1}\otimes\mathcal{O}_{\mathbf{T}_0},\mathbf{G}_0))\to\mathbf{T}_0$
together with the natural open embedding
$\mathbf{Y}_0\overset{i_0}\hookrightarrow\mathbf{Y}$ such that $p_{\mathbf{Y}_0}=p_\mathbf{Y}\circ i_0$
and set
$\mathbf{A}_0:=\mathbf{A}|_{\mathbf{Y}_0},\
\boldsymbol{\mathcal{L}}_{\mathbf{Y}_0}:=\boldsymbol{\mathcal{L}}_\mathbf{Y}|_{\mathbf{Y}_0},\
\mathbf{W}_{\mathbf{Y}_0}:=p_{\mathbf{Y}_0}^*\mathbf{W}_0$.
The monad $p_{\mathbf{Y}_0}^*$(\ref{complex for E dual}):
\begin{equation}\label{lifted monad}
0\to\mathcal{O}_{\mathbb{P}^3}(-1)\boxtimes H_{2m-1}\otimes\boldsymbol{\mathcal{L}}_{\mathbf{Y}_0}\to
\mathcal{O}_{\mathbb{P}^3}\boxtimes\mathbf{W}_{\mathbf{Y}_0}\otimes\boldsymbol{\mathcal{L}}_{\mathbf{Y}_0}
{\overset{\simeq}\to}\mathcal{O}_{\mathbb{P}^3}\boxtimes\mathbf{W}_{\mathbf{Y}_0}^\vee
\to\mathcal{O}_{\mathbb{P}^3}(1)\boxtimes H_{2m-1}^\vee\otimes\mathcal{O}_{\mathbf{Y}_0}\to0,
\end{equation}

Now pick a monomorphism $j:H_{m-1}\hookrightarrow H_m$ and let $\widetilde{Z}$ be any irreducible component of
$Z_m$. Assume that $\widetilde{Z}(j)$ is nonempty, hence dense in $\widetilde{Z}$
according to Lemma \ref{Z*(j)} (in particular, such $j$ exists for $\widetilde{Z}=Z$ by the same Lemma). Consider
the morphism
$r_\mathbf{T}:\widetilde{Z}(j)\to\mathbf{T}'$ defined in (\ref{rZ}).
Note that from the definition (\ref{Def Fi}) of the locally free $\mathcal{O}_\mathbf{T}$-sheaf
$\mathbb{G}_1=\Ext^1_\pi(\mathbf{E}(-3),\omega_\pi)$ it
follows that the formation of $\mathbb{G}_1^\vee$ commutes with the base change. In particular, the definition
(\ref{rZ}) of the morphism $r_Z$ and the definition (\ref{fZ}) of the sheaf $\mathbf{G}_Z$ imply that
$\mathbf{G}_Z=r_\mathbf{T}^*\mathbb{G}_1^\vee$.
Hence the isomorphism (\ref{fZ}) gives a subbundle morphism
\begin{equation}\label{iZ}
i_Z:\mathcal{O}_{\widetilde{Z}(j)}\to\mathcal{H}om(H_{2m-1}\otimes\mathcal{O}_{\widetilde{Z}^*(j)},\mathbf{G}_Z)=
r_\mathbf{T}^*\mathcal{H}om(H_{2m-1}\otimes\mathcal{O}_{\mathbf{T}},\mathbb{G}_1^\vee),\ \ \
\end{equation}
$$
{\rm im}i_Z\subset\mathcal{I}som(H_{2m-1}\otimes\mathcal{O}_{\mathbf{T}_0},\mathbf{G}_0).
$$
Now the well known universal property of $\mathbf{P}$roj (see \cite[Ch. III, Prop. 7.12]{H}) and the last inclusion in
(\ref{iZ}) show that the morphism
$r_\mathbf{T}:\widetilde{Z}(j)\to\mathbf{T}'=\mathbf{T}_0$ (here we use Lemma \ref{T0=T'}) lifts to the morphism
$r_\mathbf{Y}:\widetilde{Z}(j)\to\mathbf{Y}_0$ giving
the factorization of $r_\mathbf{T}$:
\begin{equation}\label{rT=}
r_\mathbf{T}:\widetilde{Z}(j)\overset{r_\mathbf{Y}}\to\mathbf{Y}_0\overset{p_{\mathbf{Y}_0}}\to\mathbf{T}_0
\end{equation}
such that
\begin{equation}\label{tilde bfA=}
\widetilde{\mathbf{A}}_Z=r_\mathbf{Y}^*\mathbf{A},
\end{equation}
where
$\widetilde{\mathbf{A}}_Z:\mathcal{O}_{\widetilde{Z}(j)}\to\mathbf{S}_{2m-1}\otimes\mathcal{O}_{\widetilde{Z}(j)}$
is the family of nets of quadrics (\ref{tilde AZ}) and $\mathbf{A}$ is the net (\ref{family bfA}).
Moreover, consider the total space
$\mathbf{V}={\rm Spec}(S^\cdot_{\mathcal{O}_\mathbf{Y}}\boldsymbol{\mathcal{L}}_\mathbf{Y}^{-1})$
of the vector bundle $\boldsymbol{\mathcal{L}}_\mathbf{Y}$ and let
$\mathbf{V}_0=V\smallsetminus\{0$-section$\}$ be the complement of the 0-section in $\mathbf{V}$,
with the projection $\rho:\mathbf{V}_0\to\mathbf{Y}$. The morphism
$r_\mathbf{Y}:\widetilde{Z}(j)\to\mathbf{Y}_0$ naturally lifts to a morphism
$r_\mathbf{V}:\widetilde{Z}(j)\to\mathbf{V}_0$, i.e. $r_\mathbf{Y}$ factorizes as
$r_\mathbf{Y}=\rho\circ r_\mathbf{V}$:
\begin{equation}\label{rY=}
\xymatrix{
\widetilde{Z}(j)\ar[d]_-{r_\mathbf{Y}}\ar[r]^{r_\mathbf{V}} & \mathbf{V}_0\ar[d]^{\rho}\\
\mathbf{Y}_0\ar@{^(->}[r] & \mathbf{Y}.}
\end{equation}
so that, by (\ref{tilde bfA=}),
\begin{equation}\label{2tilde bfA=}
\widetilde{\mathbf{A}}_Z=r_\mathbf{V}^*\rho^*\mathbf{A}.
\end{equation}
Next, there is a well defined morphism
$\mu:\mathbf{V}_0\to\mathbf{S}_{2m-1}:v\mapsto(\rho^*\mathbf{A})(\mathbf{s}(v))$ where $\mathbf{s}$ is the canonical
section of
$\rho^*\boldsymbol{\mathcal{L}}_\mathbf{Y}\simeq\mathcal{O}_{\mathbf{V}_0}$.
Now (\ref{2tilde bfA=}) means that $\tilde{\lambda}_j=\mu\circ r_\mathbf{V}$:
\begin{equation}\label{lambda'j=}
\tilde{\lambda}_j:\widetilde{Z}(j)\overset{r_\mathbf{V}}\to\mathbf{V}_0\overset{\mu}\to\mathbf{S}_{2m-1}
\end{equation}
where the morphism $\tilde{\lambda}_j:\widetilde{Z}\to\mathbf{S}_{2m-1}$ is defined in Lemma
\ref{fibre of lambda}(ii).
\begin{remark}\label{rV}
By definition, the morphism $r_\mathbf{V}$ considered in the diagram above is well defined as the morphism
$r_\mathbf{V}:Z_m(j)\to\mathbf{V}$.
\end{remark}

\end{sub}

\vspace{0.2cm}

\begin{sub}{\bf Irreducibility of $Z_m$.}\label{final arguments}
\rm

\vspace{2mm}

Take an arbitrary point $z_0=(D_0,\phi_0)\in Z$ with $\phi_0\ne0$.
According to Lemma \ref{Z*(j)}(i) there exists a monomorphism $j:H_{m-1}\hookrightarrow H_m$ such that $Z(j)$ is
a dense open subset of $Z$.
Hence there exists a smooth affine curve $C$ with a marked point $0\in C$ and a morphism $g:C\to Z$ such that
$g(0)=z_0$ and $g(C^*)\subset Z(j)$ where $C^*:=C\smallsetminus\{0\}$. For any $x\in C$ set $(D_x,\phi_x):=g(x)$.
Here, for all $x\in C$, by definition $A_1(x):=D_x^{-1}$ is an isomorphism
$H_m\otimes V\overset{\simeq}\to H_m^\vee\otimes V^\vee$ and also $A_2(x):=\phi_x\circ j$ is a homomorphism
$H_{m-1}\otimes V\overset{\simeq}\to H_m^\vee\otimes V^\vee$.
Hence, picking an isomorphism $\xi:H_m\oplus H_{m-1}\overset{\simeq}\to H_{2m-1}$, we may consider the matrix
$A(x)=\left(
\begin{array}{cc}
A_1(x) & A_2(x) \\
-A_2(x)^\vee & A_3(x)
\end{array}
\right)$
with $A_3(x)=-A_2(x)^\vee\circ A_1(x)^{-1}\circ A_2(x)$ as a homomorhism (net of quadrics)
$A(x): H_{2m-1}\otimes V\to H_{2m-1}^\vee\otimes V^\vee$
of rank
\begin{equation}\label{rkA=4m}
\rk A(x)=\rk A_1(x)=4m,\ \ \ x\in C.
\end{equation}
We thus have a family of nets of quadrics
$\mathbf{A}_C=\{A(x)\}_{x\in C}$
and its restriction
$\mathbf{A}_{C^*}=\{A(x)\}_{x\in C}|{C^*}$.
\footnote{Equivalently,  using Lemma \ref{E2(z) is tHooft}(iii),
one can define $A(x)$ as
$\lambda_j(g(x)),\ x\in C$.}

Consider the composition
$r_\mathbf{Y}\circ g:C^*\to\mathbf{Y}_0\hookrightarrow\mathbf{Y}$.
Since $Y$ is projective, this morphism extends to the morphism
$\psi_\mathbf{Y}:C\to\mathbf{Y}$ such that
$\mathbf{A}_C=\psi_\mathbf{Y}^*\mathbf{A}$.
As $A(0)\ne0$, it follows that $\psi_\mathbf{Y}$ lifts to the morphism
$\psi_\mathbf{V}:C\to\mathbf{V}_0$
such that $\psi_\mathbf{Y}=\rho\circ\psi_\mathbf{V}$.
We also have the composition
$\psi_\mathbf{T}=p_\mathbf{Y}\circ\psi_\mathbf{Y}:C\to\mathbf{T}$ and the commutative diagram
\begin{equation}\label{AC}
\xymatrix{
H_{2m-1}\otimes V\otimes\mathcal{O}_C\ar[d]_-{\tau_C}\ar[r]^-{\mathbf{A}_C} &
H_{2m-1}^\vee\otimes V^\vee\otimes\mathcal{O}_C\\
\mathbf{G}_C\otimes V\ar[r]^-{\psi_\mathbf{Y}^*\mathbb{A}} &
\mathbf{G}_C^\vee\otimes V^\vee\ar[u]_-{\tau_C^\vee}}
\end{equation}
where $\mathbf{G}_C:=\psi_\mathbf{Y}^*\mathbf{G}_\mathbf{Y}$,
$\tau_C:=\psi_{\mathbf{Y}}^*\tau$ and $\tau$ is the universal morphism (\ref{tau}).
Consider the $\mathcal{O}_C$-sheaves
$\mathbf{W}_C=H_{2m-1}\otimes V\otimes\mathcal{O}_C/\ker\mathbf{A}_C$
and
$\mathbb{W}_C=\mathbf{G}_C\otimes V/\ker\mathbf{A}_C$
and the morphisms
$\mathbf{e}_C:H_{2m-1}\otimes V\otimes\mathcal{O}_C\twoheadrightarrow\mathbf{W}_C$,
$e_C:\mathbf{G}_C\otimes V\twoheadrightarrow\mathbb{W}_C$,
$\mathbf{q}_C:\mathbf{W}_C\to\mathbf{W}_C^\vee$,
$q_C:\mathbb{W}_C\to\mathbb{W}_C^\vee$
and
$\epsilon:\mathbf{W}_C\to\mathbb{W}_C$
induced by the diagram (\ref{AC}), so that
\begin{equation}\label{relation1 on C}
\mathbf{q}_C=\epsilon^\vee\circ q_C\circ\epsilon,
\end{equation}
\begin{equation}\label{relation2 on C}
\epsilon\circ\mathbf{e}_C=e_C\circ\tau_C.
\end{equation}
The condition (\ref{rkA=4m}) means that $\mathbf{W}_C$ is a locally free rank-$4m$  $\mathcal{O}_C$-sheaf and
$\mathbf{q}_C$ is an isomorphism. Hence (\ref{relation1 on C}) implies that  $\mathbb{W}_C$ is a locally free
rank-$4m$  $\mathcal{O}_C$-sheaf and $q_C$ is an isomorphism. This together with Lemma \ref{T0=T'}
precisely means that
\begin{equation}\label{psi(C) in T0}
\psi_\mathbf{Y}(C)\subset\mathbf{Y}_0,\ \ \ {\rm resp.,}\ \ \ \psi_\mathbf{T}(C)\subset\mathbf{T}_0.
\end{equation}
Consider the compositions
$\mathbf{a}_C:\mathcal{O}_{\mathbb{P}^3}(-1)\boxtimes H_{2m-1}\otimes\mathcal{O}_C\to
V\otimes\mathcal{O}_{\mathbb{P}^3}\boxtimes H_{2m-1}\otimes\mathcal{O}_C
\overset{\mathbf{e}_C}\to\mathcal{O}_{\mathbb{P}^3}\boxtimes\mathbf{W}_C$
and
$a_C:\mathcal{O}_{\mathbb{P}^3}(-1)\boxtimes\mathbf{G}_C\to
V\otimes\mathcal{O}_{\mathbb{P}^3}\boxtimes\mathbf{G}_C
\overset{e_C}\to\mathcal{O}_{\mathbb{P}^3}\boxtimes\mathbb{W}_C$
and the diagram of induced complexes
\begin{equation}\label{2complexes}
\xymatrix{0\ar[r] &
\mathcal{O}_{\mathbb{P}^3}(-1)\boxtimes H_{2m-1}\otimes\mathcal{O}_C\ar[r]^-{\mathbf{a}_C}\ar[d]_-{\tau_C} &
\mathcal{O}_{\mathbb{P}^3}\boxtimes\mathbf{W}_C
\ar[r]^-{{}^t\mathbf{a}_C}\ar[d]^-{\epsilon}_-{\simeq} &
\mathcal{O}_{\mathbb{P}^3}(1)\boxtimes H_{2m-1}^\vee\otimes\mathcal{O}_C\ar[r]& 0 \\
0\ar[r] & \mathcal{O}_{\mathbb{P}^3}(-1)\boxtimes\mathbf{G}_C\ar[r]^-{a_C}
& \mathcal{O}_{\mathbb{P}^3}\boxtimes\mathbb{W}_C\ar[r]^-{{}^ta_C} &
\mathcal{O}_{\mathbb{P}^3}(1)\boxtimes\mathbf{G}_C^\vee\ar[r]\ar[u]_-{\tau_C^\vee} & 0.}
\end{equation}
From (\ref{psi(C) in T0}) it follows now that the lower complex in this diagram is a genuine monad which is by
construction obtained by applying the functor $(id_{\mathbb{P}^3}\times\psi_\mathbf{T})^*$ to the monad
(\ref{complex for E dual}). In particular, its cohomology sheaf $\mathbb{E}_C$ is a rank-2 bundle. Also,
by construction, these two complexes are
isomorphic over $C^*$. However, the upper complex is apriori not right- and left-exact when restricted to
$\mathbb{P}^3\times\{0\}$. We are going to show that, in fact, it is isomorphic to the lower monad, hence it is
left- and right-exact, i.e. it is a monad.

For this, consider the monomorphism $\mathbf{i}_m:H_m\rightarrowtail H_{2m-1}$ given by the isomorphism $\xi$ above,
let
$\boldsymbol{\alpha}:
H_m\otimes V\otimes\mathcal{O}_C\rightarrowtail H_{2m-1}\otimes V\otimes\mathcal{O}_C\twoheadrightarrow\mathbf{W}_C$,
$\beta:H_m\otimes\mathcal{O}_C\rightarrowtail H_{2m-1}\otimes\mathcal{O}_C\overset{\tau_C}\to\mathbf{G}_C$
be the induced morphisms and set $\mathbf{G}_m:={\rm im}\beta$, $\mathbf{G}_{m-1}:=\mathbf{G}_C/\mathbf{G}_m$.
From (\ref{rkA=4m}) it follows that $\alpha$ is an isomorphism and, respectively, the induced morphism
$\alpha:\mathbf{G}_m\otimes V\to\mathbb{W}_C$ is an isomorphism.
Hence by (\ref{AC})-(\ref{relation2 on C})
$\beta$ is injective, $\mathbf{G}_m$
is a locally free rank-$m$ $\mathcal{O}_C$-sheaf, the morphism
$\mathbf{G}_m\rightarrowtail\mathbf{G}_C$ is a subbundle morphism,
hence $\mathbf{G}_{m-1}$
is a locally free rank-$(m-1)$ $\mathcal{O}_C$-sheaf.
We now have the induced diagram of isomorphic monads obtained similar to (\ref{2complexes}):
\begin{equation}\label{2subcomplexes}
\xymatrix{0\ar[r] &
\mathcal{O}_{\mathbb{P}^3}(-1)\boxtimes H_m\otimes\mathcal{O}_C
\ar[r]^-{\boldsymbol{\alpha}_C}\ar[d]_-{\simeq}^-{\beta_C} &
\mathcal{O}_{\mathbb{P}^3}\boxtimes\mathbf{W}_C
\ar[r]^-{{}^t\boldsymbol{\alpha}_C}\ar[d]^-{\epsilon}_-{\simeq} &
\mathcal{O}_{\mathbb{P}^3}(1)\boxtimes H_{2m-1}^\vee\otimes\mathcal{O}_C\ar[r]& 0 \\
0\ar[r] & \mathcal{O}_{\mathbb{P}^3}(-1)\boxtimes\mathbf{G}_m\ar[r]^-{\alpha_C}
& \mathcal{O}_{\mathbb{P}^3}\boxtimes\mathbb{W}_C\ar[r]^-{{}^t\alpha_C} &
\mathcal{O}_{\mathbb{P}^3}(1)\boxtimes\mathbf{G}_m^\vee\ar[r]\ar[u]^-{\simeq}_-{\beta_C^\vee} & 0.}
\end{equation}
with the isomorphism $\delta:\mathbf{E}_{2m}\overset{\sim}\to\mathbb{E}_{2m}$ of the rank-$2m$ cohomology sheaves of
these monads. (Note that, by construction, $\mathbf{E}_{2m}=\underset{x\in C}\cup E_{2m}(D_x^{-1})$.)
In addition, the diagram of natural morphisms
$$
\xymatrix{0\ar[r] &
H_m\otimes V\otimes\mathcal{O}_C
\ar[r]^-{\mathbf{i}_m}\ar[d]_-{\simeq}^-{\beta} &
H_{2m-1}\otimes V\otimes\mathcal{O}_C
\ar[r]\ar[d]^-{\tau_C} &
H_{m-1}\otimes V\otimes\mathcal{O}_C\ar[r]\ar[d]^\gamma & 0 \\
0\ar[r] & \mathbf{G}_m\otimes V\ar[r]^-{i_m}
& \mathbf{G}_{2m-1}\otimes V\ar[r] &
\mathbf{G}_{m-1}\otimes V\ar[r] & 0.}\\
$$
satisfying the relations $\boldsymbol{\alpha}=\mathbf{e}_C\circ \mathbf{i}_m,\ \alpha=e_C\circ i_m,\
\alpha\circ\beta=\epsilon\circ\boldsymbol{\alpha}$, together with the diagrams (\ref{2complexes})-(\ref{2subcomplexes}),
yields a diagram of factor-complexes
\begin{equation}\label{2factorcomplexes}
\xymatrix{0\ar[r] &
\mathcal{O}_{\mathbb{P}^3}(-1)\boxtimes H_{m-1}\otimes\mathcal{O}_C
\ar[r]^-{\bar{\boldsymbol{\alpha}}_C}\ar[d]_-{\gamma_C} &
\mathbf{E}_{2m}
\ar[r]\ar[d]^-{\delta}_-{\simeq} &
\mathcal{O}_{\mathbb{P}^3}(1)\boxtimes H_{m-1}^\vee\otimes\mathcal{O}_C\ar[r]& 0 \\
0\ar[r] & \mathcal{O}_{\mathbb{P}^3}(-1)\boxtimes\mathbf{G}_{m-1}\ar[r]
&\mathbb{E}_{2m}\ar[r] &
\mathcal{O}_{\mathbb{P}^3}(1)\boxtimes\mathbf{G}_{m-1}^\vee\ar[r]\ar[u]_-{\gamma_C^\vee} & 0}
\end{equation}
where $\bar{\boldsymbol{\alpha}}_C$ is the induced morphism.
By the above, this diagram becomes an isomorphism of monads when restricted onto $\mathbb{P}^3\times C^*$.
To show that it is an isomorphism everywhere, it is enough to show that
$\gamma_C\otimes\mathbf{k}(0):H_{m-1}\to\mathbf{G}_{m-1}\otimes\mathbf{k}(0)$ is an isomorphism. Passing to sections in
the left square of the diagram
(\ref{2factorcomplexes})$\otimes\mathcal{O}_{\mathbb{P}^3}(-1)\boxtimes\mathcal{O}_C$, we see that this condition is
equivalent to the injectivity of homomorphism of sections
$h^0(\bar{\boldsymbol{\alpha}}_C\otimes\mathbf{k}(0)):H_{m-1}\to H^0(E_{2m}(D_0^{-1})(1))$. But this homomorphism exactly
coincides with the composition
$$
s_{z_0}(j):H_{m-1}\overset{j}\hookrightarrow H_m\overset{s(z_0)={}^\sharp\phi_0}\longrightarrow
H^0(E_{2m}(D_0^{-1})(1)).
$$
Now from the definition of the subset $R_{Z}$ of $Z$ defined in Lemma \ref{codim RZ*} it follows that the
injectivity of the map $s_{z_0}(j)$ is true for any point $z_0\in Z\smallsetminus R_{Z}$ and a generic
monomorphism $j:H_{m-1}\hookrightarrow H_m$. Hence, for such point
$z_0=(D_0,\phi_0)$ the restriction of the upper complex in (\ref{2factorcomplexes}) onto $\mathbb{P}^3\times\{0\}$ is a
monad:
$
0\to H_{m-1}\otimes\mathcal{O}_{\mathbb{P}^3}(-1)\overset{s_{z_0}(j)}\to E_{2m}(D_0^{-1})
\overset{{}^ts_{z_0}(j)}\to H_{m-1}^\vee\otimes\mathcal{O}_{\mathbb{P}^3}(1)\to0,
$
which by definition coincides with the monad (\ref{E2-monad}) for $z=z_0$. (As a corollary we obtain that the diagrams
(\ref{2complexes}) and (\ref{2factorcomplexes}) are the diagrams of isomorphisms of monads for this $z_0$.)
In other words, $z_0\in\widehat{Z}(j)$ where the set $\widehat{Z}(j)$ was defined in Lemma \ref{Z*(j)}(i).

We thus have proved the following statement.
\begin{proposition}\label{hat Z(j)}
For any point $z\in Z\smallsetminus R_{Z}$ there exists a monomorphism $j:H_{m-1}\hookrightarrow H_m$ such that
$z\in Z(j,\mathbf{I})$.
\end{proposition}

Consider the morphism $r_\mathbf{V}:\widetilde{Z}(j)\to \mathbf{V}_0$ defined in diagram (\ref{rY=}). By
(\ref{lambda'j=})
we have
\begin{equation}\label{mu circ rV}
\tilde{\lambda}_j|_{\widetilde{Z}(j)}=\mu\circ r_\mathbf{V}.
\end{equation}
We now prove the following proposition.
\begin{proposition}\label{Z' in widetilde Z }
Take any irreducible component $\widetilde{Z}$ of $Z_m$ and any monomorphism
$j:H_{m-1}\hookrightarrow H_m$ such that
$\widetilde{Z}(j)$ is nonempty. Then the morphism
$r_\mathbf{V}:\widetilde{Z}(j,\mathbf{I})\to \mathbf{V}_0$
\footnote{See the definition of the sets $\widetilde{Z}(j,\mathbf{I})$ in (\ref{tilde Z(j,I)}).}
is dominating and, for a general point
$z\in\widetilde{Z}(j,\mathbf{I})$, the fibre $r_\mathbf{V}^{-1}(r_\mathbf{V}(z))$ coincides
with $V(z,j)$ where $V(z,j)$ is defined in (\ref{V(z,j)}).
Moreover, $\dim\widetilde{Z}(j,\mathbf{I})=4m(m+2)$, and there exists a dense open subset $Z'$ of
$\widetilde{Z}(j,\mathbf{I})$
such that
\begin{equation}\label{dimV(z,j)=2m}
\dim V(z,j)=2m,\ \ \ z\in Z',
\end{equation}
\begin{equation}\label{fibres coincide}
r_\mathbf{V}^{-1}(r_\mathbf{V}(z))=\tilde{\lambda}_j^{-1}(\tilde{\lambda}_j(z))=
\lambda_{(j)}^{-1}(\lambda_j(z))=V(z,j),\ \ \ z\in Z',
\end{equation}
\begin{equation}\label{codim>1}
{\rm codim}_{\widetilde{Z}(j,\mathbf{I})}(\widetilde{Z}(j,\mathbf{I})\smallsetminus Z')\ge2.
\end{equation}

\end{proposition}

\begin{proof}
First, since by definition $\widetilde{Z}(j,\mathbf{I})$ is an open subset of $Z_m$, we have by (\ref{dim Zm})
$\dim\widetilde{Z}=\dim\widetilde{Z}(j,\mathbf{I})\ge4m(m+2)$.

Next, set $\mathbf{V}_{00}:=\rho^{-1}(\mathbf{Y}_0)$. According to the diagram (\ref{rY=}) we have
$r_\mathbf{V}(\widetilde{Z}(j,\mathbf{I}))\subset\mathbf{V}_{00}$. Consider the composition of projections
$$
\mathbf{p}:\mathbf{V}_{00}\overset{\rho}\to\mathbf{Y}_0\overset{p_{\mathbf{Y}_0}}\to
\mathbf{T}_0\overset{p_\mathbf{T}}\to\mathcal{H}^0:=\mathcal{H}_{2m}^{tH-curv},
$$

$$
\mathbf{p}_j:\widetilde{Z}(j,\mathbf{I})\overset{r_\mathbf{V}}\to\mathbf{V}_{00}
\overset{\mathbf{p}}\to\mathcal{H}^0.
$$
Since the projections $\rho,\ p_{\mathbf{Y}_0}$ and $p_\mathbf{T}$ are smooth fibrations with fibers of dimensions,
respectively, 1,
$(2m-1)^2-1$ and $(2m-1)$, and $\dim\mathcal{H}^0=\dim\mathcal{H}=8m$ (cf. (\ref{dim T})), it follows that
\begin{equation}\label{dim V00}
\dim\mathbf{V}_{00}=\dim({\rm fibre\ of}\ \mathbf{p})+\dim\mathcal{H}^0=2m(2m-1)+8m=4m^2+6m.
\end{equation}
Whence,
\begin{equation}\label{dim fibre of rV}
\dim\{{\rm generic\ fibre\ of}\ r_\mathbf{V}:\widetilde{Z}(j,\mathbf{I})\to\mathbf{V}_{00}\}\ge
\dim\widetilde{Z}(j,\mathbf{I})-\dim\mathbf{V}_{00}\ge
\end{equation}
$$
\ge4m(m+2)-(4m^2+6m)=2m.
$$
Now take an arbitrary point $z\in\widetilde{Z}(j,\mathbf{I})$ and set $v:=r_\mathbf{V}(z),\ A:=\tilde{\lambda}_j(z)$. From
(\ref{mu circ rV}) it follows that $A=\mu(v)$ and so by Lemma \ref{fibre of lambda}(ii)
\begin{equation}\label{contained}
r_\mathbf{V}^{-1}(v)\subset\tilde{\lambda}_j^{-1}(A)=V(z,j),
\end{equation}
where $V(z,j)$ is described in (\ref{V(z,j)}). Using Remark \ref{special tH bdls}, we rewrite (\ref{V(z,j)}) as:
\begin{equation}\label{dim V(z,j)}
\dim V(z,j)=
\left\{\begin{array}{cc}
2m, & {\rm if}\ \mathbf{p}(z)\in\mathcal{H}^*,\\
2m+1, & {\rm if}\ \mathbf{p}(z)\in\mathcal{H}_s,
\end{array}\right.
\end{equation}
where we set
$$
\mathcal{H}^*:=\mathcal{H}^0\smallsetminus\mathcal{H}_s.
$$
As $\mathbf{p}:\mathbf{V}_{00}\to\mathcal{H}^0$ is a smooth fibration with fibres
of dimension $2m(2m-1)$ (see (\ref{dim V00})), formulas (\ref{codim Hs}), (\ref{dim V00}), (\ref{contained}) and
(\ref{dim V(z,j)}) yield
\begin{equation}\label{dim of ...}
\dim \mathbf{p}_j^{-1}(\mathcal{H}_s)\le4m^2+6m-3+(2m+1)=4m(m+2)-2<\dim\widetilde{Z}(j,\mathbf{I}).
\end{equation}
Thus, $\mathbf{p}_j(\widetilde{Z}(j,\mathbf{I}))\not\subset\mathcal{H}_s$, i.e. there is a dense open subset $Z'$ of
$\widetilde{Z}(j,\mathbf{I})$ such that $\mathbf{p}_j(Z')\subset\mathcal{H}^*$. In particular,
(\ref{contained}) and (\ref{dim V(z,j)}) imply
\begin{equation}\label{again dim fibre of rV}
\dim r_\mathbf{V}^{-1}(r_\mathbf{V}(z))\le\dim V(z,j)=2m,\ \ \ z\in Z'.
\end{equation}
On the other hand, since $Z'$ is dense in $\widetilde{Z}(j,\mathbf{I})$, (\ref{dim fibre of rV}) yields
$\dim r_\mathbf{V}^{-1}(r_\mathbf{V}(z))\ge2m,\ z\in Z'$.
Comparing this with (\ref{again dim fibre of rV}), (\ref{dim V00}) and the inequality $\dim Z'\ge4m(m+2)$,
we obtain that
\begin{equation}\label{1dims}
\dim r_\mathbf{V}(Z')=\dim\mathbf{V}_{00}=4m^2+6m,\ \ \
\end{equation}
\begin{equation}\label{2dims}
r_\mathbf{V}^{-1}(r_\mathbf{V}(z))=V(z,j),\ \ \ \dim V(z,j)=2m,\ \ \ z\in Z'.
\end{equation}
Moreover, (\ref{codim>1}) follows from (\ref{dim of ...}). Now since the minimal possible dimension of $V(z,j)$ is $2m$,
the equality (\ref{fibres coincide}) follows from Lemma \ref{fibre of lambda}(ii-iii) (see (\ref{3fibre=V(z,j)}) and
(\ref{5fibre=V(z,j)})) by the
semicontinuity of dimension of fibres of a morphism of irreducible varieties.
This together with (\ref{1dims}) and (\ref{2dims}) yields Proposition.
\end{proof}

\vspace{0.2cm}
Now we are ready to finish the proof of Theorem \ref{Irreducibility of Zm}.

\vspace{0.3cm}
\noindent
{\it End of the proof of Theorem \ref{Irreducibility of Zm}.}

(i) We prove the irreducibility of $Z_m$, and the surjectivity of the projection
$p_m:Z\to(\mathbf{S}_m^\vee)^0:(D,\phi)\mapsto D$ will be a by-product of this proof.
First, $Z_m$ contains an irreducible component $Z$ introduced in Proposition
\ref{part of inductn step}. Assume that there exists
another irreducible component $Z'$ of $Z_m$. Let
$b:\mathbf{\Phi}_m\smallsetminus\{0\}\to P(\mathbf{\Phi}_m)$
be the canonical projection and
$\mathbf{b}:=id\times b:(\mathbf{S}_m^\vee)^0\times(\mathbf{\Phi}_m\smallsetminus\{0\})\to
(\mathbf{S}_m^\vee)^0\times P(\mathbf{\Phi}_m)$
be the induced projection. The equations of $Z_m$ in $(\mathbf{S}_m^\vee)^0\times\mathbf{\Phi}_m$
(see (\ref{tildeZm})-(\ref{Zm})) are homogeneous with respect to affine coordinates in $\mathbf{\Phi}_m$,
hence there exist irreducible closed subsets
$\underline Z$ and ${\underline Z}'$ and the closed subset ${\underline Z}_m$
in $(\mathbf{S}_m^\vee)^0\times P(\mathbf{\Phi}_m)$
such that $Z=\mathbf{b}^{-1}(\underline Z)\cup\{0\}$, respectively, $Z'=\mathbf{b}^{-1}({\underline Z}')\cup\{0\}$,
respectively, $Z_m=\mathbf{b}^{-1}({\underline Z}_m)\cup\{0\}$.
Moreover, by construction $\underline Z$ and ${\underline Z}'$
are irreducible components of ${\underline Z}_m$.

Take any point
\begin{equation}\label{take y}
y=(D_0,<\phi>)\in{\underline Z}'\smallsetminus
{\underline Z}'\cap{\underline Z}
\end{equation}
and consider the projective space $\mathbb{P}=\{D_0\}\times P(\mathbf{\Phi}_m)$,
$\dim \mathbb{P}=6m^2-1$. By definition, the sets
${\underline Z}_m(D_0)={\underline Z}'\cap P_D$ and  ${\underline Z}'(D_0)={\underline Z}'\cap P_D$
are closed subsets of $\mathbb{P}$ such that
\begin{equation}\label{Z'(D0) in Zm(D0)}
y\in{\underline Z}'(D_0)\subset{\underline Z}_m(D_0)
\end{equation}
and by Remark\ref{Z as zeroset} we have
$\codim_{\mathbb{P}}{\underline Z}'(D_0)\le5m(m-1)$
\begin{equation}\label{codim Zm(D0)}
\dim_{\mathbb{P}}{\underline Z}_m(D_0)\ge m^2+5m-1\ge1,\ \ \ m\ge1.
\end{equation}
By definition, ${\underline Z}_m(D_0)$ is given in $\mathbb{P}$ by $5m(m-1)$ global equations of the form
$\phi^\vee\circ D_0\circ\phi\in\mathbf{S}_m$. Hence, in view of (\ref{codim Zm(D0)}) ${\underline Z}_m(D_0)$
is connected.

Next, by Proposition \ref{part of inductn step}(ii) the morphism $pr_1:Z\to(\mathbf{S}_m^\vee)^0:(D,\phi)\mapsto D$
is dominant, so that the induced projective morphism
${\underline Z}\to(\mathbf{S}_m^\vee)^0:(D,<\phi>)\mapsto D$
is also dominant, hence surjective,\footnote{This clearly implies the surjectivity of projection
$p_m=pr_1:Z\to(\mathbf{S}_m^\vee)^0$.}
 since ${\underline Z}$ is closed in $(\mathbf{S}_m^\vee)^0\times P(\mathbf{\Phi}_m)$.
In particular, the set
${\underline Z}(D_0)={\underline Z}\cap\mathbb{P}$
is a nonempty closed subset of
${\underline Z}_m(D_0)$.
In addition, by (\ref{take y}) $y\in{\underline Z}_m(D_0)\smallsetminus{\underline Z}(D_0)$. Hence, since
${\underline Z}_m(D_0)$ is connected, it contains an irreducible component, say, ${\underline Z}''(D_0)$
distinct from ${\underline Z}(D_0)$ and intersecting ${\underline Z}(D_0)$. Let
${\underline Z}''$ be an irreducible
component of ${\underline Z}_m$ containing ${\underline Z}''(D_0)$, hence distinct from ${\underline Z}(D_0)$.
We thus have
\begin{equation}\label{nonempty}
{\underline Z}\cap{\underline Z}''\ne\emptyset.
\end{equation}
Let $Z''=\mathbf{b}^{-1}({\underline Z}'')\cup\{0\}$.
By construction $Z''$ is an irreducible component of $Z_m$ such that, in view of (\ref{nonempty}),
there exists a point
\begin{equation}\label{exists point z}
z=(D,\phi)\in Z\cap Z'',\ \ \ \ \phi\ne0.
\end{equation}

Since $Z_m$ is given in $(\mathbf{S}_m^\vee)^0\times\mathbf{\Phi}_m$ by $5m(m-1)$ equations (see (\ref{hm^-1(0)}))
and $Z$ has dimension $4m(m+2)$ (Proposition \ref{part of inductn step}). Hence, outside of its intersection with
other irreducible components of $Z_m$, $Z$ is a locally complete intersection of codimension $5m(m-1)$ in
$(\mathbf{S}_m^\vee)^0\times\mathbf{\Phi}_m$. Now it follows easily from the connectedness in codimension 1 of
locally complete intersections (see \cite{H2}) that through any point of intersection of $Z$ with other components
of $Z_m$ (e.g., through the point $z$ in (\ref{exists point z})) there passes a component, say,
$\widetilde{Z}$ of $Z_m$, distinct from $Z$, such that
${\rm codim}_{Z}Z\cap\widetilde{Z}=1$.

Take any irreducible component $F$ of $Z\cap\widetilde{Z}$ having codimension 1 in $Z$. From Lemma \ref{codim RZ*}
it follows now that the set $F':=F\smallsetminus(R_{Z}\cap\{$union of all possible components of
$Z\cap\widetilde{Z}$ distinct from $F\})$ is dense open in $F$. Take any point $z\in F'$. By
Proposition \ref{hat Z(j)} there exists a monomorphism $j:H_{m-1}\hookrightarrow H_m$ such that
$z\in Z(j,\mathbf{I})$. Then by Proposition \ref{Z' in widetilde Z }, in which we take $Z$ for $\widetilde{Z}$,
it follows that:

1) there exists a dense open subset
$Z'$ of $Z(j,\mathbf{I})$ such that $F^*:=F'\cap Z'$ is dense open in $F$ (see
(\ref{codim>1})),

2) for any point $z\in F^*$,
$\lambda_j^{-1}(\lambda_j(z))=\tilde{\lambda}_j^{-1}(\tilde{\lambda}_j(z))=V(z,j)\simeq\mathbf{k}^{2m}$.
(In fact, apply formula (\ref{fibres coincide}) to $Z(j,\mathbf{I})$ and to $\widetilde{Z}(j,\mathbf{I})$,
respectively). The last equality means that
\begin{equation}\label{V in Zcap...}
z=(D,\phi)\in V(z,j)\subset Z\cap\widetilde{Z},\ \ \ \dim V(z,j)=2m.
\end{equation}
Now we obtain from (\ref{V in Zcap...}) and diagram (\ref{diag V(z,j)}) that there
exists a monomorphism
$j'_\mathbf{k}:\mathbf{k}\hookrightarrow V(z,j)$ for which the induced homomorphism
${}^\sharp\phi':=({}^\sharp\phi\circ j,j'_\mathbf{k}):H_m=H_{m-1}\oplus\mathbf{k}\to V(z,j)\hookrightarrow
H_m^\vee\otimes\wedge^2V^\vee$ is such that, in notations of (\ref{diag n+1}), the point
$z'=(D,\phi')\in V(z,j)$ satisfies the condition:
$$
the\ composition\ s(z'):H_m\to H_m^\vee\otimes\wedge^2V^\vee\overset{c_D}\twoheadrightarrow H^0(E_{2m}(D^{-1})(1))\
is\ injective.
$$
Az $z\in Z(j,\mathbf{I})$, the composition
$s(z')\circ j:H_{m-1}\to H^0(E_{2m}(D^{-1})(1))$
is also injective. This together with the above condition and exactly means that the point $z'\in V(z,j)\subset$
satisfies both conditions (\textbf{I}) and (\textbf{II}) in the definition of $\widetilde{Z}(j)$ in Lemma
\ref{Z*(j)}. It follows now from (\ref{V in Zcap...}) that $\widetilde{Z}(j)$ is nonempty.

We are now in conditions of Proposition \ref{Z' in widetilde Z } which we apply to the irreducible sets
$Z(j)$ and $\widetilde{Z}(j)$. Consider the morphism $r_\mathbf{V}:Z_m(j)\to\mathbf{V}^0$
and its restrictions $r:=r_\mathbf{V}|_{Z(j)}$ and $\tilde{r}:=r_\mathbf{V}|_{\widetilde{Z}(j)}$.
Then according to Proposition \ref{Z' in widetilde Z } there exist dense open subsets $Z'$ of $Z(j)$ and,
respectively, $\widetilde{Z}'$ of $\widetilde{Z}(j)$, such that
$\mathbf{V}':=r(Z')=\tilde{r}(\widetilde{Z}')$.
Now, for a general point $v\in\mathbf{V}'$ and an arbitrary point $z\in r^{-1}(v)\cap Z'$, one has by
(\ref{fibres coincide}):
$$
r^{-1}(v)=V(z,j)=\lambda_{(j)}^{-1}(v)=\tilde{r}^{-1}(v).
$$
This is clearly a contradiction, since, by assumption, $Z(j)$ and $\widetilde{Z}(j)$ are distinct varieties.
Hence $Z_m$ is irreducible.

The surjectivity of the morphism $p_m:Z_m\to(\mathbf{S}_m^\vee)^0$ was already mentioned in the footnote 7 above.
Theorem \ref{Irreducibility of Zm} is proved.

\end{sub}

\vspace{0.5cm}

\section{Appendix: two results of general position }\label{Appendix}

\vspace{0.5cm}

In this Appendix we prove Theorem \ref{generic Vm} and Proposition \ref{nondeg for general}.

\begin{sub} {\bf Proof of Theorem \ref{generic Vm}.}
\rm

We first need to recall some definitions
and standard facts from theory of determinantal varieties.

\begin{definition}\label{rk(x)}
Let $U$ and $U'$ be two vector spaces of dimensions respectively $m$ and $n$, where $m\ge n$.
Consider the projective space $P(U\otimes U')$.
We say that {\it a point $x\in P(U\otimes U')$ has rank} $r$ (and denote this as
$\rk(x)=r$), if

(i) there exist unique subspaces $U_r(x)\subset U$ and $U'_r(x)\subset U'$ of dimensions
$\dim U_r(x)=\dim U'_r(x)=r$ such that $x\in P(U_r(x)\otimes U'_r(x))$, and

(ii) there do not exist subspaces $\tilde{U}\subset U$ and $\tilde{U}'\subset U'$ of dimension
$\dim \tilde{U}=\dim \tilde{U}'<r$ such that $x\in P(\tilde{U}\otimes \tilde{U}')$.
\end{definition}

The following Lemma is a well known fact from the theory of determinantal varieties (see, e. g., \cite{R}).
\begin{lemma}\label{Determ}
Each point $x\in P(U\otimes U')$ has a uniquely defined rank ${\rm rk}(x)$,
$1\le{\rm rk}(x)\le n$. Moreover, for a given point $x\in P(U\otimes U')$ of rank ${\rm rk}(x)=r$
such that $x\in W\otimes W'$ for some subspaces $W\subset U$ and $W'\subset U'$, the
subspaces $U_r(x)\subset U$ and $U'_r(x)\subset U'$ of dimensions $\dim U_k(x)=\dim U'_k(x)=r$
defined in (i) above are such that $U_r(x)\subset W$ and $U'_r(x)\subset W'$.
\end{lemma}

\begin{proof}
According to Definition \ref{rk(x)} in which we put $U=H^\vee_{2m+1},\ U'=V^\vee$, each
point $x\in P(H^\vee_{2m+1}\otimes V^\vee)$
has rank
$1\le\rk(x)\le\dim V^\vee=4$
\footnote{Everywhere in this proof by the rank of a point $x$ of a given subspace of $P(H^\vee_{2m+1}\otimes V^\vee)$ we understand
its rank as of a point in $P(H^\vee_{2m+1}\otimes V^\vee)$.}
. Thus
\begin{equation}\label{P(W)}
P(W^\vee_{4m+4})=\underset{r=1}{\overset{4}\cup}Z_r,
\end{equation}
where
$$
Z_r:=\{x\in P(W^\vee_{4m+4})\ |\ rk(x)=r\},\ \ 1\le r\le4,
$$
are locally closed subsets of $P(W^\vee_{4m+4})$. Consider the Grassmannian
$$
G:=G(m,H^\vee_{2m+1})
$$
and its locally closed subsets
\begin{equation}\label{def of Sigma r}
\Sigma_r:=\{V_m\in G\ |\ V_m\supset U_r(x)\ {\rm for\ some\ point}\ x\in Z_r\},\ \ \ 1\le r\le4.
\end{equation}
In view of Lemma \ref{Determ} the condition $x\in Z_r\cap P(V_m\otimes V^\vee)$ means that
$x\in Z_r\cap P(U_r\otimes V^\vee)$ for some $r$-dimensional subspace $U_r=U_r(x)\subset V_m$. This
together with (\ref{P(W)}) and (\ref{def of Sigma r}) shows that
$$
\{V_m\in G\ |\ P(V_m\otimes V^\vee)\cap P(W^\vee_{4m+4})\ne\emptyset\}=
\underset{r=1}{\overset{4}\cup}\Sigma_r.
$$
Now the theorem says that $\underset{r=1}{\overset{4}\cup}\Sigma_r\underset{\ne}\subset G$.
Thus, to prove the theorem, it is enough to show that
\begin{equation}\label{dim Sigma r}
\dim\Sigma_r<\dim G, \ \ \ 1\le r\le4.
\end{equation}
We are starting now the proof of (\ref{dim Sigma r}) for $r=4,3,2,1$.

(i) {\bf Case $\mathbf{r=4}$}. Set
$\Gamma_4:=\{(x,U)\in P(W_{4m+4}^\vee)\times G(4,H^\vee_{2m+1})\ |\ \rk(x)=4\ {\rm and}
\ U=U_4(x)\}$
and let
$P(W_{4m+4}^\vee)\overset{p_4}\leftarrow\Gamma_4\overset{q_4}\to G(4,H^\vee_{2m+1})$
be the projections. By construction, $p_4(\Gamma_4)=Z_4$, and by the definition \ref{rk(x)}(i) the projection
$p_4:\Gamma_4\to Z_4$ is a bijection. Hence
$$
\dim q_4(\Gamma_4)\le\dim\Gamma_4=\dim Z_4\le\dim P(W_{4m+4}^\vee)=4m+3.
$$
By construction we have the graph of incidence
$$
\Pi_4=\{(U,V_m)\in q_4(\Gamma_4)\times\Sigma_4\ |\ U\subset V_m\}
$$
with surjective projections
$q_4(\Gamma_4)\overset{pr_1}\leftarrow\Pi_4\overset{pr_2}\to\Sigma_4$
and a fibre
\begin{equation}\label{Grass4}
pr_1^{-1}(U)\simeq G(m-4,H^\vee_{2m+1}/U)
\end{equation}
over an arbitrary point $U\in q_4(\Gamma_4)$.
(In fact, the condition
$U\subset V_m\subset H^\vee_{2m+1}$
means that
$V_m/U\in G(m-4,H^\vee_{2m+1}/U)$.)
Hence
$$
\dim\Sigma_4\le\dim\Pi_4=\dim q_4(\Gamma_4)+\dim G(m-4,H^\vee_{2m+1}/U)\le4m+3+(m-4)(m+1)=
m(m+1)-1=
$$
$=\dim G-1<\dim G$, i.e. (\ref{dim Sigma r}) is true for $r=4$.

\vspace{0.4cm}
(ii) {\bf Case $\mathbf{r=3}$}. Consider the projection
$f_3:Z_3\to P( V^\vee)^\vee=\mathbb{P}^3: x\mapsto V_3(x)$,
where the pair of 3-dimensional spaces
$(U_3(x),V_3(x)), \ \ \ U_3(x)\subset H^\vee_{2m+1}$
and
$V_3(x)\subset V^\vee$,
is determined uniquely by the point $x$ via the condition $x\in P(U_3(x)\otimes V_3(x))$,
since $\rk(x)=3$ (see Definition \ref{rk(x)} and Lemma \ref{Determ}).
Now for a given 3-dimensional subspace $V_3\subset V^\vee$ set
\begin{equation}\label{Sigma3(V3)}
\Sigma_3(V_3)=\{V_m\in G\ |\ V_m\supset U_3(x)\ {\rm for\ some\ point}\ x\in f_3^{-1}(V_3)\}.
\end{equation}
Comparing this with (\ref{def of Sigma r}) for $r=3$ we obtain
\begin{equation}\label{Sigma3=}
\Sigma_3=\underset{V_3\subset V^\vee}\cup\Sigma_3(V_3).
\end{equation}
Note that a priori $f_3$ is not necessarily surjective. Hence,
\begin{equation}\label{dim Sigma3}
\dim\Sigma_3\le\dim\Sigma_3(V_3)+3.
\end{equation}
We are going to obtain an estimate for the dimension of $\Sigma_3(V_3)$ for an arbitrary
3-dimensional subspace $V_3$ of $ V^\vee$. This subspace defines a commutative diagram
\begin{equation}\label{Di1}
\xymatrix{
& 0\ar[d] & 0\ar[d] & 0\ar[d] & \\
0\ar[r] & F\ar[r]\ar[d] & \Omega_{\mathbb{P}^3}\ar[r]\ar[d] & \mathcal{I}_z(-1)\ar[r]\ar[d] &
0 \\
0\ar[r] & V_3\otimes\mathcal{O}_{\mathbb{P}^3}(-1)\ar[r]\ar[d] &
 V^\vee\otimes\mathcal{O}_{\mathbb{P}^3}(-1)\ar[r]\ar[d] &
\mathcal{O}_{\mathbb{P}^3}(-1)\ar[r]\ar[d] & 0 \\
0\ar[r] &  \mathcal{I}_z\ar[r]\ar[d] & \mathcal{O}_{\mathbb{P}^3}\ar[r]\ar[d] &
\mathbf{k}_z\ar[r]\ar[d] & 0 \\
& 0 & 0 & 0, &}
\end{equation}
where $z=P(\ker: V\twoheadrightarrow V_3^\vee)$ is a point in $\mathbb{P}^3$
and the sheaf $F$ has an $\mathcal{O}_{\mathbb{P}^3}$-resolution
$0\to\mathcal{O}_{\mathbb{P}^3}(-3)\to3\mathcal{O}_{\mathbb{P}^3}(-2)\to F\to0$.
Twisting this resolution by the vector bundle $E$ and passing to cohomology we obtain
the equalities $H^1(F\otimes E)\simeq H^2(E(-3))=H_{2m+1},\ H^2(F\otimes E)=0$.  Respectively,
passing to cohomology in diagram (\ref{Di1}) twisted by $E$ and using the above equalities
and evident relations
$H^0(E\otimes\mathbf{k}_z)\simeq\mathbf{k}^2,\ \ H^1(E\otimes\mathbf{k}_z)=0$
implies the diagram

\begin{equation}\label{Di2}
\xymatrix{
& 0\ar[d] & 0\ar[d] & \mathbf{k}^2\ar@{>->}[d] & \\
0\ar[r] & H_{2m+1}\ar[r]\ar[d] & W_{4m+4}^\vee\ar[r]\ar[d]
& H^1(E\otimes\mathcal{I}_z(-1))\ar[d]\ar[r] & 0 \\
0\ar[r] & H^\vee_{2m+1}\otimes V_3\ar[r]^\lambda\ar[d] &
H^\vee_{2m+1}\otimes V^\vee\ar[r]\ar[d]^{mult} & H^\vee_{2m+1}\ar[d]\ar[r] & 0 \\
\mathbf{k}^2\ \ar@{>->}[r] & H^1(E\otimes\mathcal{I}_z)\ar[r]\ar[d] &
H^\vee_{4m}\ar[r]\ar[d] & 0 \\
& 0 & 0. &}
\end{equation}
In this diagram the composition $\epsilon:=mult\circ\lambda$ is surjective. Hence, setting
$W_{2m+3}(V_3):=\ker\epsilon$, where $\dim W_{2m+3}(V_3)=2m+3$, we obtain a commutative
diagram
$$\xymatrix{
& 0\ar[d] & 0\ar[d] & \\
0\ar[r] & W_{2m+3}(V_3)\ar[r]\ar[d] & W_{4m+4}^\vee\ar[r]\ar[d]
& H^\vee_{2m+1}\ar@{=}[d] \\
0\ar[r] & H^\vee_{2m+1}\otimes V_3\ar[r]^\lambda\ar[d]^\epsilon &
H^\vee_{2m+1}\otimes V^\vee\ar[r]\ar[d]^{mult} & H^\vee_{2m+1} \\
& H^\vee_{4m}\ar@{=}[r]\ar[d] & H^\vee_{4m}\ar[d] \\
& 0 & \ 0 &}
$$
which yields the relation
\begin{equation}\label{Di3}
W_{2m+3}(V_3)=H^\vee_{2m+1}\otimes V_3\cap W_{4m+4}^\vee,
\end{equation}
where the intersection is taken in $H^\vee_{2m+1}\otimes V^\vee$.
Set
$$
Z_3(V_3):=\{x\in P(W_{2m+3}(V_3))\ |\ \rk(x)=3\}.
$$
The relation (\ref{Di3}) and Lemma \ref{Determ} imply the bijection
\begin{equation}\label{bijectn0}
Z_3(V_3)\overset{\simeq}\to f_3^{-1}(V_3).
\end{equation}
Consider the graph of incidence
$\Gamma_3(V_3):=\{(x,U)\in Z_3(V_3)\times G(3,H^\vee_{2m+1})\ |U=U_3(x)\}$
with projections
$Z_3(V_3)\overset{p_3}\leftarrow\Gamma_3(V_3)\overset{q_3}\to G(3,H^\vee_{2m+1})$.
By Lemma \ref{Determ}, $p_3(\Gamma_3(V_3))=Z_3(V_3)$ and the projection
$p_3:\Gamma_3(V_3)\to Z_3(V_3)$ is a bijection. Hence
\begin{equation}\label{dim q3(...)<}
\dim q_3(\Gamma_3(V_3))\le\dim\Gamma_3(V_3)=\dim Z_3(V_3)\le\dim P(W_{2m+3}(V_3))=2m+2.
\end{equation}
Consider the graph of incidence
$$
\Pi_3(V_3)=\{(U,V_m)\in q_3(\Gamma_3(V_3))\times\Sigma_3(V_3)\ |\ U\subset V_m\}
$$
with projections
$q_3(\Gamma_3(V_3))\overset{pr_1}\leftarrow\Pi_3(V_3)\overset{pr_2}\to\Sigma_3(V_3)$
and a fibre
\begin{equation}\label{Grass3}
pr_1^{-1}(U)\simeq G(m-3,H^\vee_{2m+1}/U)
\end{equation}
over an arbitrary point $U\in q_3(\Gamma_3(V_3))$ (cf. (\ref{Grass4})).
The projection $\Pi_3(V_3)\overset{pr_2}\to\Sigma_3(V_3)$ is surjective in view of
(\ref{bijectn0}). Hence, using (\ref{dim q3(...)<}), we obtain
$$
\dim\Sigma_3(V_3)\le\dim\Pi_3(V_3)=\dim q_3(\Gamma_3(V_3))+\dim G(m-3,H^\vee_{2m+1}/U)\le2m+2+
(m-3)(m+1)=
$$
$=m^2-1$. This together with (\ref{dim Sigma3}) and the assumption $m\ge3$ yields
$\dim\Sigma_3\le m^2+2=\dim G+2-m<\dim G$, i.e. (\ref{dim Sigma r}) holds for $r=3$.

\vspace{0.4cm}
Before proceeding to the case $r=2$ we need to make a small digression on jumping lines of $E$.
Introduce some more notation. For a given line $l\subset\mathbb{P}^3$ we have
$E|l\simeq\mathcal{O}_{\mathbb{P}^1}(d)\oplus\mathcal{O}_{\mathbb{P}^1}(-d)$
for a well-defined nonnegative integer $d$ called the {\it jump of} $E|l$ and denoted also by
$d_E(l)$; respectively, the line $l$ is called a {\it jumping line of jump $d$ of} $E$.
Set $G_{2,4}:=G(2, V^\vee)$ and $J_k(E):=\{l\in G_{2,4}\ |\ d_E(l)\ge k\}$,
$J^*_k(E):=J_k(E)\smallsetminus J_{k+1}(E),\  0\le k$.
From the semicontinuity of $E|l,\ l\in G_{2,4}$, it follows that $J _k(E)$ (resp., $J^*_k(E))$
is a closed (resp., locally closed) subset of $G_{2,4},\ k\ge0$. Moreover, by a well-known theorem of
Grauert-M\"ulich, $J^*_0(E)$ is a dense open subset of $G_{2,4}$. Next, since $E\in I'_{2m+1}$,
it follows that
\begin{equation}\label{emptyjump}
J_{2m+1}(E)=\emptyset,
\end{equation}
so that
\begin{equation}\label{jump2m-1}
J_{2m-1}(E)=J^*_{2m-1}(E)\sqcup J^*_{2m}(E).
\end{equation}
We will use below the following lemma.
\begin{lemma}\label{jumps}
Let $E\in I'_{2m+1}$. Then

$(1)\ \dim J_{2m-1}(E)\le1$.

$(2)\ \dim J^*_k(E)\le3$ for $1\le k\le 2m-2$.
\end{lemma}
\begin{proof}
(1) Suppose the contrary, i.e. $\dim J_{2m-1}(E)\ge2$. Take any irreducible surface
$S\subset J_{2m-1}(E)$ and let $D$ be the degree of $S$ with respect to the sheaf
$\mathcal{O}_{G_{2,4}}(1)$. Fix an integer $r\ge 5$ and take any irreducible curve $C$
belonging to the linear series $\big|\mathcal{O}_{G_{2,4}}(r)|_S\big|$. Then the degree
$\deg C$ w.r.t. $\mathcal{O}_{G_{2,4}}(1)$ equals to $Dr$, hence $\deg C\ge5$. Hence by
\cite[Lemma 6]{C} there exist two distinct lines, say, $l_1,l_2\in C$, which intersect in
$\mathbb{P}^3$. Let the plane
$\mathbb{P}^2$ be the span of $l_1$ and $l_2$ in $\mathbb{P}^3$. Now the exact triple
$0\to E(-2)|_{\mathbb{P}^2}\to E|_{\mathbb{P}^2}\to E|_{l_1\cup l_2}\to0$ implies
\begin{equation}\label{tripleA}
H^0(E|_{\mathbb{P}^2})\to H^0(E|_{l_1\cup l_2})\to H^1(E(-2)|_{\mathbb{P}^2}).
\end{equation}
Next, as $[E]\in I_{2m+1}$, we have $h^0(E(-1))=h^1(E(-2))=0$, hence the exact
triple $0\to E(-2)\to E(-1)\to E(-1)|_{\mathbb{P}^2}\to0$ implies
\begin{equation}\label{tripleB}
H^0(E(-1)|_{\mathbb{P}^2})=0.
\end{equation}
Now assume $h^0(E|_{\mathbb{P}^2})>0$. Then a section
$0\ne s\in H^0(E|_{\mathbb{P}^2})$
defines an injection
$\mathcal{O}_{\mathbb{P}^2}\overset{s}\hookrightarrow E|_{\mathbb{P}^2}$. This injection and
(\ref{tripleB}) show that the zero-set $Z$ of the section $s$ is 0-dimensional and the injection
$s$ extends to a triple
$0\to\mathcal{O}_{\mathbb{P}^2}\overset{s}\to E|_{\mathbb{P}^2}\to
\mathcal{I}_{Z,\mathbb{P}^2}\to0$.
Whence
\begin{equation}\label{tripleC}
h^0(E|_{\mathbb{P}^2})\le1.
\end{equation}
Furthermore, equality (\ref{tripleB}) together with Riemann-Roch and Serre duality for the vector bundle
$E(-1)|_{\mathbb{P}^2}$ shows that $h^1(E(-2)|_{\mathbb{P}^2})=2m+1$. Whence in view of
(\ref{tripleA}) and (\ref{tripleB}) we obtain
\begin{equation}\label{tripleD}
h^0(E|_{l_1\cup l_2})\le2m+2.
\end{equation}
On the other hand, let $x:=l_1\cap l_2$. Since by construction $l_1,l_2\in J_{2m-1}(E)$, it
follows from (\ref{jump2m-1}) that either
$E|_{l_i}\simeq\mathcal{O}_{\mathbb{P}^2}(2m-1)\oplus\mathcal{O}_{\mathbb{P}^2}(1-2m)$,
or $E|_{l_i}\simeq\mathcal{O}_{\mathbb{P}^2}(2m)\oplus\mathcal{O}_{\mathbb{P}^2}(-2m)$,
hence $h^0(E\otimes\mathcal{I}_{x,l_i})\ge2m-1,\ i=1,2$. This clearly implies
$h^0(E|_{l_1\cup l_2})\ge h^0(E\otimes\mathcal{I}_{x,l_1\cup l_2})\ge
h^0(E\otimes\mathcal{I}_{x,l_1})+h^0(E\otimes\mathcal{I}_{x,l_2})=4m-2$. Comparing this with
(\ref{tripleD}) we obtain the inequality $2m+2\ge4m-2$, i.e. $m\le2$. This contradicts to the
assumption $m\ge3$. Hence, the assertion (1) follows.

(2) This is an immediate corollary of the theorem of Grauert-M\"ulich.
The lemma is proved.
\end{proof}

\vspace{0.4cm}
(iii) {\bf Case $\mathbf{r=2}$}. Here our notation and argument are completely parallel to those in the case $r=3$ above.
Consider a morphism $f_2:Z_2\to G_{2,4}: x\mapsto V_2(x)$,
where the pair of 2-dimensional spaces
$(U_2(x),V_2(x)), \ \ \ U_2(x)\subset H^\vee_{2m+1}$
and
$V_2(x)\subset V^\vee$,
is determined uniquely by the point $x$ via the condition $x\in P(U_2(x)\otimes V_2(x))$,
since $\rk(x)=2$ (see Lemma \ref{Determ}).

According to
(\ref{emptyjump}) we may assume that
$l\in J^*_k(E)$ for some $0\le k\le2m$, i.e.
$$
h^0(E|l)=2,\ \ \ h^1(E|l)=0, \ \ \ {\rm if}\ \ \ l\in J^*_0(E),
$$
respectively,
\begin{equation}\label{h0(E|l)}
h^0(E|l)=k+1,\ \ \ h^1(E|l)=k-1, \ \ \ {\rm if}\ \ \ l\in J^*_k(E),\ \ \ 1\le k\le2m.
\end{equation}
Now, for $1\le k\le2m$ and a given subspace $V_2\in J^*_k$, set
\begin{equation}\label{Sigma2(V2)}
\Sigma_{2,k}(V_2)=\{V_m\in G\ |\ V_m\supset U_2(x)\ {\rm for\ some\ point}\ x\in f_2^{-1}(V_2)\}.
\end{equation}
Then similarly to (\ref{Sigma3=}) we have
$$
\Sigma_2=\underset{k=0}{\overset{2m}\cup}\underset{V_2\in J^*_k}\cup\Sigma_{2,k}(V_2).
$$
Hence, in view of Lemma \ref{jumps}
\begin{equation}\label{dim Sigma2}
\dim\Sigma_2\le\underset{\overset{V_2\in J^*_k}{0\le k\le2m}}
\max(\dim\Sigma_{2,k}(V_2)+\dim J^*_k).
\end{equation}
We are going to obtain an estimate for the dimension of $\Sigma_{2,k}(V_2)$ for an arbitrary
2-dimensional subspace $V_2$ in $J^*_k,\ 0\le k\le2m$. This subspace defines a commutative diagram
\begin{equation}\label{Di4}
\xymatrix{
& 0\ar[d] & 0\ar[d] & 0\ar[d] & \\
0\ar[r] & \mathcal{O}_{\mathbb{P}^3}(-2)\ar[r]^s\ar[d] & \Omega_{\mathbb{P}^3}\ar[r]\ar[d] &
F\ar[r]\ar[d] & 0 \\
0\ar[r] & V_2\otimes\mathcal{O}_{\mathbb{P}^3}(-1)\ar[r]\ar[d] &
 V^\vee\otimes\mathcal{O}_{\mathbb{P}^3}(-1)\ar[r]\ar[d] &
V'_2\otimes\mathcal{O}_{\mathbb{P}^3}(-1)\ar[r]\ar[d] & 0 \\
0\ar[r] &  \mathcal{I}_l\ar[r]\ar[d] & \mathcal{O}_{\mathbb{P}^3}\ar[r]\ar[d] &
\mathcal{O}_l\ar[r]\ar[d] & 0 \\
& 0 & 0 &\ 0, &}
\end{equation}
where $V'_2:= V^\vee/V_2$, $l=P((V'_2)^\vee)$ is a line in $\mathbb{P}^3$,
and $F:=\coker s$.
Passing to cohomology in the diagram (\ref{Di4}) twisted by $E$, we obtain the diagram

\begin{equation}\label{Di5}
\xymatrix{
& & 0\ar[d] & H^0(E|l)\ar@{>->}[d] & \\
& & W_{4m+4}^\vee\ar@{=}[r]\ar[d]
& H^1(E\otimes F)\ar[d] & \\
0\ar[r] & H^\vee_{2m+1}\otimes V_2\ar[r]\ar@{=}[d] &
H^\vee_{2m+1}\otimes V^\vee\ar[r]\ar[d]^{mult} & H^\vee_{2m+1}\otimes V'_2\ar[d]^{\epsilon_2}\ar[r] & 0 \\
H^0(E|l)\ \ar@{>->}[r] & H^1(E\otimes\mathcal{I}_l)\ar[r] &
H^\vee_{4m}\ar[r]^{\epsilon_1}\ar[d] & H^1(E|l)\ar[d]\ar[r] & 0 \\
& & 0 & 0. &}
\end{equation}

Assume first that $1\le k\le2m$. (The case $k=0$ is treated below.)
In this case (\ref{h0(E|l)}) and the diagram (\ref{Di5}) lead to the diagram
$$
\xymatrix{
& 0\ar[d] & 0\ar[d] & 0\ar[d] & \\
0\ar[r] & W_{k+1}(V_2)\ar[r]\ar[d] & W_{4m+4}^\vee\ar[r]\ar[d] & \ker{\epsilon_2}\ar[d]\ar[r] & 0 \\
0\ar[r] & H^\vee_{2m+1}\otimes V_2\ar[r]\ar[d] &
H^\vee_{2m+1}\otimes V^\vee\ar[r]\ar[d]^{mult} & H^\vee_{2m+1}\otimes V'_2\ar[d]^{\epsilon_2}\ar[r] & 0 \\
0\ar[r] & \ker{\epsilon_1}\ar[r]\ar[d] & H^\vee_{4m}\ar[d]\ar[r]^{\epsilon_1} & H^1(E|l)\ar[d]\ar[r] & 0 \\
& 0 & 0 &\ 0, & }
$$
where we set
$W_{k+1}(V_2):=H^0(E|l)$. Here according to (\ref{h0(E|l)}) we have $\dim W_{k+1}(V_2)=k+1,\ \
\dim\ker{\epsilon_1}=4m-k+1,\ \ \dim\ker{\epsilon_2}=4m-k+3$. This diagram yields the relation (cf. (\ref{Di3}))
\begin{equation}\label{Di6}
W_{k+1}(V_2)=H^\vee_{2m+1}\otimes V_2\cap W_{4m+4}^\vee,
\end{equation}
where the intersection is taken in $H^\vee_{2m+1}\otimes V^\vee$.
Set
$$
Z_{2,k}(V_2):=\{x\in P(W_{k+1}(V_2))\ |\ \rk(x)=2\}.
$$
The relation (\ref{Di6}) and Lemma \ref{Determ} imply the bijection
\begin{equation}\label{bijectn2}
Z_{2,k}(V_2)\overset{\simeq}\to f_2^{-1}(V_2).
\end{equation}
Consider the graph of incidence
$\Gamma_{2,k}(V_2):=\{(x,U)\in Z_{2,k}(V_2)\times G(2,H^\vee_{2m+1})\ |\ U=U_2(x)\}$
with projections
$Z_{2,k}(V_2)\overset{p_2}\leftarrow\Gamma_{2,k}(V_2)\overset{q_2}\to G(2,H^\vee_{2m+1})$.
By construction, $p_2(\Gamma_{2,k}(V_2))=Z_{2,k}(V_2)$ and the projection
$p_2:\Gamma_{2,k}(V_2)\to Z_{2,k}(V_2)$ is a bijection. Hence
\begin{equation}\label{dim q2(...)<}
\dim q_2(\Gamma_{2,k}(V_2))\le\dim\Gamma_{2,k}(V_2)=\dim Z_{2,k}(V_2)
\le\dim P(W_{k+1}(V_2))=k.
\end{equation}
Consider the graph of incidence
$$
\Pi_{2,k}(V_2)=\{(U,V_m)\in q_2(\Gamma_{2,k}(V_2))\times\Sigma_{2,k}(V_2)\ |\ U\subset V_m\}
$$
with projections
$q_2(\Gamma_{2,k}(V_2))\overset{pr_1}\leftarrow\Pi_{2,k}(V_2)\overset{pr_2}\to\Sigma_{2,k}(V_2)$
and a fibre
$$
pr_1^{-1}(U)\simeq G(m-2,H^\vee_{2m+1}/U)
$$
over an arbitrary point $U\in q_2(\Gamma_{2,k}(V_2))$ (cf. (\ref{Grass4}) and (\ref{Grass3})).
The projection $\Pi_{2,k}(V_2)\overset{pr_2}\to\Sigma_{2,k}(V_2)$ is surjective in view of
(\ref{bijectn2}). Hence using (\ref{dim q2(...)<}) we obtain
\begin{equation}\label{dim Sigma2k}
\dim\Sigma_{2,k}(V_2)\le\dim\Pi_{2,k}(V_2)=\dim q_2(\Gamma_{2,k}(V_2))+
\dim G(m-2,H^\vee_{2m+1}/U)\le
\end{equation}
$$
\le k+(m-2)(m+1)=m^2-m-2+k=\dim G-(2m-k+2),\ \ \ \ 1\le k\le 2m.
$$
Now consider the case $k=0$. In this case one has $h^0(E|l)=2$ and, respectively,
$\dim q_2(\Gamma_{2,0}(V_2))\le\dim\Gamma_{2,0}(V_2)=\dim Z_{2,0}(V_2)
\le\dim P(W_1(V_2))=1$, instead of (\ref{dim q2(...)<}).
Hence, similar to the above we obtain for $k=0$:
$$
\dim\Sigma_{2,0}(V_2)\le1+(m-2)(m+1)=m^2-m-1=\dim G-(2m+1).
$$
The last inequality together with (\ref{dim Sigma2k}), (\ref{dim Sigma2}), Lemma \ref{jumps} and the assumption
$m\ge3$ yields $\dim\Sigma_2<\dim G$, i.e. (\ref{dim Sigma r}) is true for $r=2$.

\vspace{0.4cm}
(iv)  {\bf Case $\mathbf{r=1}$}. Again the notation and argument goes along the same lines as in cases $r=4,3$ and 2 above.
Consider the projection
$f_1:Z_1\to P( V^\vee)=(\mathbb{P}^3)^\vee: x\mapsto V_1(x)$,
where the pair of 1-dimensional spaces
$(U_1(x),V_1(x)), \ \ \ U_1(x)\subset H^\vee_{2m+1}$
and
$V_1(x)\subset V^\vee$,
is determined uniquely by the point $x$ via the condition $x\in P(U_1(x)\otimes V_1(x))$,
since $\rk(x)=1$ (see Lemma \ref{Determ}).
Now for a given subspace $V_1\in(\mathbb{P}^3)^\vee$ set
$$
\Sigma_1(V_1):=\{V_m\in G\ |\ V_m\supset U_1(x)\ {\rm for\ some\ point}\ x\in f_1^{-1}(V_1)\}.
$$
Then similar to (\ref{Sigma3=}) we have
\begin{equation}\label{Di}
\Sigma_1=\underset{V_1\in(\mathbb{P}^3)^\vee}\cup\Sigma_1(V_1).
\end{equation}
Hence,
\begin{equation}\label{dim Sigma1}
\dim\Sigma_1\le\dim\Sigma_1(V_1)+3.
\end{equation}
We are going to obtain an estimate for the dimension of $\Sigma_1(V_1)$ for an arbitrary
1-dimensional subspace $V_1$ of $ V^\vee$. This subspace $V_1$ defines a commutative diagram
\begin{equation}\label{Di7}
\xymatrix{
& & 0\ar[d] & 0\ar[d] & \\
& & \Omega_{\mathbb{P}^3}\ar@{=}[r]\ar[d] & \Omega_{\mathbb{P}^3}\ar[d] & \\
0\ar[r] & V_1\otimes\mathcal{O}_{\mathbb{P}^3}(-1)\ar[r]\ar@{=}[d] &
 V^\vee\otimes\mathcal{O}_{\mathbb{P}^3}(-1)\ar[r]\ar[d] &
V_3\otimes\mathcal{O}_{\mathbb{P}^3}(-1)\ar[r]\ar[d] & 0 \\
0\ar[r] & \mathcal{O}_{\mathbb{P}^3}(-1)\ar[r] & \mathcal{O}_{\mathbb{P}^3}\ar[r]\ar[d] &
\mathcal{O}_{\mathbb{P}^2}\ar[r]\ar[d] & 0 \\
& & 0 & \ 0. &}
\end{equation}
Note that to the point $V_1\in(\mathbb{P}^3)^\vee$ there corresponds a projective
plane $P(V_1)$ in $\mathbb{P}^3$ and
set $B(E):=\{V_1\in(\mathbb{P}^3)^\vee\ |\ h^0(E|_{P(V_1)})\ne0\}$.
It is known that, for $m\ge1,$
$\dim B(E)\le2$
(see \cite{B1}). Moreover, in view of (\ref{tripleC}),
\begin{equation}\label{h0(E|PV1)}
h^0(E|_{P(V_1)})=1,\ \ \ V_1\in B(E).
\end{equation}
Passing to cohomology in diagram (\ref{Di7}) twisted by $E$ and using the equality
$h^0(E)=0$ for $[E]\in I_{2m+1}$ we obtain the diagram

\begin{equation}\label{Di8}
\xymatrix{
& & 0\ar[d] & H^0(E|_{P(V_1)})\ar@{>->}[d] & \\
& & W_{4m+4}^\vee\ar@{=}[r]\ar[d] & W_{4m+4}^\vee\ar[d] & \\
0\ar[r] & H^\vee_{2m+1}\otimes V_1\ar[r]^\lambda\ar@{=}[d] &
H^\vee_{2m+1}\otimes V^\vee\ar[r]\ar[d]^{mult} & H^\vee_{2m+1}\otimes V_3\ar[d]\ar[r] & 0 \\
H^0(E|_{P(V_1)})\ \ar@{>->}[r] & H^\vee_{2m+1}\ar[r] &
H^\vee_{4m}\ar[r]\ar[d] & H^1(E|_{P(V_1)})\ar[d]\ar[r] & 0 \\
& & 0 & \ 0. & }
\end{equation}
Let $V_1\in B(E)$.
Setting $W_1(V_1):=\ker(mult\circ\lambda)=H^0(E|_{P(V_1)})$,
where by (\ref{h0(E|PV1)}) $\dim W_1(V_1)=1$, we obtain from (\ref{Di8}) a commutative diagram
$$
\xymatrix{
& 0\ar[d] & 0\ar[d] & 0\ar[d] & \\
0\ar[r] & W_1(V_1)\ar[r]\ar[d] & W_{4m+4}^\vee\ar[r]\ar[d]
& W_{4m+4}^\vee/W_1(V_1) \ar[d]\ar[r] & 0 \\
0\ar[r] & H^\vee_{2m+1}\otimes V_1\ar[r]^\lambda\ar[d]^\epsilon &
H^\vee_{2m+1}\otimes V^\vee\ar[r]\ar[d]^{mult} & H^\vee_{2m+1}\otimes V_3\ar[d]\ar[r] & 0 \\
0\ar[r] & H^\vee_{2m+1}/W_1(V_1)\ar[r]\ar[d] & H^\vee_{4m}\ar[r]\ar[d] &
H^1(E|_{\mathbb{P}^2(V_1)})\ar[d]\ar[r] & 0 \\
& 0 & 0 & \ 0, & }
$$
hence a relation
\begin{equation}\label{Di9}
W_1(V_1)=H^\vee_{2m+1}\otimes V_1\cap W_{4m+4}^\vee,
\end{equation}
where the intersection is taken in $H^\vee_{2m+1}\otimes V^\vee$.
Set
$$
Z_1(V_1):=\emptyset \ \ {\rm if} \ \ V_1\ne B(E),\ \ \ {\rm respectively,}\ \
Z_1(V_1):=P(W_1(V_1))=\{pt\} \ \ {\rm if} \ \ V_1\in B(E).
$$

The relation (\ref{Di9}) and Lemma \ref{Determ} imply the bijection
\begin{equation}\label{bijectn1}
Z_1(V_1)\overset{\simeq}\to f_1^{-1}(V_1),\ \ \ V_1\in(\mathbb{P}^3)^\vee,
\end{equation}
Consider the graph of incidence
$\Gamma_1(V_1):=\{(x,U)\in Z_1(V_1)\times P(H^\vee_{2m+1})\ |\ U=U_1(x)\}$
with projections
$Z_1(V_1)\overset{p_1}\leftarrow\Gamma_1(V_1)\overset{q_1}\to P(H^\vee_{2m+1})$.
By construction, $p_1(\Gamma_1(V_1))=Z_1(V_1)$ and the projection
$p_4:\Gamma_1(V_1)\to Z_1(V_1)$ is a bijection. Hence
\begin{equation}\label{dim q1(...)<}
\dim q_1(\Gamma_1(V_1))\le\dim\Gamma_1(V_1)=\dim Z_1(V_1)\le0.
\end{equation}
Consider the graph of incidence
$$
\Pi_1(V_1)=\{(U,V_m)\in q_1(\Gamma_1(V_1))\times\Sigma_1(V_1)\ |\ U\subset V_m\}
$$
with projections
$q_1(\Gamma_1(V_1))\overset{pr_1}\leftarrow\Pi_1(V_1)\overset{pr_2}\to\Sigma_1(V_1)$
and a fibre
$$
pr_1^{-1}(U)\simeq G(m-1,H^\vee_{2m+1}/U)
$$
over an arbitrary point $U\in q_1(\Gamma_1(V_1))$.
The projection $\Pi_1(V_1)\overset{pr_2}\to\Sigma_1(V_1)$ is surjective in view of
(\ref{bijectn1}). Hence using (\ref{dim q1(...)<}) we have
$$
\dim\Sigma_1(V_1)\le\dim\Pi_1(V_1)=\dim q_1(\Gamma_1(V_1))+\dim G(m-1,H^\vee_{2m+1}/U)\le0+
(m-1)(m+1)=
$$
$=m^2-1$. This together with (\ref{dim Sigma1}) and the assumption $m\ge3$ yields
$\dim\Sigma_1\le m^2+2=\dim G+2-m<\dim G$, i.e. (\ref{dim Sigma r}) holds for $r=1$.
Theorem is proved.
\end{proof}

\end{sub}

\begin{sub} {\bf Proof of Proposition \ref{nondeg for general}.}
\rm

Before giving the proof of this Proposition, we need some preliminary arguments.
For any point $B\in\mathbf{S}_{m+1}$ let $\hat{B}:S^2H_{m+1}\to\wedge^2V^\vee$ denote the induced homomorphism.
We have a morphism of affine varieties
\begin{equation}\label{bold b}
\mathbf{b}:\ H_{m+1}\times\mathbf{S}_{m+1}\to\wedge^2V^\vee:\ (h,B)\mapsto\hat{B}(h\otimes h).
\end{equation}
Fix a basis $e_1,e_2,e_3,e_4$ in $V$.
Then the point $B\in\mathbf{S}_{m+1}$ considered as a homomorphism $B:H_{m+1}\otimes V\to H_{m+1}^\vee\otimes V^\vee$
can be represented by a skew-symmetric block matrix
\begin{equation}\label{matrix B}
B=\left(
\begin{array}{cccc}
0 & A_{12} & A_{13} & A_{14}\\
-A_{12} & 0 & A_{23} & A_{24}\\
-A_{13} & -A_{23} & 0 & A_{34} \\
-A_{14} & -A_{24} & -A_{34} & 0
\end{array}
\right)\\
\end{equation}
where $A_{ij}\in S^2H_{m+1}^\vee,\ 1\le i<j\le4$. Here we consider $A_{ij}$ as the quadratic forms
\begin{equation}\label{quadr forms}
H_{m+1}\to\mathbf{k}:x\mapsto A_{ij}(x),\ 1\le i<j\le4,
\end{equation}
on $H_{m+1}$. Respectively, in the projective space
$P(H_{m+1})\simeq\mathbb{P}^m$ there are defined quadrics
\begin{equation}\label{Q ij}
Q_{ij}(B):=\{<x>\in P(H_{m+1})\ |\ A_{ij}(x)=0\},\ \ \ \ 1\le i<j\le4.
\end{equation}

Let $K\subset\wedge^2V^\vee$ be the cone of decomposable vectors, $K=\{w\in\wedge^2V^\vee|\ \rk(w:V\to V^\vee)\le2\}$,
and, for $m\ge1$, set
\begin{equation}\label{M m+1}
M_{m+1}:=\{B\in\mathbf{S}_{m+1}|\ \mathbf{b}(H_{m+1}\times\{B\})\subset K\}.
\end{equation}
By construction, $M_{m+1}$ is a closed subset of $\mathbf{S}_{m+1}$, and we consider it as a reduced subscheme of
$\mathbf{S}_{m+1}$.

Consider first the cases $m=0,1$ and 2. An explicit computation shows that

(i) $M_1,M_2$ and $M_3$ are irreducible and, moreover,
\begin{equation}\label{M1,M2,M3}
M_1=K,\ \ \ M_{m+1}\subset\mathbf{S}_{m+1}\smallsetminus(\mathbf{S}_{m+1})^0,\ \ \
\codim_{\mathbf{S}_{m+1}}M_{m+1}=2, \ \ \ m=1,2;
\end{equation}

(ii) $M_3^*:=\{B\in M_3|\ Y_3(B):=Q_{13}(B)\cap Q_{23}(B)$ is a 4-ple of distinct points in the projective plane
$P(H_3)\}$ is a dense open subset of $M_3$.

Now proceed to the case $m\ge3$. In this case, set
\begin{equation}\label{S*m+1}
\mathbf{S}^*_{m+1}:=\{B\in\mathbf{S}_{m+1}|\ Y_{m+1}(B):=Q_{13}(B)\cap Q_{23}(B)\ {\rm is\ an\ integral\ codimension\ 2}
\end{equation}
$$
{\rm subscheme\ of\ the\ projective\ space}\ P(H_{m+1})\}.
$$
Since $m\ge3$, $\mathbf{S}^*_{m+1}$ is a dense open subset of $\mathbf{S}_{m+1}$.

\begin{lemma}\label{1 B in S*}
For $m\ge3$ let $B\in\mathbf{S}^*_{m+1}\cap M_{m+1}$. Then $B\not\in\mathbf{S}^0_{m+1}$.
\end{lemma}
\begin{proof}
We represent a given point $B\in\mathbf{S}^*_{m+1}\cap M_{m+1}$ by matrix (\ref{matrix B}). Then, under the notation
(\ref{quadr forms}), for
$x\in H_{m+1}$, we obtain a skew-symmetric $(4\times4)$-matrix with entries in $\mathbf{k}$
\begin{equation}\label{B(x)}
B(x)=\left(
\begin{array}{cccc}
0 & A_{12}(x) & A_{13}(x) & A_{14}(x)\\
-A_{12}(x) & 0 & A_{23}(x) & A_{24}(x)\\
-A_{13}(x) & -A_{23}(x) & 0 & A_{34}(x) \\
-A_{14}(x) & -A_{24}(x) & -A_{34}(x) & 0
\end{array}
\right).\\
\end{equation}
The condition $B\in M_{m+1}$ by definition means that the matrix $B(x)$
is degenerate, i.e. its Pfaffian vanishes identically as a polynomial function on $H_{m+1}$:
\begin{equation}\label{Pf vanishes}
A_{12}(x)A_{34}(x)-A_{13}(x)A_{24}(x)+A_{14}(x)A_{23}(x)\equiv0,\ \ \ x\in H_{m+1}.
\end{equation}
Since $B\in\mathbf{S}^*_{m+1}$, from (\ref{Q ij}) and (\ref{S*m+1}) it follows that the quadrics $Q_{13}(B)$ and
$Q_{23}(B)$ are integral and their intersection $Y:=Y_{m+1}(B)$
is integral of codimension 2 in $P(H_{m+1})$. In this case (\ref{Pf vanishes}) implies that either
$Q_{12}(B)\supset Y$, or $Q_{34}(B)\supset Y$. Let, say, $Q_{34}(B)\supset Y_{m+1}(B)$. This means
that $A_{34}(x)\in H^0(\mathcal{I}_{Y,\mathbb{P}^m}(2))$. Now, passing to sections of the exact triple\\
$0\to\mathcal{O}_{\mathbb{P}^m}(-2)\to2\mathcal{O}_{\mathbb{P}^m}\overset{A_{13}(x),A_{23}(x)}\longrightarrow
\mathcal{I}_{Y,\mathbb{P}^m}(2)\to0$,
we obtain that $A_{34}(x)=\alpha A_{13}(x)+\beta A_{23}(x)$ for some $\alpha,\beta\in\mathbf{k}$.
Substituting this relation into (\ref{Pf vanishes}) we obtain a relation
$
A_{13}(x)(\alpha A_{12}(x)-A_{24}(x))+A_{23}(x)(\beta A_{12}(x)+A_{14}(x))\equiv0.
$
Since $Q_{13}$ and $Q_{23}$ are integral, the last relation implies that either

(i) $A_{23}=\lambda A_{13},\ A_{24}-\alpha A_{12}=\lambda(\beta A_{12}+A_{14})$ for some $\lambda\in\mathbf{k}$, or

(ii) $\beta A_{12}+A_{14}=\mu A_{13},\ A_{24}-\alpha A_{12}=\mu A_{23}$ for some $\mu\in\mathbf{k}$.

Substituting the relations (i)
into (\ref{matrix B}) and denoting $\gamma=\alpha+\lambda\beta$, we obtain
\begin{equation}\label{new matrix B}
B=\left(
\begin{array}{cccc}
0 & A_{12} & A_{13} & A_{14}\\
-A_{12} & 0 & \lambda A_{13} & \gamma A_{12}+\lambda A_{14}\\
-A_{13} & -\lambda A_{23} & 0 & \gamma A_{13} \\
-A_{14} & -\gamma A_{12}-\lambda A_{14} & -\gamma A_{13} & 0
\end{array}
\right).\\
\end{equation}
Adding the multiplied by $\lambda$ first block column of this matrix to its
fourth block column,
and then performing a similar operation with block rows, we obtain the matrix
\begin{equation}\label{matrix B'}
B'=\left(
\begin{array}{cccc}
0 & A_{12} & A_{13} & A_{14}\\
-A_{12} & 0 & \lambda A_{13} & \lambda A_{14}\\
-A_{13} & -\lambda A_{13} & 0 & 0 \\
-A_{14} & -\lambda A_{14} & 0 & 0
\end{array}
\right)\\
\end{equation}
which is degenerate. Hence $B$ is also degenerate. A similar computation with relations (ii) also gives the
degenerateness of $B$. Lemma is proved.
\end{proof}

From Lemma \ref{1 B in S*} it follows that, for any irreducible component $M'_{m+1}$ of $M_{m+1}$,
\begin{equation}\label{codim le2}
1\le\codim_{\mathbf{S}_{m+1}}M'_{m+1}\le2,\ \ \ m\ge3.
\end{equation}
Indeed, from this Lemma we obtain that
$\mathbf{S}^*_{m+1}\cap\mathbf{S}^0_{m+1}\cap  M_{m+1}=\emptyset$.
Since
$\mathbf{S}^*_{m+1}\cap\mathbf{S}^0_{m+1}$
is a dense open subset of $\mathbf{S}_{m+1}$, it follows that $M_{m+1}\ne\mathbf{S}_{m+1}$, i.e.
$1\le\codim_{\mathbf{S}_{m+1}}M_{m+1}$.
On the other hand, $K$ is a nonempty divisor in $\wedge^2V^\vee$, hence $\mathbf{b}_{m+1}^{-1}(K)$ is a nonempty
divisor of $H_{m+1}\times\mathbf{S}_{m+1}$. Since $M_{m+1}$ is nonempty (in fact, $\{0\}\in M_{m+1}$), counting of
dimensions of the fibres of the natural projection $\mathbf{b}_{m+1}^{-1}(K)\to\mathbf{S}_{m+1}$ shows that,
for any irreducible component $M'_{m+1}$ of $M_{m+1}$, $\codim_{\mathbf{S}_{m+1}}M'_{m+1}\le2$, and (\ref{codim le2})
follows.

\begin{lemma}\label{2 B in S*}
For $m\ge3$ let $M'_{m+1}$ be any irreducible component of $M_{m+1}$. Then
$\mathbf{S}^*_{m+1}\cap M'_{m+1}\ne\emptyset$. Hence $\mathbf{S}^*_{m+1}\cap M'_{m+1}$ is a dense open subset of
$M'_{m+1}$.
\end{lemma}
\begin{proof}

1) Consider first the case $m=3$. Choose coordinates $x_1,...,x_4$ in $H_4$ and let
$H_1$ and $H_3$ be the subspaces of $H_4$ given by
the equations $x_1=x_2=x_3=0$ and $x_4=0$, respectively. The direct sum decomposition $H_4=H_1\oplus H_3$ induces
the iclusion of a direct summand $\mathbf{S}_1\oplus\mathbf{S}_3\hookrightarrow\mathbf{S}_4$.
Considering this inclusion as an embedding of an affine subspace
$\mathbf{S}_1\times\mathbf{S}_3\hookrightarrow\mathbf{S}_4$,
we obtain from (\ref{M1,M2,M3}) and from the definition (\ref{M m+1}) that
\begin{equation}\label{M3 in M4}
M_3=(\{0\}\times\mathbf{S}_3)\cap M_4,\ \ \ K=(\mathbf{S}_1\times\{0\})\cap M_4.
\end{equation}
This together with (\ref{codim le2}) and the irreducibility of $M_3$ (see property (i) above) implies that,
for an arbitrary irreducible component $M'_4$ of $M_4$,
\begin{equation}\label{M3 in M'4}
M_3=(\{0\}\times\mathbf{S}_3)\cap M'_4,\ \ \ K=(\mathbf{S}_1\times\{0\})\cap M'_4.
\end{equation}
Note that (\ref{M1,M2,M3}) and (\ref{codim le2}) imply that
\begin{equation}\label{codim S4}
\codim_{\mathbf{S}_4}M'_4=2.
\end{equation}

Take any point $B'\in M^*_3$, and let $A_{i3}(B')(x_1,x_2,x_3)$ be the quadratic forms on $H_3$ corresponding to the
entries $A_{i3}(B'),\ i=1,2,$ of the matrix $B'$. Then the set $Y_3(B')$ is given in the projective space $P(H_3)$ by
the equations
$\{A_{i3}(B')(x_1,x_2,x_3)=0,\ i=1,2\}$. Now take an arbitrary point $B''\in\mathbf{S}_1\simeq\wedge^2V^\vee$
and, according to (\ref{matrix B}), consider $B''$ as a skew-symmetric matrix $(a_{ij}(B''))$. Then the point
$B:=(B',B'')\in\mathbf{S}_1\times\mathbf{S}_3$ determines the scheme $Y_4(B)$ (see (\ref{S*m+1})) which is given
in the projective space $P(H_4)$ by the equations
\begin{equation}\label{Y4(B)}
A_{i3}(B')(x_1,x_2,x_3)-a_{i3}(B'')x_4^2=0,\ \ \ i=1,2.
\end{equation}
Consider the sets
$U'=\{(B',B'')\in\mathbf{S}_1\times\mathbf{S}_3|Y_3(B')=Q_{13}(B')\cap Q_{23}(B')$
is a 4-ple of distinct points in the plane $P(H_3)\}$ and
$U''=\{(B',B'')\in\mathbf{S}_1\times\mathbf{S}_3|a_{i3}(B'')\ne0,\ i=1,2\}$.
These sets dense open subsets of $\mathbf{S}_1\times\mathbf{S}_3$, and from
(\ref{M3 in M'4}) and the property (ii) above it follows that $M_4'':=M'_4\cap U'\cap U''$ is a dense open subset of
$M'_4$. Now for any point $B=(B',B'')\in M_4''$ the equations (\ref{Y4(B)}) can be rewritten as follows
\begin{equation}\label{new Y4(B)}
A(x_1,x_2,x_3):=A_{13}(B')(x_1,x_2,x_3)a_{23}(B'')-A_{23}(B')(x_1,x_2,x_3)a_{13}(B'')=0,
\end{equation}
$$
A_{13}(B')(x_1,x_2,x_3)-a_{13}(B'')x_4^2=0.
$$
Consider the conic $C(B)=\{A(x_1,x_2,x_3)=0\}$ in $\mathbb{P}^2$. Then $M_4'''=\{B\in M_4''\ |\ C(B)$ is
integral$\}$ is a dense open subset of $M_4''$.
By construction, the set $\{A_{13}(B')(x_1,x_2,x_3)=0\}\cap C(B)$ coincides with the set $Y_3(B')$ which by
definition is a 4-ple of distinct points in $\mathbb{P}^2$. Therefore the equations (\ref{new Y4(B)}) defining $Y_4(B)$
show that $Y_4(B)$ is a double cover of $C(B)$ ramified in $Y_3(B')$, hence it is an integral elliptic quartic curve
in $\mathbb{P}^3$. In other words, $B\in\mathbf{S}^*_4\cap M'_4$. This means that
$M_4'''\subset\mathbf{S}^*_4\cap M'_4$, so that $\mathbf{S}^*_4\cap M'_4$ is dense open in $M'_4$.

2) The argument in the case $m\ge4$ is similar to the above. Choose coordinates $x_1,...,x_{m+1}$ in $H_{m+1}$ and let
$H_{m-3}$ and $H_4$ be the subspaces of $H_{m+1}$ given by
the equations $x_1=...=x_4=0$ and $x_5=...=x_{m+1}=0$, respectively. The direct sum decomposition
$H_{m+1}=H_{m-3}\oplus H_4$
induces the iclusion of a direct summand $\mathbf{S}_4\hookrightarrow\mathbf{S}_{m+1}$.
Considering this inclusion as an embedding of an affine subspace
$\mathbf{S}_4\hookrightarrow\mathbf{S}_{m+1}$,
we obtain from the definition (\ref{M m+1}) that, similar to (\ref{M3 in M'4}),
\begin{equation}\label{M4 in Mm+1}
M_4=\mathbf{S}_4\cap M_{m+1}.
\end{equation}
Now let $M'_{m+1}$ be any irreducible component of $M_{m+1}$. From (\ref{codim le2}), (\ref{codim S4}) and
(\ref{M4 in Mm+1}) it follows that, for any irreducible component $M'_4$ of $\mathbf{S}_4\cap M'_{m+1}$,
the set $M'^*_4=\mathbf{S}^*_4\cap M'_4$ is a dense open subset of $M'_4$. By definition, an arbitrary point
$B\in M'^*_4$ is such that $Y_4(B)$ is an integral quartic curve in $\mathbb{P}^3$. From the construction of
the embedding $\mathbf{S}_4\hookrightarrow\mathbf{S}_{m+1}$ it follows now that, for this point $B$ considered
as a point in $M'_{m+1}$, the scheme $Y_{m+1}(B)$ is a cone in $P(H_{m+1})$ over $Y_4(B)$. Hence $Y_{m+1}(B)$
is an integral codimension 2 subscheme of $P(H_{m+1})$, i.e. $B\in\mathbf{S}_{m+1}^*$. This means that
$\mathbf{S}^*_{m+1}\cap M'_{m+1}$ is a dense open subset of $M'_{m+1}$.
\end{proof}

\begin{corollary}\label{M m+1 degenerate}
For any $m\ge0$, $M_{m+1}\subset\mathbf{S}_{m+1}\smallsetminus\mathbf{S}_{m+1}^0$.
\end{corollary}
\begin{proof}
For $m\le2$ this statement follows from (\ref{M1,M2,M3}). Let $m\ge3$ and
let $M'_{m+1}$ be any irreducible component of $M_{m+1}$. By Lemma \ref{1 B in S*}
$\mathbf{S}^*_{m+1}\cap M'_{m+1}\subset\mathbf{S}_{m+1}\smallsetminus\mathbf{S}_{m+1}^0$.
Since $\mathbf{S}_{m+1}\smallsetminus\mathbf{S}_{m+1}^0$ is a closed subset of $\mathbf{S}_{m+1}$
and by Lemma \ref{2 B in S*} the set $\mathbf{S}^*_{m+1}\cap M'_{m+1}$ is a dense open subset of an irreducible
set $M'_{m+1}$, it follows that $M'_{m+1}\subset\mathbf{S}_{m+1}\smallsetminus\mathbf{S}_{m+1}^0$.
\end{proof}

We are now ready to prove Proposition \ref{nondeg for general}.

{\it Proof of Proposition \ref{nondeg for general}}.
Let  $D\in(\mathbf{S}^\vee_{m+1})^0$, i.e. $D$ is a
nondegenerate homomorphism $D:H_{m+1}^\vee\otimes V^\vee\to
H_{m+1}\otimes V$. Assume that, for any monomorphism
$j:H_m^\vee\hookrightarrow H_{m+1}^\vee$, the composition
$j_D:=j^\vee\circ D\circ j:H_m^\vee\otimes V^\vee\to H_m\otimes V$
is degenerate. We will show that this leads to a contradiction.
For this, represent $j$ dually as a monomorphism
$j_\mathbf{k}:\mathbf{k}\hookrightarrow H_{m+1}$. Consider the
nondegenerate homomorphism $B:=D^{-1}:H_{m+1}\otimes V\to
H_{m+1}^\vee\otimes V^\vee$ and the induced skew-symmetric
homomorphism $j_B:=j_\mathbf{k}^\vee\circ B\circ
j_\mathbf{k}:V\simeq\mathbf{k}\otimes V\to \mathbf{k}^\vee\otimes
V^\vee\simeq V^\vee$. Then the degenerateness of $j_D$ is equivalent
to the degenerateness of $j_B$. As above, the homomorphism $B$ can be
represented by a skew-symmetric matrix (\ref{matrix B}).
In this notation, the degenerateness of the
homomorphism $j_B$ for any $j_\mathbf{k}:\mathbf{k}\hookrightarrow H_{m+1}$ just means that, for any vector
$x\in H_{m+1}$, the skew-symmetric $(4\times4)$-matrix $B(x)$ in (\ref{B(x)})
is degenerate, i.e., by definition, $B\in M_{m+1}$. Then by Corollary \ref{M m+1 degenerate} $B$ is degenerate.
This contradiction proves Proposition.

\end{sub}

\end{document}